\documentclass[12pt,a4]{amsart}
\usepackage[a4paper, left=28mm, right=28mm, top=28mm, bottom=34mm]{geometry}
\usepackage[all]{xy}
\usepackage{amsmath}
\usepackage{amssymb}
\usepackage{amsthm}
\usepackage{mathrsfs}
\theoremstyle{definition}
\newtheorem{thm}{Theorem}[section]
\newtheorem{lem}[thm]{Lemma}
\newtheorem{cor}[thm]{Corollary}
\newtheorem{prop}[thm]{Proposition}

\theoremstyle{definition}

\newtheorem{rem}[thm]{Remark}

\newtheorem{defn}[thm]{Definition}

\def\A{{\mathbb A}}
\def\F{{\mathbb F}}
\def\G{{\mathbb G}}
\def\Q{{\mathbb Q}}
\def\R{{\mathbb R}}
\def\Z{{\mathbb Z}}
\def\C{{\mathbb C}}
\def\O{{\mathscr O}}

\def\N{{\mathbb N}}

\def\Br{\mathop{\mathrm{Br}}\nolimits}

\def\End{\mathop{\mathrm{End}}\nolimits}

\def\Gr{\text{\rm Gr}}

\def\Aut{\mathop{\mathrm{Aut}}\nolimits}

\def\Cl{\mathop{\mathrm{Cl}}\nolimits}

\def\Frob{\mathop{\mathrm{Frob}}\nolimits}

\def\Fil{\mathop{\mathrm{Fil}}\nolimits}

\def\Gal{\mathop{\mathrm{Gal}}\nolimits}

\def\Hom{\mathop{\mathrm{Hom}}\nolimits}

\def\Im{\mathop{\mathrm{Im}}\nolimits}

\def\Ker{\mathop{\mathrm{Ker}}\nolimits}
\def\id{\mathop{\mathrm{id}}\nolimits}

\def\GL{\mathop{\mathrm{GL}}\nolimits}

\def\GSp{\mathop{\mathrm{GSp}}\nolimits}

\def\GSpin{\mathop{\mathrm{GSpin}}\nolimits}
\def\SO{\mathop{\mathrm{SO}}\nolimits}

\def\Res{\mathop{\mathrm{Res}}\nolimits}

\def\Pic{\mathop{\mathrm{Pic}}\nolimits}

\def\Spec{\mathop{\rm Spec}}
\def\Spf{\mathop{\rm Spf}}

\def\cris{\text{\rm cris}}

\def\rank{\mathop{\text{\rm rank}}\nolimits}

\def\det{\mathop{\mathrm{det}}\nolimits}

\def\KS{\mathop{\mathrm{KS}}}

\def\Sh{\mathop{\mathrm{Sh}}\nolimits}

\def\A{\mathbb{A}}

\def\Fil{\mathop{\mathrm{Fil}}\nolimits}
\def\dR{\mathop{\mathrm{dR}}}
\def\K{\mathop{\mathrm{K}}\nolimits}

\newcommand{\plim}[1][]{\mathop{\varprojlim}\limits_{#1}}

\def\Gr{\mathop{\mathrm{Gr}}\nolimits}
\def\M{\mathop{\mathfrak{M}}\nolimits}
\def\MdR{\mathop{\mathfrak{M}_{\dR}}\nolimits}
\def\Mcris{\mathop{\mathfrak{M}_{\cris}}\nolimits}
\newcommand{\et}{\mathrm{\acute{e}t}}
\newcommand{\proet}{\mathrm{pro\acute{e}t}}
\newcommand{\ad}{\mathrm{ad}}
\newcommand{\ch}{\mathrm{ch}}

\begin{document}

\title[CM liftings of $K3$ surfaces and the Tate conjecture]{CM liftings of $K3$ surfaces over finite fields and their applications to the Tate conjecture}

\author{Kazuhiro Ito}
\address{Department of Mathematics, Faculty of Science, Kyoto University, Kyoto 606-8502, Japan}
\email{kito@math.kyoto-u.ac.jp}

\author{Tetsushi Ito}
\address{Department of Mathematics, Faculty of Science, Kyoto University, Kyoto 606-8502, Japan}
\email{tetsushi@math.kyoto-u.ac.jp}

\author{Teruhisa Koshikawa}
\address{Research Institute for Mathematical Sciences, Kyoto University, Kyoto 606-8502, Japan}
\email{teruhisa@kurims.kyoto-u.ac.jp}




\maketitle

\begin{abstract}
We give applications of integral canonical models of
orthogonal Shimura varieties and the Kuga-Satake morphism to
the arithmetic of $K3$ surfaces over finite fields.
We prove every $K3$ surface of finite height over a finite field
admits a characteristic $0$ lifting whose generic fiber is
a $K3$ surface with complex multiplication.
Combined with the results of Mukai and Buskin,
we prove the Tate conjecture for the square of
a $K3$ surface over a finite field.
To obtain these results,
we construct an analogue of Kisin's algebraic group for
a $K3$ surface of finite height,
and construct characteristic $0$ liftings of the $K3$ surface
preserving the action of tori in the algebraic group.
We obtain these results for $K3$ surfaces over finite fields of any characteristics,
including those of characteristic $2$ or $3$.
\end{abstract}

\section{Introduction} \label{Section:Introduction}

The integral canonical models of orthogonal Shimura varieties
and the Kuga-Satake morphism have applications to the arithmetic of $K3$ surfaces over finite fields.
For example, Madapusi Pera used it to prove the Tate conjecture for divisors on $K3$ surfaces over finitely generated fields \cite{MadapusiTateConj}. (See \cite{KimMadapusiIntModel} for the case of characteristic $2$.)

The aim of this paper is to give further applications.
More specifically, we shall prove the following results:
\begin{enumerate}
\item (see Theorem \ref{Theorem:CMlifting})
Every $K3$ surface $X$ of finite height over a finite field $\F_q$
with $q$ elements admits a CM lifting after replacing $\F_q$ by its finite extension (i.e.\ it admits a characteristic $0$ lifting whose generic fiber has complex multiplication).
\item (see Theorem \ref{Theorem:TateConjecture})
The Tate conjecture holds for algebraic cycles of codimension $2$ on the square $X \times X$ of any $K3$ surface $X$ (of any height) over $\F_q$.
\end{enumerate}

These results are consequences of our results on
characteristic $0$ liftings of $K3$ surfaces;
see Theorem \ref{Theorem:LiftingTorusAction}.

Our strategy of the proof is as follows.
Let $(X,\mathscr{L})$ be a quasi-polarized $K3$ surface of finite height over $\overline{\F}_q$.
Here $\mathscr{L}$ is a (primitive) line bundle on $X$ which is big and nef.
We shall attach an algebraic group $I$ over $\Q$ to each polarized $K3$ surface $(X,\mathscr{L})$ of finite height over $\overline{\F}_q$,
which is an analogue of the algebraic group attached by Kisin to each mod $p$ point on
the integral canonical model of a Shimura variety of Hodge type \cite{KisinModp}.
Then, for each maximal torus $T \subset I$ over $\Q$,
we shall construct a characteristic $0$ lifting of the quasi-polarized $K3$ surface $(X,\mathscr{L})$ such that the action of each element of $T(\Q)$
on the singular cohomology of the generic fiber preserves the $\Q$-Hodge structure.
We use integral canonical models of Shimura varieties
to control the rationality of the action of $T(\Q)$.
From these results, the result (1) follows by comparing the rank of the algebraic group $I$ and the general spin group attached to the orthogonal Shimura variety.
Combined with the results of Mukai and Buskin on the Hodge conjecture for products of $K3$ surfaces,
we shall prove the action of every element of $T(\Q)$ is induced by an algebraic cycle of codimension $2$ on $X \times X$.
Applying this result for several maximal tori $T \subset I$, the result (2) follows.

Note that we do not impose any conditions on the characteristic of the base field.
Thus, the main results of this paper are valid over finite fields of any characteristics,
including those of characteristic $2$ or $3$.
To overcome certain technical difficulties, we essentially use the integral comparison theorems of Bhatt-Morrow-Scholze \cite{BMS}, at least in small characteristics.
(Note that, when the characteristic is greater than or equal to $5$,
we can avoid most of the technical difficulties.
Instead, we can use the results of Nygaard-Ogus to obtain the main results of this paper.
See Section \ref{Subsection:Remarks on the characteristic and the Kuga-Satake morphism}.)

In the course of writing this paper,
we found an error in the proof of the \'etaleness of the Kuga-Satake morphism
in characteristic $2$,
which was also used in the proof of the Tate conjecture for $K3$ surfaces in characteristic $2$ \cite{KimMadapusiIntModel}.
We correct it using our results on $F$-crystals on orthogonal Shimura varieties,
which depend on the integral comparison theorem of Bhatt-Morrow-Scholze \cite{BMS}; see Remark \ref{Remark: Kim-Madapusi Pera gap} for details. (See also Remark \ref{Remark:Bloch-Kato and BMS}.)

In the rest of Introduction, we shall first give precise statements on our results on
CM liftings and the Tate conjecture;
see Theorem \ref{Theorem:CMlifting} and Theorem \ref{Theorem:TateConjecture}.
Then we explain our results on characteristic $0$ liftings (see Theorem \ref{Theorem:LiftingTorusAction}),
and how to obtain (1) and (2) from them.

\subsection{CM liftings of $K3$ surfaces of finite height over finite fields}
\label{Subsection:Introduction:CMLifting}

First we state our results on CM liftings.

Recall that a projective smooth surface $X$ over a field is called a \textit{$K3$ surface}
if its canonical bundle is trivial and it satisfies $H^1(X,\mathcal{O}_X)=0$.
More generally, an algebraic space $\mathscr{X}$ over a scheme $S$
is a $K3$ surface over $S$ if $\mathscr{X} \to S$ is proper, smooth, and
every geometric fiber is a $K3$ surface.

We say a projective $K3$ surface $Y$ over $\C$ has
\textit{complex multiplication} (\textit{CM})
if the Mumford-Tate group associated with the singular cohomology
$H^2_B(Y, \Q)$ is commutative;
see Section \ref{Subsection:K3 surfaces with complex multiplication}.
We say a $K3$ surface $Y$ over a number field $F$ has CM
if $Y_{\C}$ has CM for every embedding $F \hookrightarrow \C$.

We fix a prime number $p$, and a power $q$ of $p$.
Let $X$ be a $K3$ surface over $\F_q$.
We say $X$ admits a \textit{CM lifting} if there exist a number field $F$,
a finite place $v$ of $F$ with residue field $\F_q$,
and a $K3$ surface $\mathscr{X}$ over
the localization $\O_{F, (v)}$ of the ring of integers $\O_{F}$ of $F$ at $v$ such that
the special fiber $\mathscr{X}_{\F_q}$ is isomorphic to $X$, and
the generic fiber $\mathscr{X}_{F}$ is a $K3$ surface with CM.
The height $h$ of the formal Brauer group of $X$
is called the \textit{height} of $X$;
it satisfies $1 \leq h \leq 10$ or $h = \infty$.
When $1 \leq h \leq 10$ (resp.\ $h = \infty$),
we say $X$ is \textit{of finite height} (resp.\ \textit{supersingular}).

Here is the first main theorem of this paper.

\begin{thm}[see Corollary \ref{Corollary:CMLifting}]
\label{Theorem:CMlifting}
Let $X$ be a $K3$ surface over $\F_q$.
If $X$ is of finite height,
then there is a positive integer $m \geq 1$ such that
$X_{\F_{q^m}} := X \times_{\Spec \F_q} \Spec \F_{q^m}$
admits a CM lifting.
\end{thm}

\begin{rem}
After we completed the first draft of this paper,
the authors learned Yang also proved the above theorem under the additional conditions that $p\geq 5$ and $X$ admits a quasi-polarization whose degree is not divisible by $p$; see \cite[Theorem 1.6]{Yang}.
Under these assumptions, our method (or a simplified version presented in Section \ref{Subsection:Outline}) and Yang's method share several ingredients but there is one difference;
Yang used Kisin's result \cite[Theorem 0.4]{KisinModp} on the CM liftings, up to isogeny, of closed points of the special fiber of the integral canonical model of a Shimura variety of Hodge type, while we give a refinement of Kisin's result (or argument) itself; see Theorem \ref{Theorem:LiftingTorusAction} for details.
\end{rem}

\begin{rem}
Deuring proved that every elliptic curve over a finite field
admits a characteristic $0$ lifting whose generic fiber is
an elliptic curve with CM;
see \cite[Theorem 1.7.4.6]{ChaiConradOort}.
Theorem \ref{Theorem:CMlifting} is an analogue of this result for $K3$ surfaces of finite height.
It is an interesting question to ask whether
Theorem \ref{Theorem:CMlifting} holds
also for supersingular $K3$ surfaces over finite fields.
Our methods in this paper cannot be applied to
supersingular $K3$ surfaces.
\end{rem}

\begin{rem}
We also have similar results on the existence of quasi-canonical liftings (in the sense of Nygaard-Ogus) of $K3$ surfaces of finite height over a finite field; see Corollary \ref{Corollary:QuasiCanonicalLifting}.
\end{rem}

\subsection{The Tate conjecture for the squares of $K3$ surfaces over finite fields}

Next we state our results on the Tate conjecture.
(For the statement of the Tate conjecture,
see \cite[Conjecture 0.1]{NygaardOgus}, \cite[Section 1]{Tate}, \cite[Conjecture 1.1]{Totaro} for example.)

As the second main theorem of this paper,
we shall prove the Tate conjecture for the square
of a $K3$ surface over a finite field.

\begin{thm}[see Theorem \ref{Theorem:TateConjecture:Restate}]
\label{Theorem:TateConjecture}
Let $X$ be a $K3$ surface (of any height) over $\F_q$.
We put
$X \times X := X \times_{\Spec \F_q} X$ and
$X_{\overline{\F}_q} \times X_{\overline{\F}_q} := X_{\overline{\F}_q} \times_{\Spec \overline{\F}_q} X_{\overline{\F}_q}$.
Then, for every $i$, the $\ell$-adic cycle class map
\[
\mathrm{cl}^{i}_{\ell} \colon Z^{i}(X \times X)\otimes_{\Z}\Q_{\ell} \to H^{2i}_{\rm{\acute{e}t}}(
X_{\overline{\F}_q} \times X_{\overline{\F}_q},
\Q_{\ell}{(i)})^{\Gal(\overline{\F}_q/\F_q)}
\]
is surjective for every prime number $\ell \neq p$.
Moreover, for every $i$, the crystalline cycle class map
\[
\mathrm{cl}^{i}_{\cris} \colon Z^{i}(X \times X)\otimes_{\Z}\Q_p \to
H^{2i}_{\cris}( (X \times X)/W(\F_q))^{\varphi = p^i} \otimes_{\Z} \Q
\]
is surjective.
\end{thm}

Here $Z^i(X \times X)$ denotes the group of algebraic cycles of codimension $i$ on $X \times X$,
and $W(\F_q)$ is the ring of Witt vectors of $\F_q$.
The map $\varphi$ denotes the action of the absolute Frobenius endomorphism on
the crystalline cohomology.

\begin{rem}
Theorem \ref{Theorem:TateConjecture} was previously known to hold for some $K3$ surfaces.
\begin{enumerate}
\item Theorem \ref{Theorem:TateConjecture} obviously holds for any $i \notin \{\, 1,2,3 \,\}$.

\item The surjectivity of $\mathrm{cl}^{1}_{\ell}$ and $ \mathrm{cl}^{3}_{\ell}$
follow from the Tate conjecture for $X$
\cite{Charles13, KimMadapusiIntModel, MadapusiTateConj, Maulik};
see also Lemma \ref{Lemma:TateConjecturePrimitivePart}.

\item Theorem \ref{Theorem:TateConjecture} holds when $X$ is supersingular.
In fact, the Tate conjecture for $X$ implies
the Picard number of $X_{\overline{\F}_q}$ is $22$;
see Lemma \ref{Lemma:SupersingularPicardNumber}.
Then the Tate conjecture for the square $X \times X$ follows by the K\"unneth formula;
see Lemma \ref{Lemma:TateConjectureSupersingular}
and Remark \ref{Remark:TateConjectureSupersingularPower}.

\item Zarhin proved the Tate conjecture for $X \times X$
when $X$ is an ordinary $K3$ surface; see \cite[Corollary 6.1.2]{Zarhin}.
Here a $K3$ surface $X$ is called \textit{ordinary} if
it is of height $1$.
(More generally, Zarhin proved the Tate conjecture for
any power $X \times \cdots \times X$ of
an ordinary $K3$ surface $X$.)

\item Yu-Yui proved the Tate conjecture for $X \times X$
when $X$ satisfies some conditions on
the characteristic polynomial of the Frobenius morphism;
see \cite[Lemma 3.5, Corollary 3.6]{Yu-Yui}.
\end{enumerate}
\end{rem}

In the cases studied by Zarhin and Yu-Yui, it turns out that
all the Tate cycles of codimension $2$ on $X \times X$  are spanned by
the classes of the cycles of the form
$X \times \{ x_0 \}$, $\{ x_0 \} \times X$, and $D_1 \times D_2$,
and the classes of the graphs of powers of the Frobenius morphism on $X$.
Here $x_0$ is a closed point on $X$, and $D_1$ and $D_2$ are divisors on $X$.
In general, there are Tate classes on $X \times X$ 
which are not spanned by these classes.
Therefore, in order to prove Theorem \ref{Theorem:TateConjecture} in full generality,
we shall prove the algebraicity of Tate cycles on $X \times X$
which are not spanned by Tate cycles considered by Zarhin and Yu-Yui.
We shall prove it by constructing characteristic $0$ liftings,
and applying the results of Mukai and Buskin on the Hodge conjecture.

\subsection{Construction of characteristic $0$ liftings preserving the action of tori}

Here we explain our results on
the construction of characteristic $0$ liftings of $K3$ surfaces.

Let $X$ be a $K3$ surface over $\F_q$,
and $\mathscr{L}$ a line bundle on $X$ defined over $\F_q$ which gives a primitive quasi-polarization.
Assume that $X$ is of finite height.
After replacing $\F_q$ by a finite extension of it,
the \textit{Kuga-Satake abelian variety} $A$ associated with $(X,\mathscr{L})$
is defined over $\F_q$.
(Precisely, we shall use the Kuga-Satake abelian variety introduced
by Madapusi Pera in \cite{MadapusiIntModel, MadapusiTateConj},
which has dimension $2^{21}$; it is larger than the dimension of the classical Kuga-Satake abelian variety.
See Section \ref{Subsection:Integral canonical models and Kuga-Satake abelian schemes}.)

We have an action of a general spin group,
denoted by $\GSpin(L_{\Q})$ in this paper, on the cohomology of $X$ and $A$.
We put $G := \GSpin(L_{\Q})$ in this section.
We do not recall the precise definition of $G$ here.
Instead, we give some of its properties:
\begin{enumerate}
\item For every prime number $\ell \neq p$,
the group of $\Q_{\ell}$-valued points $G(\Q_{\ell})$ acts
on the primitive part
\[
P^{2}_{\mathrm{\acute{e}t}}( X_{\overline{\F}_q}, \Q_{\ell}{(1)})
:= \mathrm{ch}_{\ell}(\mathscr{L})^{\perp} \subset
H^{2}_{\mathrm{\acute{e}t}}( X_{\overline{\F}_q}, \Q_{\ell}{(1)})
\]
of the $\ell$-adic cohomology of $X$
and the $\ell$-adic cohomology
\[ H^{1}_{\mathrm{\acute{e}t}}( A_{\overline{\F}_q}, \Q_{\ell}) \]
of $A$.
 
\item There is a $G(\Q_{\ell})$-equivariant $\Q_{\ell}$-linear map 
\[ P^{2}_{\mathrm{\acute{e}t}}( X_{\overline{\F}_q}, \Q_{\ell}{(1)})
\to
\End_{\Q_{\ell}}(H^{1}_{\mathrm{\acute{e}t}}( A_{\overline{\F}_q}, \Q_{\ell})^{\vee}),
\]
where $()^{\vee}$ denotes the $\Q_{\ell}$-linear dual.

\item There is an element
$\Frob_{q} \in G(\Q_{\ell})$
whose action on
$P^{2}_{\mathrm{\acute{e}t}}( X_{\overline{\F}_q}, \Q_{\ell}{(1)})$
(resp. $H^{1}_{\mathrm{\acute{e}t}}( A_{\overline{\F}_q}, \Q_{\ell})$)
coincides with the action of the geometric Frobenius morphism on
the $\ell$-adic cohomology of
$X$ (resp.\ $A$).
\end{enumerate}

Following Kisin \cite{KisinModp},
we attach an algebraic group $I$ over $\Q$ to
the quasi-polarized $K3$ surface $(X,\mathscr{L})$;
see Definition \ref{Definition: algebraic group I}.
Instead of giving the precise definition here, we give its properties:
\begin{enumerate}
\item The group of $\Q$-valued points $I(\Q)$
is considered as a subgroup of the multiplicative group of the endomorphism algebra of $A_{\overline{\F}_q}$ tensored with $\Q$:
\[
I(\Q) \subset (\End_{\overline{\F}_q}(A_{\overline{\F}_q}) \otimes_{\Z} \Q)^{\times}.
\]
\item For every prime number $\ell \neq p$,
there is an embedding $I_{\Q_{\ell}} \hookrightarrow G_{\Q_{\ell}}$, and an element of $G(\Q_{\ell})$ is in $I(\Q_{\ell})$
if and only if it commutes with $\Frob^m_{q}$ for a sufficiently divisible $m \geq 1$.
\item The algebraic groups $G$ and $I$ have the same rank.
\end{enumerate}

The existence of an algebraic group $I$ over $\Q$
which satisfies these properties is not obvious;
it is considered as Kisin's group-theoretic interpretation
and generalization of
Tate's original proof of the Tate conjecture for endomorphisms of
abelian varieties over finite fields.

As the third main theorem of this paper,
we shall construct a characteristic $0$ lifting of
a quasi-polarized $K3$ surface of finite height
preserving the action of a maximal torus of
the algebraic group $I$.

\begin{thm}[see Theorem \ref{Theorem:LiftingTorusAction:Restate}]
\label{Theorem:LiftingTorusAction}
Let $T \subset I$ be a maximal torus over $\Q$.
Then there exist a finite extension $K$ of $W(\overline{\F}_q)[1/p]$
and a quasi-polarized $K3$ surface $(\mathcal{X},  \mathcal{L})$ over $\O_K$
such that the special fiber
$(\mathcal{X}_{\overline{\F}_q}, \mathcal{L}_{\overline{\F}_q})$
is isomorphic to $(X_{\overline{\F}_q}, \mathscr{L}_{\overline{\F}_q})$,
and, for every embedding $K \hookrightarrow \C$,
the quasi-polarized $K3$ surface
$(\mathcal{X}_{\C},  \mathcal{L}_{\C})$
satisfies the following properties:
\begin{enumerate}
\item The $K3$ surface $\mathcal{X}_{\C}$ has CM.

\item There is a homomorphism of algebraic groups over $\Q$
\[ T \to \SO(P^2_{B}(\mathcal{X}_{\C}, \Q(1))). \]
Here $P^2_{B}(\mathcal{X}_{\C}, \Q(1))$ is the primitive part of the Betti cohomology of $\mathcal{X}_{\C}$.

\item
For every $\ell \neq p$,
the action of $T(\Q_{\ell})$ on
$P^2_{B}(\mathcal{X}_{\C},\,\Q(1)) \otimes_{\Q} \Q_{\ell}$
is identified with the action of $T(\Q_{\ell})$ on
$P^2_{\mathrm{\acute{e}t}}(X_{\overline{\F}_q},\Q_{\ell}(1))$
via the canonical isomorphisms
\[
P^2_{B}(\mathcal{X}_{\C}, \Q(1)) \otimes_{\Q} \Q_{\ell}
\cong P^2_{\mathrm{\acute{e}t}}(\mathcal{X}_{\C}, \Q_{\ell}(1))
\cong P^2_{\mathrm{\acute{e}t}}(X_{\overline{\F}_q},\Q_{\ell}(1))
\]
(using the embedding $K \hookrightarrow \C$,
we consider $K$ as a subfield of $\C$).

\item The action of every element of $T(\Q)$ on
$P^2_{B}(\mathcal{X}_{\C}, \Q(1))$
preserves the $\Q$-Hodge structure on it.
\end{enumerate}
\end{thm}

\begin{rem}
It is known that every $K3$ surface with CM is defined
over a number field; see 
Proposition \ref{Proposition:CM K3 is defined over a number field} and Remark \ref{Remark:CM K3 is defined over a number field}.
Therefore, Theorem \ref{Theorem:LiftingTorusAction}
implies Theorem \ref{Theorem:CMlifting}.
\end{rem}

\begin{rem}
Our construction of characteristic $0$ liftings
relies on the theory of integral canonical models of Shimura varieties
of Hodge type developed by Milne, Vasiu, Kisin, and Kim-Madapusi Pera.
Especially, we need an explicit description of the completion at a closed point of the special fiber of the integral canonical model of a Shimura variety of Hodge type given by Kisin when $p \geq 3$
\cite{KisinIntModel},
and by Kim-Madapusi Pera when $p=2$ \cite{KimMadapusiIntModel}.
(When $p \geq 5$, we can also use the results of Nygaard-Ogus \cite{NygaardOgus}
to obtain necessary results on characteristic $0$ liftings;
see Section \ref{Subsection:Remarks on the characteristic and the Kuga-Satake morphism} and Remark \ref{Remark:NygaardOgus}.)
\end{rem}

\begin{rem}
When $X$ is ordinary, Theorem \ref{Theorem:LiftingTorusAction}
was essentially proved by Nygaard in \cite{Nygaard}
although the algebraic group $I$ did not appear there.
When $X$ is ordinary, the canonical lifting of $X$ is a CM lifting.
On the other hand, when the height of $X$ is finite and $p \geq 5$,
Nygaard-Ogus proved $X$ admits quasi-canonical liftings \cite{NygaardOgus}.
But a quasi-canonical lifting is not necessarily a CM lifting.
\end{rem}

\subsection{Remarks on the characteristic and the Kuga-Satake morphism}
\label{Subsection:Remarks on the characteristic and the Kuga-Satake morphism}

In this paper, we do not put any restrictions on the characteristic $p$.
There are several technical difficulties in small characteristics.
But, when $p \geq 5$ and $p$ does not divide the degree of the quasi-polarization,
most of the technical difficulties disappear
and the proofs of the main theorems can be considerably simplified.
See below for some details.

We construct characteristic $0$ liftings of quasi-polarized $K3$ surfaces
corresponding to characteristic $0$ liftings of formal Brauer groups
in Section \ref{Section:Lifting of a special point}.
Our construction depends on the calculations of $F$-crystals in Section \ref{Section:$F$-crystals on Shimura varieties},
which depend on the integral comparison theorems of Bhatt-Morrow-Scholze \cite{BMS}.
When $p \geq 5$, we can avoid them.
Instead, we can use the results of Nygaard-Ogus \cite{NygaardOgus}
to obtain necessary results on liftings of $K3$ surfaces;
see Remark \ref{Remark:NygaardOgus}.

Our notation on the Shimura varieties is slightly complicated because
we use the Kuga-Satake morphism introduced by Madapusi Pera in \cite{MadapusiTateConj, KimMadapusiIntModel}, which is denoted by
\[
\KS \colon M^{\mathrm{sm}}_{2d, \K^p_0, \Z_{(p)}}\to Z_{\K^p_0}(\Lambda).
\]
(See Section \ref{Subsection:Kuga-Satake morphism}.)
To define $Z_{\K^p_0}(\Lambda)$,
we embed the Shimura variety, which is the target of the classical Kuga-Satake morphism,
into a larger Shimura variety,
and put additional structures (called \textit{$\Lambda$-structures});
see Definition \ref{Definition:Lambda structure}.

We use the morphism $\KS$ to avoid certain technical difficulties which arise when $p$ divides the degree of the quasi-polarization.
(The same technique was used by Madapusi Pera in
\cite{MadapusiIntModel, MadapusiTateConj, KimMadapusiIntModel}.
See also Remark \ref{Remark:Hyperspecial}.)

When $p$ does not divide the degree of the quasi-polarization, we can avoid it,
and can directly work with the classical Kuga-Satake morphism
into the integral canonical model of the (smaller) Shimura variety.

\begin{rem}
In the course of writing this paper,
we found some issues on the proof of the \'etaleness of the Kuga-Satake morphism.
The \'etaleness was used in our construction of characteristic $0$ liftings
in Section \ref{Section:Lifting of a special point}.
It was also used by Madapusi Pera in his proof of the Tate conjecture for $K3$ surfaces \cite{MadapusiTateConj, KimMadapusiIntModel}.
We can avoid these issues using our results in
Section \ref{Section:$F$-crystals on Shimura varieties};
see Remark \ref{Remark: Kim-Madapusi Pera gap} and
Remark \ref{Remark:Bloch-Kato and BMS} for details.
After we communicated the first draft of this paper to Madapusi Pera, he found a somewhat different argument; see \cite{MadapusiErratum}. 
\end{rem}

\subsection{Outline of the proofs of the main theorems}
\label{Subsection:Outline}

We shall prove Theorem \ref{Theorem:CMlifting} and Theorem \ref{Theorem:LiftingTorusAction} at the same time.
Then, combined with the results of Mukai and Buskin,
we shall prove Theorem \ref{Theorem:TateConjecture}.

\subsubsection*{Proofs of Theorem \ref{Theorem:CMlifting} and Theorem \ref{Theorem:LiftingTorusAction} (when $p \geq 5$)}

In order to simplify the exposition,
we first explain the proof of Theorem \ref{Theorem:LiftingTorusAction}
when $p \geq 5$ using the results of Nygaard-Ogus.
In the following argument, we replace $\F_q$ by a sufficiently large
finite extension of it. We put $W := W(\F_q)$.

Let $\widehat{\Br}:= \widehat{\Br}(X)$ be the formal Brauer group associated with $X$.
First we shall show $I_{\Q_p}$ acts on $\widehat{\Br}$, up to isogeny.
Then we take a finite totally ramified extension $E$ of $W[1/p]$,
and a one-dimensional smooth formal group $\mathcal{G}$ over $\O_{E}$
lifting $\widehat{\Br}$
such that the action of $I_{\Q_p}$ on $\widehat{\Br}$
lifts to an action of $I_{\Q_p}$ on $\mathcal{G}$,
up to isogeny.

The lifting $\mathcal{G}$ defines filtrations on
$P^2_{\cris}(X/W) \otimes_{W} E$ and
$H^1_{\cris}(A/W) \otimes_{W} E$ as follows.
The Kuga-Satake construction gives embeddings which are homomorphisms of $F$-isocrystals after inverting $p$:
\[
\mathbb{D}(\widehat{\Br})(1) \subset
P^2_{\cris}(X/W)(1) \subset
\widetilde{L}_{\cris}  \subset
\End_{W}(H^1_{\cris}(A/W)^{\vee}).
\]
(Here $\mathbb{D}(\widehat{\Br})$ is the Dieudonn\'e module of $\widehat{\Br}$ considered as a connected $p$-divisible group.
For the $W$-module $\widetilde{L}_{\cris}$,
see Section \ref{Subsection:Local Systems on Shimura Varieties} (4).)
The lifting $\mathcal{G}$ defines a filtration on
$\mathbb{D}(\widehat{\Br})(1) \otimes_{W} E$:
\[ \mathrm{Fil}^1(\mathcal{G})
   \subset \mathbb{D}(\widehat{\Br})(1) \otimes_{W} E. \]
Thus, it gives the filtration on $P^2_{\cris}(X/W)(1) \otimes_{W} E$:
\[
  \mathrm{Fil}^1(\mathcal{G})
  \subset \mathrm{Fil}^1(\mathcal{G})^{\perp}
  \subset P^2_{\cris}(X/W)(1) \otimes_{W} E.
\]
Take a generator
$e \in \mathrm{Fil}^1(\mathcal{G})$,
and denote the image of the action of $e$ by
\[ \mathrm{Fil}^1 := \Im(e) \subset H^1_{\cris}(A/W) \otimes_{W} E. \]
It gives a filtration on $H^1_{\cris}(A/W) \otimes_{W} E$,
which does not depend on the choice of $e$.

When $p \geq 5$, the results of Nygaard-Ogus \cite{NygaardOgus} imply
the existence of a lifting $(\mathcal{X}, \mathcal{L})$ over $\O_E$
corresponding to the filtration defined as above.

We shall show that, for every embedding $E \hookrightarrow \C$,
the action of an element of $T(\Q)$ on $P^2_{B}(\mathcal{X}_{\C}, \Q(1))$
preserves the $\Q$-Hodge structure.
To show this, we note that each element of $T(\Q)$
can be considered as an element of
$(\End_{\F_q}(A) \otimes_{\Z} \Q)^{\times}$.
Since its action preserves the filtration on
$H^1_{\cris}(A/W) \otimes_{W} E$,
it lifts to an element of
$(\End_{\C}(\mathcal{A}_{\C}) \otimes_{\Z} \Q)^{\times}$,
where $\mathcal{A}_{\C}$ is the Kuga-Satake abelian variety over $\C$
associated with $(\mathcal{X}_{\C}, \mathcal{L}_{\C})$.
In particular, it preserves the Hodge structure on the singular cohomology
$H^1_{B}(\mathcal{A}_{\C}, \Q)$.
Since we have a $T(\Q)$-equivariant embedding respecting the $\Q$-Hodge structures
\[
P^2_{B}(\mathcal{X}_{\C}, \Q(1)) \hookrightarrow
\End_{\C}(H^1_{B}(\mathcal{A}_{\C}, \Q)^{\vee}),
\]
the action of each element of $T(\Q)$ on $P^2_{B}(\mathcal{X}_{\C}, \Q(1))$
preserves the $\Q$-Hodge structure on it.

As the algebraic groups $G$ and $I$ have the same rank,
we conclude that the Mumford-Tate group of
$P^2_{B}(\mathcal{X}_{\C}, \Q(1))$
is commutative.
Thus $\mathcal{X}_{\C}$ is a $K3$ surface with CM.
Consequently, the quasi-polarized $K3$ surface
$(\mathcal{X}_E, \mathcal{L}_E)$ is defined over a number field,
and Theorem \ref{Theorem:CMlifting} and Theorem \ref{Theorem:LiftingTorusAction}
are proved when $p \geq 5$.

\subsubsection*{Proofs of Theorem \ref{Theorem:CMlifting} and Theorem \ref{Theorem:LiftingTorusAction} (for $p = 2$ or $3$)}

When $p = 2$ or $3$, we cannot use the results of Nygaard-Ogus
to construct liftings of $K3$ surfaces.
Instead, we use $p$-adic Hodge theory and
the \'etaleness of the Kuga-Satake morphism to construct 
liftings.
(The following argument works for any $p$, including $p \geq 5$.)

Using the $p$-adic Tate module of $\mathcal{G}$ and its dual,
we first construct a $\Z_p[\Gal(\overline{E}/E)]$-module $\widetilde{L}_p$
such that $\widetilde{L}_p[1/p]$ is a crystalline $\Gal(\overline{E}/E)$-representation
whose Hodge-Tate weights are in $\{ -1, 0, 1 \}$;
see Lemma \ref{Lemma:Lifting-LCris}.
On the other hand, we can show the filtration
$\mathrm{Fil}^1 \subset H^1_{\cris}(A/W) \otimes_{W} E$
gives the structure of
a weakly admissible filtered $\varphi$-module on $H^1_{\cris}(A/W)[1/p]$.
It corresponds to a crystalline representation of $\Gal(\overline{E}/E)$,
which is denoted by $H_{\mathrm{\acute{e}t}, \Q_p}$.

Next we find a $\Gal(\overline{E}/E)$-stable $\Z_p$-lattice
in $H_{\mathrm{\acute{e}t}, \Q_p}$ as follows.
We can show there is an embedding of
$\Gal(\overline{E}/E)$-representations
\[
\widetilde{L}_p[1/p] \hookrightarrow \End_{\Q_p}(H_{\mathrm{\acute{e}t}, \Q_p}).
\]
Moreover, it can be shown that there is an isomorphism
of $\Q_p$-vector spaces
\[ \Cl(\widetilde{L}_p[1/p]) \cong H_{\mathrm{\acute{e}t}, \Q_p} \]
such that the actions of $\Gal(\overline{E}/E)$ on $\widetilde{L}_p$
and $H_{\mathrm{\acute{e}t}, \Q_p}$ factor through a homomorphism
\[
\Gal(\overline{E}/E) \to \GSpin(\widetilde{L}_p)(\Z_p)
\subset \Cl(\widetilde{L}_p)^{\times}.
\]
Here $\Cl(\widetilde{L}_p)^{\times}$ acts on
$\Cl(\widetilde{L}_p[1/p]) \cong H_{\mathrm{\acute{e}t}, \Q_p}$
by the left multiplication.
We take a $\Gal(\overline{E}/E)$-stable $\Z_p$-lattice
in $H_{\mathrm{\acute{e}t}, \Q_p}$
corresponding to 
$\Cl(\widetilde{L}_p) \subset \Cl(\widetilde{L}_p[1/p])$.
Then we take a $p$-divisible group $\mathcal{H}$ over $\O_E$
corresponding to it.

Let $K$ be the composite of $E$ and $W(\overline{\F}_q)[1/p]$.
We can show that the $p$-divisible group $\mathcal{H}_{\O_K}$ satisfies
a certain technical condition, called ``adaptedness.''
We have an $\O_K$-valued point $\widetilde{s}$ on
the integral canonical model of the Shimura variety lifting
the $\overline{\F}_q$-valued point associated with $(X,\mathscr{L})$.
(Precisely, $\widetilde{s}$ is an $\O_K$-valued point on the target $Z_{\K^p_0}(\Lambda)$ of the Kuga-Satake morphism.)
It gives an abelian scheme $\mathcal{A}$ over $\O_K$ lifting $A$
whose associated $p$-divisible group $\mathcal{A}[p^{\infty}]$ is $\mathcal{H}_{\O_K}$.

By the \'etaleness of the Kuga-Satake morphism,
we obtain a quasi-polarized $K3$ surface $(\mathcal{X}, \mathcal{L})$ over $\O_K$
corresponding to $\widetilde{s}$.

The rest of the argument is the same as before.

\subsubsection*{Proof of Theorem \ref{Theorem:TateConjecture}}

Fix a prime number $\ell \neq p$.
By the K\"unneth formula, we have
\begin{align*}
&\quad H^{4}_{\rm{\acute{e}t}}( X_{\overline{\F}_q} \times X_{\overline{\F}_q}, \Q_{\ell}{(2)}) \\
&\cong \bigoplus_{(i,j) = (0,4),(2,2)(4,0)}
H^{i}_{\rm{\acute{e}t}}( X_{\overline{\F}_q}, \Q_{\ell})
\otimes_{\Q_{\ell}} H^{j}_{\rm{\acute{e}t}}( X_{\overline{\F}_q}, \Q_{\ell})
\otimes_{\Q_{\ell}} \Q_{\ell}(2).
\end{align*}
It is enough to show every element fixed by $\Frob_q$
in the component of type $(2,2)$ is spanned by the classes of algebraic cycles of codimension $2$ on $X \times X$.
By the Poincar\'e duality, such an element can be considered as an endomorphism of
$H^{2}_{\rm{\acute{e}t}}(  X_{\overline{\F}_q}, \Q_{\ell}(1))$
commuting with $\Frob_q$.

Thus, we consider the action of $I(\Q_{\ell})$ on
$P^{2}_{\rm{\acute{e}t}}(  X_{\overline{\F}_q}, \Q_{\ell}(1))$.
It can be shown that,
after replacing $\F_q$ by a finite extension of it,
there exist maximal tori $T_1,\ldots,T_n \subset I$ over $\Q$ such that
the $\Q_{\ell}$-vector space of endomorphisms on
$P^{2}_{\rm{\acute{e}t}}(X_{\overline{\F}_q}, \Q_{\ell}(1))$
commuting with $\Frob_q$ is spanned by
the images of $T_1(\Q),\ldots,T_n(\Q)$.

Therefore, it is enough to show that, for each $i$, the action of every element of $T_i(\Q)$
on $P^{2}_{\rm{\acute{e}t}}(  X_{\overline{\F}_q}, \Q_{\ell}(1))$
comes from an algebraic cycle of codimension $2$ on $X \times X$.
It can be proved by combining
Theorem \ref{Theorem:LiftingTorusAction}
with the results of Mukai and Buskin on the Hodge conjecture for
certain Hodge cycles on the product of two $K3$ surfaces over $\C$.

\subsection{Outline of this paper}

As explained in Section \ref{Subsection:Remarks on the characteristic and the Kuga-Satake morphism},
most of the technical difficulties can be avoided when $p \geq 5$ and $p$ does not divide the degree of the quasi-polarization.
The readers who are mainly interested in the applications
of the integral canonical models of orthogonal Shimura varieties to
CM liftings and the Tate conjecture may skip earlier sections for the first time of reading, and may go directly to
Section \ref{Section:Lifting of a special point}.

The organization of this paper is as follows.

In Section \ref{Section:Clifford algebras and general spin groups},
we recall basic results on Clifford algebras and general spin groups.
In Section \ref{Section:Breuil-Kisin modules},
we recall basic results on Breuil-Kisin modules and
integral $p$-adic Hodge theory.
Then, in Section \ref{Section:Shimura varieties} and
Section \ref{Section:Moduli spaces of K3 surfaces and the Kuga-Satake morphism},
we fix notation and recall necessary results on the integral canonical models of orthogonal Shimura varieties and
the Kuga-Satake morphism used in this paper.

In Section \ref{Section:$F$-crystals on Shimura varieties},
we compare $F$-crystals on Shimura varieties and the crystalline cohomology of $K3$ surfaces.
We essentially use the integral comparison theorems of Bhatt-Morrow-Scholze \cite{BMS}.
We also explain how to avoid some issues on the proof of the \'etaleness of the Kuga-Satake morphism;
see Remark \ref{Remark: Kim-Madapusi Pera gap} and
Remark \ref{Remark:Bloch-Kato and BMS}.

In Section \ref{Section:Lifting of a special point},
we construct a characteristic $0$ lifting of
the $\overline{\F}_q$-valued point on the Shimura variety
corresponding to characteristic $0$ liftings of formal Brauer group
of the $K3$ surface.
We construct such liftings using our results on $F$-crystals in
Section \ref{Section:$F$-crystals on Shimura varieties}
and $p$-adic Hodge theory.

In Section \ref{Section:Kisin's algebraic groups},
we define and study an analogue of Kisin's algebraic group
associated with a quasi-polarized $K3$ surface of finite height over a finite field.
We also study its action on the formal Brauer group.

In Section \ref{Section:Lifting of K3 surfaces over finite fields with actions of tori},
we combine our results in
Section \ref{Section:Lifting of a special point} and
Section \ref{Section:Kisin's algebraic groups}
to construct a characteristic $0$ lifting of a $K3$ surface of finite height over $\overline{\F}_q$
preserving the action of a maximal torus of $I$.
Then we prove Theorem \ref{Theorem:CMlifting} and Theorem \ref{Theorem:LiftingTorusAction}.
In Section \ref{Section:TateConjecture},
combined with the results of Mukai and Buskin,
we prove Theorem \ref{Theorem:TateConjecture}.

Finally, in Section \ref{Section:CompatibilityComparisonIsom},
we give necessary results on the compatibility of
several comparison isomorphisms in $p$-adic Hodge theory used in this paper.

\subsection{Notation}\label{Subsection: Notation}

Throughout this paper, we fix a prime number $p$,
and $q$ denotes a power of $p$.
We denote a finite field with $q$ elements by $\F_q$,
and an algebraic closure of $\F_q$ by $\overline{\F}_q$.

For a perfect field $k$ of characteristic $p >0 $,
the ring of Witt vectors of $k$ is denoted by $W(k)$.
The Frobenius automorphism of $W(k)$ is denoted by
$\sigma \colon W(k) \to W(k)$.
If the field $k$ is clear from the context, 
we omit $k$ and denote it simply by $W$.

A $\textit{quadratic space}$ over a commutative ring $R$ means a free $R$-module $M$ of finite rank
equipped with a quadratic form $Q$.
We equip $M$ with a symmetric bilinear pairing $( \ , \ )$ defined by
$(x, y)=Q(x+y)-Q(x)-Q(y)$ for $x, y \in M$.
For a module $M$ over a commutative ring $R$ equipped with
a symmetric bilinear form $( \ , \ )$, we say $M$
(or the bilinear form $( \ , \ )$) is \textit{even}
if, for every $x \in M$, we have $(x, x)=2a$ for some $a \in R$.

The base change of a module or a scheme is denoted by a subscript.
For example, for a module $M$ over a commutative ring $R$ and an $R$-algebra $R'$,
the tensor product $M\otimes_{R}R'$ is denoted by $M_{R'}$.
For a scheme (or an algebraic space) $X$ over $R$,
the base change $X \times_{\Spec R} \Spec R'$ is denoted by
$X_{R'}$.
We use similar notation for the base change of group schemes,
$p$-divisible groups, line bundles, morphisms between them, etc.
For a homomorphism $f \colon M \to N$ of $R$-modules,
the base change $f_{R'} \colon M_{R'} \to N_{R'}$
is also denoted by the same notation $f$ if there is no possibility of confusion.
For an element $x \in M$, the $R$-submodule of $M$ generated by $x$ is denoted by
$\langle x \rangle$.
The dual of $M$ as an $R$-module is denoted by
$M^{\vee} := \Hom_R(M,R)$.

\section{Clifford algebras and general spin groups}\label{Section:Clifford algebras and general spin groups}
In this section, we introduce notation on quadratic spaces and Clifford algebras which will be used in this paper.
Our basic references are \cite{Bass}, \cite[Section 1]{MadapusiIntModel}.

\subsection{Embeddings of lattices}\label{Subsection: Embeddings of lattices}

A quadratic space 
$U:=\Z{x} \oplus \Z{y}$
whose associated bilinear form is given by $(x, x)=(y, y)=0$ and $(x, y)=1$ is called the \textit{hyperbolic plane}. 
The \textit{$K3$ lattice} $\Lambda_{K3}$ is defined by
\[
\Lambda_{K3} := {E_8}^{\oplus{2}}\oplus{U}^{\oplus{3}},
\]
which is a quadratic space over $\Z$.
It is unimodular and the signature of it is $(19, 3)$. 

We fix a positive integer $d >0$.
Let $L$ be an orthogonal complement of $x - dy$ in $\Lambda_{K3}$, where $x - dy$ is considered as an element in the third $U$. 
Hence $L$ is equal to
\[
{E_8}^{\oplus{2}}\oplus{U}^{\oplus{2}} \oplus \langle x + dy \rangle,
\]
and the signature of it is $(19, 2)$.

The following result is well-known.  

\begin{lem}\label{LatticeEmbedding}
Let $p$ be a prime number.
There is a quadratic space $\widetilde{L}$ of rank $22$ over $\Z$
satisfying the following properties:
\begin{enumerate}
\item The signature of it is $(20, 2)$.
\item $\widetilde{L}$ is self-dual at $p$ (i.e.\ the discriminant of $\widetilde{L}$ is not divisible by $p$).
\item There is an embedding 
$L \hookrightarrow \widetilde{L}$
as quadratic spaces which sends $L$ onto a direct summand of $\widetilde{L}$ as a $\Z$-module. 
\end{enumerate}
\end{lem}

\begin{proof}
This result was proved in \cite[Lemma 6.8]{MadapusiIntModel} when $p>2$.
Here we briefly give a proof which is valid for every prime number $p$.
We consider a quadratic space 
$L' := \Z{v_1} \oplus \Z{v_2}$
such that the associated bilinear form is given by
$(v_1, v_1)=2d$, $(v_2, v_2)=2p$, and $(v_1, v_2)=1.$
We put 
\[
 \widetilde{L} := {E_8}^{\oplus{2}}\oplus{U}^{\oplus{2}} \oplus L'.
 \]
This is of signature $(20, 2)$ and self-dual at $p$.
We have an embedding of quadratic spaces
\[
L = {E_8}^{\oplus{2}}\oplus{U}^{\oplus{2}} \oplus \langle x + dy \rangle
\hookrightarrow  \widetilde{L}
\]
which is the identity on ${E_8}^{\oplus{2}}\oplus{U}^{\oplus{2}}$
and sends $x + dy$ to $v_1$.
\end{proof}

\subsection{Clifford algebras and general spin groups}\label{Subsection: Clifford algebras and general spin groups}

In the rest of this paper, we fix an embedding of quadratic spaces $L \subset \widetilde{L}$ as in Lemma \ref{LatticeEmbedding}.

Let $\Cl:=\Cl(\widetilde{L})$ be the \textit{Clifford algebra} over $\Z$
associated with the quadratic space $(\widetilde{L}, q_{\widetilde{L}})$. 
There is an embedding of $\Z$-modules
$\widetilde{L} \hookrightarrow \Cl$
which is universal for morphisms
$f \colon \widetilde{L} \to R$ 
of $\Z$-modules into an associative $\Z$-algebra $R$
such that $f(v)^2=q_{\widetilde{L}}(v)$ for every $v \in \widetilde{L}$.  
The algebra $\Cl$ has a $\Z/2\Z$-grading structure
$\Cl:=\Cl^{+} \oplus \Cl^{-}$,
where $\Cl^{+}$ is a subalgebra of $\Cl$.
The quadratic space $\widetilde{L}$ is naturally embedded into $\Cl^{-}$.

Let $\Z_{(p)}$ be the localization of $\Z$ at $p$.
We define the \textit{general spin group} 
$\widetilde{G}:=\GSpin(\widetilde{L}_{\Z_{(p)}})$
over $\Z_{(p)}$ by
\[
\widetilde{G}(R):=\{\, g \in (\Cl^{+}_R)^{\times} \mid g\widetilde{L}_Rg^{-1}=\widetilde{L}_R \, \, \text{in} \, \Cl^{-}_R \, \}
\]
for every $\Z_{(p)}$-algebra $R$.
Since $\widetilde{L}$ is self-dual at $p$, 
the group scheme $\widetilde{G}$ is a reductive group scheme over $\Z_{(p)}$.

The \textit{special orthogonal group} 
$\widetilde{G}_0:= \SO(\widetilde{L}_{\Z_{(p)}})$
is a reductive group scheme over $\Z_{(p)}$,
whose generic fiber $\widetilde{G}_{0, \Q}:=\widetilde{G}_0 \otimes_{\Z_{(p)}}\Q$ is $\SO(\widetilde{L}_{\Q})$.

We have the canonical morphism 
$\widetilde{G} \to \widetilde{G}_0$
defined by $g \mapsto (v \mapsto gvg^{-1})$, 
whose kernel is the multiplicative group $\G_{m, \Z_{(p)}}$ over $\Z_{(p)}$.
We have the following exact sequence of group schemes over $\Z_{(p)}$:
\[
1 \to
\G_{m, \Z_{(p)}} \to
\widetilde{G}=\GSpin(\widetilde{L}_{\Z_{(p)}}) \to
\widetilde{G}_0=\SO(\widetilde{L}_{\Z_{(p)}}) \to 1.
\]

\subsection{Representations of general spin groups and Hodge tensors}
\label{Subsection: Representations of general spin groups and Hodge tensors}

We define a $\Z$-module $H$ by $H:= \Cl$.
We consider $H_{\Z_{(p)}}$ as a $\widetilde{G}$-representation over $\Z_{(p)}$
by the left multiplication.
We have a closed embedding of group schemes over $\Z_{(p)}$:
\[
\widetilde{G} \hookrightarrow \GL(H_{\Z_{(p)}}).
\]

The representation $H_{\Z_{(p)}}$ has a natural $\Z/2\Z$-grading structure,
which is preserved by the action of $\widetilde{G}$.
Let $p^{+} \colon H_{\Z_{(p)}} \to H^{+}_{\Z_{(p)}} \hookrightarrow H_{\Z_{(p)}}$ (resp.\ $p^{-} \colon H_{\Z_{(p)}} \to H^{-}_{\Z_{(p)}} \hookrightarrow H_{\Z_{(p)}}$)
be an idempotent corresponding to $H^{+}_{\Z_{(p)}}$ (resp.\ $H^{-}_{\Z_{(p)}}$).

The representation $H_{\Z_{(p)}}$ is equipped with a right action of $\Cl_{\Z_{(p)}}$ given by the right multiplication, 
which commutes with the action of $\widetilde{G}$.
We fix a $\Z_{(p)}$-basis $\{ e_i \}_{1 \leq i \leq 2^{22}}$ for $\Cl_{\Z_{(p)}}$, and let $r_{e_i} \colon H_{\Z_{(p)}} \to H_{\Z_{(p)}}$ be the endomorphism defined by $x \mapsto xe_i$.

We regard $\widetilde{L}_{\Z_{(p)}}$
as a $\widetilde{G}$-representation via the canonical homomorphism 
$\widetilde{G} \to \widetilde{G}_0$ as above. 
Then the injective homomorphism 
\[
i \colon \widetilde{L}_{\Z_{(p)}} \hookrightarrow \End_{\Z_{(p)}}(H_{\Z_{(p)}})\]
defined by $v \mapsto (h \mapsto vh)$ is $\widetilde{G}$-equivariant.
The cokernel of this homomorphism $i$ is torsion-free as a $\Z_{(p)}$-module.

As in \cite[Section 1]{MadapusiIntModel},
we define a non-degenerate symmetric bilinear form  $[ \ , \ ]$ on
$\End_{\Q}(H_{\Q})$
by
\[
[g_1 , g_2]:=2^{-21} \cdot \mathrm{Tr}(g_1 \circ g_2)
\]
for $g_1, g_2 \in \End_{\Q}(H_{\Q})$.
Then the embedding
$i \otimes_{\Z_{(p)}}\Q$
is an isometry.
In \cite[Lemma 1.4]{MadapusiIntModel}, Madapusi Pera proved that there is a unique idempotent 
\[
\pi \colon \End_{\Q}(H_{\Q}) \to \End_{\Q}(H_{\Q})
\]
satisfying the following properties:
\begin{enumerate}
	\item The image of $\pi$ is $i(\widetilde{L}_{\Q})$.
	\item The kernel of $\pi$ is the orthogonal complement $i(\widetilde{L}_{\Q})^{\perp}$ of $i(\widetilde{L}_{\Q})$ in $\End_{\Q}(H_{\Q})$ with respect to the bilinear pairing $[ \ , \ ]$.
	\item $\widetilde{G}_{\Q}$ is the stabilizer of the $\Z/2\Z$-grading structure, the right action of $\Cl_{\Z_{(p)}}$, and the idempotent $\pi$, i.e.\ the stabilizer of $p^{\pm}$, $\{ r_{e_i} \}_{1 \leq i \leq 2^{22}}$, and $\pi$.
\end{enumerate}

As in \cite[(1.3.1)]{KisinIntModel}, let $H^{\otimes}_{\Z_{(p)}}$ be the direct sum of all $\Z_{(p)}$-modules obtained from $H_{\Z_{(p)}}$ by taking tensor products, duals, symmetric powers, and exterior powers. (In fact, symmetric powers and exterior powers are unnecessary; see \cite{Deligne:LettertoKisin}.)
By \cite[Proposition 1.3.2]{KisinIntModel}, the group scheme $\widetilde{G}$ over $\Z_{(p)}$
is the stabilizer of a finite collection of tensors
\[
\{ s_{\alpha}\} \subset H^{\otimes}_{\Z_{(p)}}.
\]
(See also \cite[Lemma 4.7]{KimMadapusiIntModel}.)

In the rest of this paper, we fix such tensors
$\{ s_{\alpha}\} $.
We may assume that $\{ s_{\alpha}\}$ includes the tensors $\{ s_{\beta}\}$ corresponding to $p^{\pm}$, $\{ r_{e_i} \}_{1 \leq i \leq 2^{22}}$, and the endomorphism $\pi'$, where we put $\pi':=p^n\pi$ for a sufficiently large $n$ such that $\pi'$ maps $\End_{\Z_{(p)}}(H_{\Z_{(p)}})$ into itself.

\subsection{Filtrations on Clifford algebras defined by isotropic elements}\label{Subsection: Filtrations on Clifford algebras defined by isotropic elements}

In this subsection, let $F$ be a field of characteristic $0$.

Take a non-zero element $e \in \widetilde{L}_F$ satisfying $(e, e)=0$. 
We consider an endomorphism
\[
i(e):=(i\otimes_{\Z_{(p)}}F)(e) \in \End_{F}(H_{F})
\]
which is the image of $e$ under the embedding
\[
i\otimes_{\Z_{(p)}}F \colon \widetilde{L}_{F} \hookrightarrow \End_{F}(H_{F}).\]
Let $i(e)(H_F)$ be the image of the endomorphism
$i(e) \colon H_F \to H_F$.

We define a decreasing filtration
$\{\Fil^i(\widetilde{L}_F)\}_i$ (resp.\ $\{\Fil^i(H_F)\}_i$)
on $\widetilde{L}_F$ (resp.\ $H_F$)
and a decreasing filtration  on $H_F$ by
\begin{align*}
\mathrm{Fil}^i(\widetilde{L}_F)&:=
\begin{cases}
0 & i \geq 2, \\
\langle e \rangle & i = 1, \\
\langle e \rangle^{\perp} & i=0, \\
\widetilde{L}_F & i \leq -1,
\end{cases}
&
\mathrm{Fil}^i(H_F) &:=
\begin{cases}
0 & i \geq 1, \\
i(e)(H_F) & i=0, \\
H_F & i \leq -1.
\end{cases}
\end{align*}

\begin{prop}\label{Proposition: Filtrations on Clifford algebras}
\begin{enumerate}
	\item The dimension of $\Fil^0(H_F)$ as an $F$-vector space is $2^{21}$.
	\item We define a decreasing filtration
	$\{\Fil^i(\End_{F}(H_{F}))\}_i$ on $\End_{F}(H_{F})$ by
	\[
\Fil^i(\End_F(H_F))=\{\ g \in \End_F(H_F) \mid g(\Fil^j(H_F)) \subset \Fil^{i+j}(H_F) \ \text{for every $j$} \ \}.
\]
Then the homomorphism 
	\[
	i \otimes_{\Z_{(p)}}F \colon \widetilde{L}_{F} \hookrightarrow \End_{F}(H_{F})
	\]
	preserves the filtrations.
	\item The $\Z/2\Z$-grading structure and the right action of $\Cl_{\Z_{(p)}}$ preserve the filtration $\{\Fil^i(H_F) \}_i$ on $H_F$.
	The endomorphism $\pi' = p^n \pi$ of $\End_F(H_F)$ preserve the filtration on
	$\{\Fil^i(\End_{F}(H_{F}))\}_i$ on $\End_{F}(H_{F})$.
\end{enumerate}
\end{prop}
\begin{proof}
$(1)$
We have a decomposition
\[
\widetilde{L}_F=\langle e \rangle \oplus \langle f \rangle \oplus (\langle e \rangle \oplus \langle f \rangle)^{\perp} 
\]
such that $(e, f)=1$ and $(f, f)=0$.
Let $v_3, \dotsc, v_{22}$ be an orthogonal basis for $(\langle e \rangle \oplus \langle f \rangle)^{\perp}$.
We have
\[
H_F = \bigoplus_{a_j \in \{ 0, 1 \}} \langle e^{a_1}f^{a_2}v^{a_3}_3\dots v^{a_{22}}_{22} \rangle
\]
by \cite[\S 9.3, Th\'eor{\`e}me 1]{Bourbaki}.

Since $e^2=2^{-1}(e, e)=0$ in the Clifford algebra $H_F=\Cl_F$,
we have
\[
\Fil^0(H_F)= \bigoplus_{a_j \in \{ 0, 1 \}} \langle ef^{a_2}v^{a_3}_3\dots v^{a_{22}}_{22} \rangle.
\] 
Hence the dimension of $i(e)(H_F)$ as an $F$-vector space is $2^{21}$.

$(2)$
The assertion immediately follows from the description of the filtration.

$(3)$
It is clear that the $\Z/2\Z$-grading structure and the right action of $\Cl_{\Z_{(p)}}$ preserve the filtration on $H_F$.
We shall show the idempotent $\pi$ preserves the filtration on $\End_F(H_F)$.

We recall an explicit description of $\pi$ from \cite[Section 1]{MadapusiIntModel}.
We put $\beta_j:=(v_j, v_j)^{-1}$ for $3 \leq j \leq 22$.
Then $\pi$ sends $g \in \End_F(H_F)$ to the element $\pi(g) \in \End_F(H_F)$ given by
\[
\pi(g)=[g, i(f)]i(e) + [g, i(e)]i(f) + \sum_{3 \leq j \leq 22}\beta_j[g, i(v_j)]i(v_j).
\]
Recall that we have $[g_1, g_2] = 2^{-21} \cdot \mathrm{Tr}(g_1 \circ g_2)$ for $g_1, g_2 \in \End_{\Q}(H_{\Q})$,
and the embedding
\[ i \otimes_{\Z_{(p)}}F \colon \widetilde{L}_{F} \hookrightarrow \End_{F}(H_{F}) \]
preserves filtrations by $(2)$.
Hence it is enough to show the following assertions to prove that $\pi$ preserves the filtration on $\End_F(H_F)$:
\begin{itemize}
	\item $\mathrm{Tr}(g \circ i(e))=0$ for every $g \in \Fil^0(\End_F(H_F))$.
	\item $\mathrm{Tr}(g \circ i(e))=0$ and $\mathrm{Tr}(g \circ i(v_j))=0$ for every $g \in \Fil^1(\End_F(H_F))$ and every $3 \leq j \leq 22$.
\end{itemize}
These assertions can be checked by using the explicit basis for $H_F$ as above.
\end{proof}

\section{Breuil-Kisin modules}
\label{Section:Breuil-Kisin modules}

\subsection{Preliminaries}\label{Subsection: Preliminaries}

In this section, we fix a perfect field $k$ of characteristic $p>0$,
and an algebraic closure $\overline{k}$ of $k$.
To simplify the notation, we put $W := W(k)$.

Let $K$ be a finite totally ramified extension of $W[1/p]$.
Let $\overline{K}$ be an algebraic closure of $K$.
We fix a uniformizer $\varpi$ of $K$, and a system
$\{ \varpi^{1/{p^n}} \}_{n \geq 0} \subset \overline{K}$ of $p^n$-th roots of $\varpi$ such that
$(\varpi^{1/{p^{n+1}}})^p=\varpi^{1/{p^n}}$.
Let $E(u) \in W[u]$ be the (monic) Eisenstein polynomial of $\varpi$
(i.e.\ it is the monic minimal polynomial of $\varpi$ over $W[1/p]$).

Let $C$ be the completion of $\overline{K}$.
Let $\O_C$ be the ring of integers of $C$.
We put 
\[
\O^{\flat}_C:=\plim[x \mapsto x^p]\O_C/p,
\]
which is a perfect $\F_p$-algebra.
It is an integral domain and we denote the field of fractions of it by $C^{\flat}$.
Let
\[ A_{\mathrm{inf}} := W(\O^{\flat}_C) \]
be the ring of Witt vectors of $\O^{\flat}_C$.
We denote by $\varphi \colon A_{\mathrm{inf}} \to A_{\mathrm{inf}}$ the automorphism induced by the Frobenius and the functoriality of Witt vectors.
There is a unique $\Gal(\overline{K}/K)$-equivariant surjection
\[
\theta \colon A_{\mathrm{inf}} \twoheadrightarrow \O_C
\]
such that its reduction modulo $p$ is the first projection
$
\O^{\flat}_C \to \O_C/p.
$
We also have a surjection
$A_{\mathrm{inf}} \twoheadrightarrow W(\overline{k})$ induced by the natural surjection $\O^{\flat}_C \to \overline{k}$.

We put
\[
\mathfrak{S} := W[[u]].
\]
The ring $\mathfrak{S}$ admits a Frobenius endomorphism $\varphi$
which acts on $W$ as the canonical Frobenius $\sigma$ and sends $u$ to $u^p$.

For the fixed system
$\{ \varpi^{1/{p^n}} \}_{n \geq 0} \subset \overline{K}$ of $p^n$-th roots of $\varpi$,
we put
\[
K_{\infty}:=\bigcup_{n\geq0}K(\varpi^{1/{p^n}}).
\]
We also put
$\varpi^{\flat}:=(\varpi^{1/{p^n}} \!\!\mod p)_{n \geq 0} \in \O^{\flat}_C$.
We have a $W$-linear homomorphism
\[
\mathfrak{S} \to A_{\mathrm{inf}}
\]
which sends $u$ to the Teichm\"uller lift
$[\varpi^{\flat}]$
of 
$\varpi^{\flat}$.
This homomorphism is compatible with the Frobenius endomorphisms.
Note that the composite
\[ \mathfrak{S} \to A_{\mathrm{inf}} \to \O_C \]
is a $W$-linear homomorphism given by 
$u \mapsto \varpi$.

We also need the $p$-adic period rings $B_{\dR}$, $B^{+}_{\dR}$, $B_{\cris}$, and $B^{+}_{\cris}$ associated with $C$ defined by Fontaine.
For the definition and basic properties of these rings,
see \cite[Section 3.3]{BMS} for example.

\subsection{Breuil-Kisin modules and crystalline Galois representations}
\label{Subsection:Breuil-Kisin modules and crystalline Galois representations}

In this subsection, we recall some basic results on Breuil-Kisin modules and crystalline Galois representations from \cite[(1.2)]{KisinIntModel}.

For an $\mathfrak{S}$-module $\M$, we put
\[
\varphi^{*}\M:=\mathfrak{S} \otimes_{\varphi, \mathfrak{S}}\M.
\]
A \textit{Breuil-Kisin module} (over $\mathscr{O}_K$ with respect to $\{ \varpi^{1/{p^n}} \}_{n \geq 0}$) is a free $\mathfrak{S}$-module $\mathfrak{M}$ of finite rank equipped with an isomorphism of $\mathfrak{S}[1/E(u)]$-modules
\[
1\otimes\varphi \colon (\varphi^{*}\mathfrak{M})[1/E(u)] \cong \mathfrak{M}[1/E(u)].
\]
We say a Breuil-Kisin module $\M$ is \textit{effective}
if the equipped isomorphism $1 \otimes \varphi$  of $\M$ is induced by a homomorphism
\[
1\otimes\varphi \colon \varphi^{*}\mathfrak{M} \to \mathfrak{M}.
\]
We say an effective Breuil-Kisin module $\M$ is of height $\leq h$
if the cokernel of $1 \otimes \varphi$ is killed by $E(u)^h$.

For a Breuil-Kisin module $\M$, we define a decreasing filtration
$\{ \Fil^i(\varphi^{*}\M) \}_i$
on
$\varphi^{*}\M$ by
\[
\Fil^i(\varphi^{*}\M) := \{ \ x \in \varphi^{*}\M \mid (1\otimes \varphi)(x) \in E(u)^i\M \ \}
\]
for every $i \in \Z$.

We say a $\Z_p[\mathrm{Gal}(\overline{K}/K)]$-module $N$ is a \textit{$\mathrm{Gal}(\overline{K}/K)$-stable $\Z_p$-lattice in a crystalline representation}
if $N$ is a free $\Z_p$-module of finite rank and $N[1/p]$ is a crystalline $\mathrm{Gal}(\overline{K}/K)$-representation.
Kisin constructed a \textit{covariant} fully faithful tensor functor 
\[
N \mapsto \mathfrak{M}(N)
\]
from the category of $\mathrm{Gal}(\overline{K}/K)$-stable $\Z_p$-lattices in crystalline representations to the category of Breuil-Kisin modules (over $\O_K$ with respect to $\{ \varpi^{1/{p^n}} \}_{n \geq 0}$); see \cite[Theorem 1.2.1]{KisinIntModel}.

In the rest of this subsection, we fix a $\mathrm{Gal}(\overline{K}/K)$-stable $\Z_p$-lattice $N$ in a crystalline representation.
We put
\[
\MdR(N):=\varphi^{*}\M(N)\otimes_{\mathfrak{S}}\O_K,
\]
where $\mathfrak{S} \to \O_K$ is a $W$-linear homomorphism given by $u \mapsto \varpi$.
Let $\Fil^i(\MdR(N)[1/p])$ be the $K$-vector subspace generated by the image of the filtration $\Fil^i(\varphi^{*}\M(N))$.
We have a decreasing filtration
$\{ \Fil^i(\MdR(N)[1/p]) \}_i$ on $\MdR(N)[1/p]$.
We have a canonical isomorphism
\[
D_{\dR}(N[1/p]):=(N\otimes_{\Z_p}B_{\dR})^{\Gal(\overline{K}/K)}
 \cong \MdR(N)[1/p],
\]  
which maps $\Fil^i(\MdR(N)[1/p])$ onto $\Fil^i(D_{\dR}(N[1/p]))$; see \cite[Theorem 1.2.1]{KisinIntModel}.
We define a decreasing filtration
$\{ \Fil^i(\MdR(N)) \}_i$ on $\MdR(N)$ by taking intersection 
\[
\Fil^i(\MdR(N)):=\Fil^i(\MdR(N)[1/p]) \cap \MdR(N).
\]
Note that, by the construction, the quotient
\[
\Gr^{i}(\mathfrak{M}_{\dR}(N)):=\Fil^i(\mathfrak{M}_{\dR}(N))/\Fil^{i+1}(\mathfrak{M}_{\dR}(N))
\]
is a free $\O_K$-module of finite rank for every $i \in \Z$.

\begin{lem}\label{Lemma: 1-st Filtration}
	Assume $\M(N)$ is effective.
	Then the image of $\Fil^1(\varphi^{*}\M(N))$ under the surjection
	\[
	\varphi^{*}\M(N) \twoheadrightarrow \MdR(N)
	\]
	coincides with $\Fil^1(\MdR(N))$.
\end{lem}
\begin{proof}
Let $\Fil'$ be the image of $\Fil^1(\varphi^{*}\M(N))$ under the surjection
	$
	\varphi^{*}\M(N) \twoheadrightarrow  \MdR(N)
	$.
It suffices to show the cokernel of the inclusion
	$
	\Fil' \hookrightarrow \MdR(N)
	$
	is $p$-torsion-free.
Since
\[
E(u)\varphi^{*}\M(N) \subset \Fil^1(\varphi^{*}\M(N)),
\]
the inverse image of $\Fil'$ under the  surjection
$\varphi^{*}\M(N) \twoheadrightarrow \MdR(N)$
is $\Fil^1(\varphi^{*}\M(N))$.
Hence the assertion follows from that the cokernel of
\[
\Fil^1(\varphi^{*}\M(N)) \hookrightarrow \varphi^{*}\M(N)
\]
is $p$-torsion-free.
\end{proof}

We put
\[
\Mcris(N) := \varphi^{*}\M(N) \otimes_{\mathfrak{S}} W,
\]
where $\mathfrak{S} \to W$ is a $W$-linear homomorphism given by
$u \mapsto 0$.
The Frobenius $1 \otimes \varphi$ of $\M(N)$
defines a $\sigma$-semilinear endomorphism of $\Mcris(N)[1/p]$, which makes $\Mcris(N)[1/p]$ a $\varphi$-module.
We have a canonical isomorphism of $\varphi$-modules
\[
D_{\cris}(N[1/p]):=(N\otimes_{\Z_p}B_{\cris})^{\Gal(\overline{K}/K)}  \cong \Mcris(N)[1/p];
\]  
see \cite[Theorem 1.2.1]{KisinIntModel}.

Let $N$ be a $\mathrm{Gal}(\overline{K}/K)$-stable free $\Z_p$-module of finite rank in a crystalline representation and $f$ a $\mathrm{Gal}(\overline{K}/K)$-equivariant endomorphism of $N$.
The endomorphism of the Breuil-Kisin module $\M(N)$ induced by $f$ is denoted by $\M(f)$.
The endomorphism of $\MdR(N)$ (resp.\ $\Mcris(N)$) induced by $f$ is denoted by $\MdR(f)$ (resp.\ $\Mcris(f)$).

\subsection{Breuil-Kisin modules and $p$-divisible groups}
\label{Subsection:Breuil-Kisin modules and p-divisible groups}

Let $\mathscr{G}$ be a $p$-divisible group over $\O_K$.
Let 
\[
T_{p}\mathscr{G}:= \plim[n]\mathscr{G}[p^n]({\overline{K}})
\]
be the $p$-adic Tate module of $\mathscr{G}$,
which is a free $\Z_p$-module of finite rank and admits a continuous action of $\Gal(\overline{K}/K)$.

For the base change $\mathscr{G}_k$ of $\mathscr{G}$,
we have a  (contravariant) crystal $\mathbb{D}(\mathscr{G}_k)$ over $\mathrm{CRIS}(k/\Z_p)$;
see \cite[D\'efinition 3.3.6]{BBM}.
Here $\mathrm{CRIS}(k/\Z_p)$ is the (absolute) crystalline site of $k$.
Its value
\[
\mathbb{D}(\mathscr{G}_k)(W) := \mathbb{D}(\mathscr{G}_k)_{W \twoheadrightarrow k}
\]
in $(\Spec k \hookrightarrow \Spec W)$
is an $F$-crystal.

For the base change $\mathscr{G}_{{\O_K}/p}$ of $\mathscr{G}$,
we have a crystal $\mathbb{D}(\mathscr{G}_{{\O_K}/p})$ over $\mathrm{CRIS}(({\O_K}/p)/\Z_p)$.
Its value
\[
\mathbb{D}(\mathscr{G}_{{\O_K}/p})(\O_K):=\mathbb{D}(\mathscr{G}_{{\O_K}/p})_{\O_K \twoheadrightarrow {\O_K}/p}
\]
in $(\Spec {{\O_K}/p} \hookrightarrow \Spec {\O_K})$
is a free $\O_K$-module of finite rank and admits a Hodge filtration
\[
\Fil^1\mathbb{D}(\mathscr{G}_{{\O_K/p}})(\O_K) \hookrightarrow \mathbb{D}(\mathscr{G}_{{\O_K}/p})(\O_K).
\]
By \cite[Proposition 3.14]{BO}, there is an isomorphism over $K$:
\[
\mathbb{D}(\mathscr{G}_k)(W)\otimes_{W} K \cong \mathbb{D}(\mathscr{G}_{{\O_K}/p})(\O_K)\otimes_{\O_K}K.
\]
Using this isomorphism,
we consider
$
\mathbb{D}(\mathscr{G}_k)(W)[1/p]
$
as a filtered $\varphi$-module.
(For the notion of filtered $\varphi$-modules, see \cite[1.1.3]{KisinCris} for example.)

In \cite[Theorem 7]{Faltings99}, Faltings constructed a $\Gal(\overline{K}/K)$-equivariant isomorphism
\[
T_{p}\mathscr{G}[1/p] \cong \Hom_{\Fil, \varphi}(\mathbb{D}(\mathscr{G}_k)(W)[1/p], B^{+}_{\cris}).
\]
(See also \cite[Section 5.2 Remarque 2]{Fargues}
and Section \ref{Subsection:Comparison isomorphisms for $p$-divisible groups}.)
It induces an isomorphism
\[
D_{\cris}((T_{p}\mathscr{G})^{\vee}[1/p]) \cong \mathbb{D}(\mathscr{G}_k)(W)[1/p], 
\]
which gives an isomorphism of filtered $\varphi$-modules. We denote it by $c_{\mathscr{G}}$. 

There are integral refinements of them:
\begin{enumerate}
    \item The composite
\[\Mcris((T_{p}\mathscr{G})^{\vee})[1/p] \cong D_{\cris}((T_{p}\mathscr{G})^{\vee}[1/p])  \overset{c_{\mathscr{G}}}{\cong} \mathbb{D}(\mathscr{G}_k)(W)[1/p]
	\]
	maps $\Mcris((T_{p}\mathscr{G})^{\vee})$ onto $\mathbb{D}(\mathscr{G}_k)(W)$.
    \item The composite
	\[
	\MdR((T_{p}\mathscr{G})^{\vee})[1/p] \cong D_{\dR}((T_{p}\mathscr{G})^{\vee}[1/p]) 
	\overset{c_{\mathscr{G}}}{\cong} \mathbb{D}(\mathscr{G}_{{\O_K/p}})(\O_K)\otimes_{\O_K}K
	\]
	maps $\MdR((T_{p}\mathscr{G})^{\vee})$ onto $\mathbb{D}(\mathscr{G}_{{\O_K/p}})(\O_K)$, and maps $\Fil^1(\MdR((T_{p}\mathscr{G})^{\vee}))$ onto $\Fil^1\mathbb{D}(\mathscr{G}_{{\O_K/p}})(\O_K)$.
\end{enumerate}

These results were proved by Kisin when $p >2$; see \cite[Theorem 1.4.2]{KisinIntModel}.
The general case follows from the results of Lau \cite{LauDisplay, LauGalois} as explained in \cite[Theorem 2.12]{KimMadapusiIntModel}.
(See also Remark \ref{Remark:Faltings integral refinement}.)

\subsection{Integral $p$-adic Hodge theory}
\label{Subsection:The integral p-adic Hodge theory}

In this subsection, we recall integral comparison theorems
proved by Bhatt-Morrow-Scholze \cite{BMS}.
(See also Section \ref{Subsection:Scholze's de Rham comparison map}, and \ref{Subsection:Crystalline comparison map of Bhatt-Morrow-Scholze}.)
Although the results in \cite{BMS} can be applied in more general cases (including the semistable case \cite{CK}), we only recall their results for $K3$ surfaces with good reduction
for simplicity.

Let
$\mathscr{X}$
be a $K3$ surface over $\O_K$.
It is possibly an algebraic space, not a scheme.
We remark that the generic fiber $\mathscr{X}_K$ and the special fiber $\mathscr{X}_k$ are schemes since a smooth proper algebraic space of dimension $2$ over a field is a scheme.
We refer to Section \ref{Subsection:Crystalline comparison map of Bhatt-Morrow-Scholze} for details on how to apply the results of \cite{BMS} to the proper smooth algebraic space $\mathscr{X}$ over $\O_K$.
We will freely use GAGA results implicitly below.

By \cite[Theorem 14.6 (i)]{BMS}, we have the following $B_{\cris}$-linear isomorphism
\[
c_{\cris, \mathscr{X}}\colon 
H^2_{\cris}(\mathscr{X}_k /W) \otimes_{W} B_{\cris} \overset{\cong}{\longrightarrow} H^2_{\et} (\mathscr{X}_{\overline{K}}, \Z_p)\otimes_{\Z_p} B_{\cris},
\]
which is compatible with the action of $\Gal(\overline{K}/K)$ and the Frobenius endomorphisms.
By its construction and \cite[Theorem 13.1]{BMS}, this is compatible with the following filtered $B_{\dR}$-linear isomorphism constructed in \cite[Theorem 8.4]{ScholzeHodge}
\[
c_{\dR, \mathscr{X}_K}\colon H^2_{\dR}(\mathscr{X}_K/K)\otimes_K B_{\dR} \overset{\cong}{\longrightarrow} H^2_{\et}(\mathscr{X}_{\overline{K}}, \Z_p) \otimes_{\Z_p} B_{\dR},
\]
which is $\Gal(\overline{K}/K)$-equivariant.
More precisely, $c_{\cris, \mathscr{X}}$ and $c_{\dR, \mathscr{X}_K}$ are compatible via the Berthelot-Ogus isomorphism
\[
H^2_{\cris}(\mathscr{X}_k/W)\otimes_{W} K \cong H^2_{\dR}(\mathscr{X}_K/K), 
\]
see Proposition \ref{Proposition:Comparison with BO}. 

The isomorphism $c_{\dR, \mathscr{X}_K}$ induces an isomorphism
\[
D_{\dR}(H^2_{\mathrm{\acute{e}t}}(\mathscr{X}_{\overline{K}}, \Z_p)[1/p]) \cong H^2_{\dR}(\mathscr{X}_K/K),
\]
which is also denoted by $c_{\dR, \mathscr{X}_K}$.
Similarly, the isomorphism $c_{\cris, \mathscr{X}}$ induces an isomorphism
\[
D_{\cris}(H^2_{\mathrm{\acute{e}t}}(\mathscr{X}_{\overline{K}}, \Z_p)[1/p])\cong H^2_{\cris}(\mathscr{X}_k/W)[1/p],
\]
which is also denoted by $c_{\cris, \mathscr{X}}$.

Let
$\M(H^2_{\mathrm{\acute{e}t}}(\mathscr{X}_{\overline{K}}, \Z_p))$
be the Breuil-Kisin module
(over $\O_K$ with respect to $\{ \varpi^{1/{p^n}} \}_{n \geq 0}$)
associated with
$H^2_{\mathrm{\acute{e}t}}(\mathscr{X}_{\overline{K}}, \Z_p)$.
We have the composite of the following filtered isomorphisms
\begin{align*}
\MdR(H^2_{\mathrm{\acute{e}t}}(\mathscr{X}_{\overline{K}}, \Z_p))[1/p] &\cong
D_{\dR}(H^2_{\mathrm{\acute{e}t}}(\mathscr{X}_{\overline{K}}, \Z_p)[1/p]) \\
&\cong H^2_{\dR}(\mathscr{X}_K/K). 
\end{align*}
We also have the composite of the following isomorphisms of $\varphi$-modules
\begin{align*}
\Mcris(H^2_{\mathrm{\acute{e}t}}(\mathscr{X}_{\overline{K}}, \Z_p))[1/p] &\cong
D_{\cris}(H^2_{\mathrm{\acute{e}t}}(\mathscr{X}_{\overline{K}}, \Z_p)[1/p]) \\
&\cong H^2_{\cris}(\mathscr{X}_k/W)[1/p]. 
\end{align*}

As in the case of $p$-divisible groups, there are integral refinements of them:

\begin{thm}[Bhatt-Morrow-Scholze \cite{BMS}]
\label{Theorem:p-adic Hodge theorem}
\begin{enumerate}
	\item The isomorphism
	\[
	\MdR(H^2_{\mathrm{\acute{e}t}}(\mathscr{X}_{\overline{K}}, \Z_p))[1/p] \cong H^2_{\dR}(\mathscr{X}_K/K)
	\]
	maps
	$\MdR(H^2_{\mathrm{\acute{e}t}}(\mathscr{X}_{\overline{K}}, \Z_p))$
	isomorphically onto
	$H^2_{\dR}(\mathscr{X}/\O_K)$.
	
	\item The isomorphism
	\[
	\Mcris(H^2_{\mathrm{\acute{e}t}}(\mathscr{X}_{\overline{K}}, \Z_p))[1/p] \cong H^2_{\cris}(\mathscr{X}_k/W)[1/p]
	\]
	maps
	$\Mcris(H^2_{\mathrm{\acute{e}t}}(\mathscr{X}_{\overline{K}}, \Z_p))$
	isomorphically onto
	$H^2_{\cris}(\mathscr{X}_k/W)$.
\end{enumerate}
\end{thm}

\begin{proof}
$(1)$ Since this part was not stated explicitly in \cite{BMS}, we shall explain how to deduce it from the results of \cite{BMS}.

First recall that
\[ \varphi^*\M(H^2_{\mathrm{\acute{e}t}}(\mathscr{X}_{\overline{K}}, \Z_p)) \otimes_{\mathfrak{S}}A_{\mathrm{inf}} \]
naturally becomes a \textit{Breuil-Kisin-Fargues module} in the sense of \cite[Definition 4.22]{BMS}. Under Fargues' equivalence \cite[Theorem 4.28]{BMS}, this corresponds to a $B_{\dR}^+$-lattice
\begin{align*}
D_{\dR}(H^2_{\et}(\mathscr{X}_{\overline{K}}, \Z_p)[1/p]) \otimes_K B_{\dR}^+  =
(H^2_{\et}(\mathscr{X}_{\overline{K}}, \Z_p) \otimes_{\Z_p}B_{\dR})^{\Gal(\overline{K}/K)} \otimes_K B_{\dR}^+ & \\ \subset 
H^2_{\et}(\mathscr{X}_{\overline{K}}, \Z_p) \otimes_{\Z_p}B_{\dR}. &
\end{align*}
Note that the following isomorphism
\begin{align*}
D_{\dR}(H^2_{\et}(\mathscr{X}_{\overline{K}}, \Z_p)[1/p]) & \cong 
(\varphi^*\M(H^2_{\mathrm{\acute{e}t}}(\mathscr{X}_{\overline{K}}, \Z_p)) \otimes_{\mathfrak{S}} C)^{\Gal (\overline{K}/K)} \\
& \cong \M_{\dR}(H^2_{\mathrm{\acute{e}t}}(\mathscr{X}_{\overline{K}}, \Z_p))[1/p] 
\end{align*}
obtained by the reduction modulo $\xi$ is the same as the one given before. 

The above Breuil-Kisin-Fargues module is realized as the $A_{\mathrm{inf}}$-valued cohomology constructed in \cite{BMS}.
In \cite{BMS}, Bhatt-Morrow-Scholze constructed the $A_{\mathrm{inf}}$-valued cohomology $H^2_{A_\mathrm{inf}}(\mathscr{X})$,
and showed that other cohomology theories can be obtained from it. (Here, $H^2_{A_\mathrm{inf}}(\mathscr{X})$ denotes the $A_\mathrm{inf}$-cohomology of the completion of $\mathscr{X}_{\O_C}$.)
In our case, $H^2_{A_\mathrm{inf}}(\mathscr{X})$ is a free $A_{\mathrm{inf}}$-module of rank $22$.
By \cite[Theorem 14.3]{BMS},
$H^2_{A_\mathrm{inf}}(\mathscr{X})$
is equipped with an $A_{\mathrm{inf}}$-linear Frobenius isomorphism
\[
1 \otimes \varphi \colon (A_{\mathrm{inf}}\otimes_{\varphi, A_{\mathrm{inf}}}H^2_{A_\mathrm{inf}}(\mathscr{X}))[1/\varphi(\xi)] \cong H^2_{A_\mathrm{inf}}(\mathscr{X})[1/\varphi(\xi)],
\]
where $\xi$ is a generator of the kernel of the surjection
$\theta \colon A_{\mathrm{inf}} \twoheadrightarrow \O_C$.
Hence it is a Breuil-Kisin-Fargues module. 
By \cite[Theorem 14.6]{BMS}, it is isomorphic to the Breuil-Kisin-Fargues module
associated with
$H^2_{\et}(\mathscr{X}_{\overline{K}}, \Z_p)$ using $c_{\dR, \mathscr{X}_K}$.
Indeed, by \cite[Theorem 14.3(iv)]{BMS}, there is an isomorphism (actually, there are two constructions giving the same map, see \cite[Proposition 6.8]{CK}):
\[
H^2_{A_\mathrm{inf}}(\mathscr{X}) \otimes_{A_\mathrm{inf}} B_{\dR} \cong
H^2_{\et}(\mathscr{X}_{\overline{K}}, \Z_p)\otimes_{\Z_P}B_{\dR}, 
\]
and $H^2_{A_\mathrm{inf}}(\mathscr{X})$ is determined by, via Fargues' equivalence, a $B_{\dR}^+$-lattice
\[
c_{\dR, \mathscr{X}_K} (H^2_{\dR}(\mathscr{X}_K/K) \otimes_K B_{\dR}^+)
\subset H^2_{\et}(\mathscr{X}_{\overline{K}}, \Z_p)\otimes_{\Z_P}B_{\dR}^+, 
\]
which is exactly
\[ D_{\dR}(H^2_{\et}(\mathscr{X}_{\overline{K}}, \Z_p)[1/p]) \otimes_K B_{\dR}^+. \] 

Therefore, there is an isomorphism
\[
\varphi^*\M(H^2_{\mathrm{\acute{e}t}}(\mathscr{X}_{\overline{K}}, \Z_p)) \otimes_{\mathfrak{S}}A_{\mathrm{inf}} \cong
H^2_{A_\mathrm{inf}}(\mathscr{X})
\]
making the following diagram commutative
\[
\xymatrix{
H^2_{\dR}(\mathscr{X}_K/K)\otimes_{K}C \ar_{\cong}[d] \ar^{\cong}[r] &
H^2_{A_\mathrm{inf}}(\mathscr{X}) \otimes_{A_\mathrm{inf}} C  \\ 
D_{\dR}(H^2_{\mathrm{\acute{e}t}}(\mathscr{X}_{\overline{K}}, \Z_p)[1/p])\otimes_{K}C \ar^-{\cong}[r] &
\MdR(H^2_{\mathrm{\acute{e}t}}(\mathscr{X}_{\overline{K}}, \Z_p)) \otimes_{\O_K}C \ar_{\cong}[u],
}
\]
and $\MdR(H^2_{\mathrm{\acute{e}t}}(\mathscr{X}_{\overline{K}}, \Z_p)) \otimes_{\O_K}\O_C$ is identified with $H^2_{A_\mathrm{inf}}(\mathscr{X}) \otimes_{A_\mathrm{inf}} \O_C$. The inverse of the composite of the left vertical arrow and the bottom arrow is the map appeared in the statement of $(1)$. 

By \cite[Theorem 14.3(ii)]{BMS}, we have an isomorphism
\[
H^2_{A_\mathrm{inf}}(\mathscr{X}) \otimes_{A_\mathrm{inf}} \O_C \cong
H^2_{\dR}(\mathscr{X}/\O_K)\otimes_{\O_K} \O_C. 
\]
It is compatible with the top arrow in the above diagram; this is essentially checked in the proof of \cite[Theorem 14.1]{BMS}, see also \cite[Proposition 5.41, Theorem 6.6]{CK}. 
We finish the proof by taking the $\Gal(\overline{K}/K)$-invariants. 

$(2)$ See \cite[Theorem 14.6(iii)]{BMS}.
We briefly recall the arguments for the convenience of the reader.
By \cite[Theorem 14.5(i), Theorem 12.1]{BMS}, there is an isomorphism 
\[
H^2_{A_\mathrm{inf}}(\mathscr{X}) \otimes_{A_\mathrm{inf}} B_{\cris} \cong
H^2_{\et}(\mathscr{X}_{\overline{K}}, \Z_p)\otimes_{\Z_P}B_{\cris}
\]
that underlies the following isomorphism used in the proof of (1)
\[
H^2_{A_\mathrm{inf}}(\mathscr{X}) \otimes_{A_\mathrm{inf}} B_{\dR} \cong
H^2_{\et}(\mathscr{X}_{\overline{K}}, \Z_p)\otimes_{\Z_P}B_{\dR}. 
\]
This induces an isomorphism
\[
H^2_{A_\mathrm{inf}}(\mathscr{X}) \otimes_{A_\mathrm{inf}} B_{\cris}^+\cong
D_{\cris}(H^2_{\et}(\mathscr{X}_{\overline{K}}, \Z_p)[1/p])\otimes_{W[1/p]}B_{\cris}^+
\]
such that the composite
\begin{align*}
\varphi^*\M(H^2_{\mathrm{\acute{e}t}}(\mathscr{X}_{\overline{K}}, \Z_p)) \otimes_{\mathfrak{S}}B_{\cris}^+
&\cong H^2_{A_\mathrm{inf}}(\mathscr{X}) \otimes_{A_\mathrm{inf}} B_{\cris}^+ \\
&\cong D_{\cris}(H^2_{\et}(\mathscr{X}_{\overline{K}}, \Z_p)[1/p])\otimes_{W[1/p]}B_{\cris}^+
\end{align*}
is the canonical map.
Namely, the map obtained by the specialization
\begin{align*}
D_{\cris}(H^2_{\et}(\mathscr{X}_{\overline{K}}, \Z_p)[1/p])&\cong 
(\Mcris(H^2_{\mathrm{\acute{e}t}}(\mathscr{X}_{\overline{K}}, \Z_p))\otimes_{W}W(\overline{k})[1/p])^{\Gal (\overline{K}/K)} \\ &\cong
\Mcris(H^2_{\mathrm{\acute{e}t}}(\mathscr{X}_{\overline{K}}, \Z_p))[1/p]  
\end{align*}
is equal to the canonical map given before.

Moreover, by the construction of $c_{\cris, \mathscr{X}}$, the following isomorphism \cite[Theorem 14.3(i)]{BMS}
\[
H^2_{A_\mathrm{inf}}(\mathscr{X}) \otimes_{A_\mathrm{inf}} W(\overline{k}) \cong
H^2_{\cris}(\mathscr{X}_k/W)\otimes_{W}W(\overline{k})
\]
induces $c_{\cris, \mathscr{X}}$ by the specialization:
\begin{align*}
D_{\cris}(H^2_{\et}(\mathscr{X}_{\overline{K}}, \Z_p)[1/p])&\cong 
(H^2_{\cris}(\mathscr{X}_k/W)\otimes_{W}W(\overline{k})[1/p])^{\Gal (\overline{K}/K)} \\ &\cong
H^2_{\cris}(\mathscr{X}_k/W)[1/p].  
\end{align*}

In conclusion, under the map in the statement of (2), $\Mcris(H^2_{\mathrm{\acute{e}t}}(\mathscr{X}_{\overline{K}}, \Z_p))$ is identified with
\[
(H^2_{A_\mathrm{inf}}(\mathscr{X}) \otimes_{A_\mathrm{inf}} W(\overline{k}))^{\Gal(\overline{K}/K)}\cong
H^2_{\cris}(\mathscr{X}_k/W). 
\]
This completes the proof. 
\end{proof}

\section{Shimura varieties}
\label{Section:Shimura varieties}

In this section, we recall basic results on Shimura varieties and
their integral models associated with general spin groups and special orthogonal groups.
We follow Madapusi Pera's paper \cite{MadapusiIntModel} for orthogonal Shimura varieties.
(For integral models of more general Shimura varieties of abelian type, see Kisin's paper \cite{KisinIntModel}.
For the construction of $2$-adic integral canonical models,
see also \cite{KimMadapusiIntModel}.)

In this section, we use the same notation as in previous sections.
Recall that we fix an embedding of quadratic spaces $L \subset \widetilde{L}$ as in Lemma \ref{LatticeEmbedding}.
\subsection{Orthogonal Shimura varieties over $\Q$}\label{Subsection:Orthogonal Shimura Varieties over Q}
Let $X_{\widetilde{L}}$ be the symmetric domain 
of oriented negative definite planes in $\widetilde{L}_{\R}$. 
We have Shimura data $(\widetilde{G}_{0, \Q}, X_{\widetilde{L}})$
and $(\widetilde{G}_{\Q}, X_{\widetilde{L}})$.
Each of them has the reflex field $\Q$; see \cite[Appendix 1, Lemma]{Andre} for example.
The canonical homomorphism
$\widetilde{G} \to \widetilde{G}_0$
induces a morphism of Shimura data
$(\widetilde{G}_{\Q}, X_{\widetilde{L}}) \to
 (\widetilde{G}_{0, \Q}, X_{\widetilde{L}})$.

We put $\widetilde{\mathrm{K}}_{0, p} := \widetilde{G}_0(\Z_{p})$ (resp.\ $\widetilde{\mathrm{K}}_{p} := \widetilde{G}(\Z_p)$), 
which is a hyperspecial subgroup. 
Let 
$\widetilde{\mathrm{K}}^p_0 \subset \widetilde{G}_0(\A^p_f)$
(resp.\ $\widetilde{\mathrm{K}}^p \subset \widetilde{G}(\A^p_f)$)
be an open compact subgroup and 
$\widetilde{\mathrm{K}}_0 := \widetilde{\mathrm{K}}_{0, p}\widetilde{\mathrm{K}}^p_0 \subset \widetilde{G}_0(\A_f)$
(resp.\ $\widetilde{\mathrm{K}} := \widetilde{\mathrm{K}}_{p}\widetilde{\mathrm{K}}^p \subset \widetilde{G}(\A_f)$).
We have the Shimura varieties
$\mathrm{Sh}_{\widetilde{\mathrm{K}}_0} := \mathrm{Sh}_{\widetilde{\mathrm{K}}_0}(\widetilde{G}_{0, \Q}, X_{\widetilde{L}})$,
$\mathrm{Sh}_{\widetilde{\mathrm{K}}} :=\mathrm{Sh}_{\widetilde{\mathrm{K}}}(\widetilde{G}_{\Q}, X_{\widetilde{L}})$
associated with the Shimura data
$(\widetilde{G}_{0, \Q}, X_{\widetilde{L}})$, $(\widetilde{G}_{\Q}, X_{\widetilde{L}})$.
We assume $\widetilde{\mathrm{K}}^p_0$ and $\widetilde{\mathrm{K}}^p$ are small enough such that 
$\mathrm{Sh}_{\widetilde{\mathrm{K}}_0}$
and
$\mathrm{Sh}_{\widetilde{\mathrm{K}}}$
are smooth quasi-projective schemes over $\Q$.  
Moreover, we assume the image of
$\widetilde{\mathrm{K}}^p$
under the homomorphism
$\widetilde{G} \to \widetilde{G}_0$
is
$\widetilde{\mathrm{K}}^p_0$. 
Then we have a finite \'etale morphism over $\Q$:
\[
\mathrm{Sh}_{\widetilde{\mathrm{K}}} \to \mathrm{Sh}_{\widetilde{\mathrm{K}}_0}.
\]

We also consider the reductive group $\SO(L_{\Q})$ over $\Q$.
Let $X_{L}$ be the symmetric domain 
of oriented negative definite planes in $L_{\R}$. 
We have a Shimura datum $(\SO(L_{\Q}), X_{L})$ and a morphism of Shimura data:
$(\SO(L_{\Q}), X_{L}) \to  (\widetilde{G}_{0, \Q}, X_{\widetilde{L}})$.
Let $\K_{0, p} \subset \SO(L_{\Q})(\Q_p)$ be the maximal subgroup which stabilizes $L_{\Z_p}$ and acts on $L^{\vee}_{\Z_p}/L_{\Z_p}$ trivially.
Let $\K^p_{0} \subset \SO(L_{\Q})(\A^p_f)$ be an open compact subgroup which stabilizes $L_{\widehat{\Z}^p}$ and acts on $L^{\vee}_{\widehat{\Z}^p}/L_{\widehat{\Z}^p}$ trivially.
We assume $\K^p_{0}$ is small enough such that it is contained in $\widetilde{\K}^p_0$ and the associated Shimura variety $\mathrm{Sh}_{\mathrm{K}_0}(\SO(L_{\Q}), X_{L})$ is a smooth quasi-projective variety over $\Q$, where $\mathrm{K}_0:=\K_{0, p}\K^p_{0}$.
Note that we have $\K_{0, p} \subset \widetilde{\K}_{0, p}$; see the proof of \cite[Lemma 2.6]{MadapusiIntModel}.
Hence we have a morphism of Shimura varieties over $\Q$:
\[
\mathrm{Sh}_{\mathrm{K}_0}(\SO(L_{\Q}), X_{L}) \to \mathrm{Sh}_{\widetilde{\mathrm{K}}_0}.
\]
\subsection{Symplectic embeddings of general spin groups}\label{Subsection: Symplectic embeddings of general spin groups}

By \cite[Lemma 3.6]{MadapusiIntModel}, there is a non-degenerate alternating bilinear form
\[
\psi \colon H \times H \to \Z
\]
satisfying the following properties: 
\begin{enumerate}
\item The left multiplication induces a closed embedding of algebraic groups over $\Q$
\[
\widetilde{G}_{\Q} \hookrightarrow \GSp := \GSp(H_{\Q}, \psi_{\Q}). 
\] 
\item The left multiplication induces a morphism of Shimura data
\[
(\widetilde{G}_{\Q}, X_{\widetilde{L}}) \to (\GSp, S^{\pm}),
\]
where $S^{\pm}$ denotes the Siegel double spaces associated with 
the symplectic space $(H_{\Q}, \psi)$.
\end{enumerate}

Let $\mathrm{K}'^p \subset \GSp(\A^p_f)$ be an open compact subgroup
containing the image of $\widetilde{\mathrm{K}}^p$.
Let $\mathrm{K}'_p \subset \GSp(\Q_p)$ be the stabilizer of $H_{\Z_p}$.
We put $\mathrm{K}' := \mathrm{K}'_p\mathrm{K}'^p$.
After replacing $\mathrm{K}'^p$ and $\widetilde{\mathrm{K}}^p$ by their open compact subgroups, we may assume that the associated Shimura variety $\mathrm{Sh}_{\mathrm{K}'}(\GSp, S^{\pm})$ is a smooth quasi-projective scheme over $\Q$ and
the morphism of Shimura data
$
(\widetilde{G}_{\Q}, X_{\widetilde{L}}) \to (\GSp, S^{\pm})
$
induces the following morphism of Shimura varieties over $\Q$:
\[
\mathrm{Sh}_{\widetilde{\mathrm{K}}} \to \mathrm{Sh}_{\mathrm{K}'}(\GSp, S^{\pm}).
\]

Let us summarize our situation by the following commutative diagram of algebraic groups over $\Q$:
\[
\xymatrix{ \GSpin(L_{\Q}) \ar@{^{(}->}[r]^-{} \ar[d]^-{} & \widetilde{G}_{\Q} \ar[d]^-{} \ar@{^{(}->}[r]^-{} & \GSp\\
\SO(L_{\Q}) \ar@{^{(}->}[r]^-{} \ar[r]^-{} & \widetilde{G}_{0, \Q}. &
}
\]
We also have the corresponding diagram of Shimura varieties over $\Q$:
\[
\xymatrix{
&
\mathrm{Sh}_{\widetilde{\mathrm{K}}} = \mathrm{Sh}_{\widetilde{\mathrm{K}}}(\widetilde{G}_{\Q}, X_{\widetilde{L}}) \ar[d]^-{} \ar[r]^-{} &
\mathrm{Sh}_{\mathrm{K}'}(\GSp, S^{\pm}) \\
\mathrm{Sh}_{\mathrm{K}_0}(\SO(L_{\Q}), X_{L}) \ar[r]^-{} \ar[r]^-{} &
\mathrm{Sh}_{\widetilde{\mathrm{K}}_0} = \mathrm{Sh}_{\widetilde{\mathrm{K}}_0}(\widetilde{G}_{0, \Q}, X_{\widetilde{L}}). &
}
\]

\subsection{Integral canonical models and the Kuga-Satake abelian scheme}\label{Subsection:Integral canonical models and Kuga-Satake abelian schemes}
 
Since
$\widetilde{\mathrm{K}}_{0, p} \subset \widetilde{G}_0(\Q_{p})$
and 
$\widetilde{\mathrm{K}}_{p} \subset \widetilde{G}(\Q_p)$
are hyperspecial subgroups,
the Shimura varieties 
$\mathrm{Sh}_{\widetilde{\mathrm{K}}}$, 
$\mathrm{Sh}_{\widetilde{\mathrm{K}}_0}$
admit the integral canonical models 
$\mathscr{S}_{\widetilde{\mathrm{K}}}$,
$\mathscr{S}_{\widetilde{\mathrm{K}}_0}$
over $\Z_{(p)}$, respectively.
(This result is proved by Kisin when $p \neq 2$ \cite{KisinIntModel},
and by Kim-Madapusi Pera when $p=2$ \cite{KimMadapusiIntModel}.
The integral canonical models are characterized by the extension properties.
See \cite{KisinIntModel} for details.)

By the construction of $\mathscr{S}_{\widetilde{\mathrm{K}}_0}$, the morphism 
$\mathrm{Sh}_{\widetilde{\mathrm{K}}} \to \mathrm{Sh}_{\widetilde{\mathrm{K}}_0}$
 extends to 
a finite \'etale morphism
$\mathscr{S}_{\widetilde{\mathrm{K}}} \to \mathscr{S}_{\widetilde{\mathrm{K}}_0}$
over $\Z_{(p)}$.

Let $m := |H^{\vee}_{\Z}/H_{\Z}|$ be the discriminant of $H_{\Z}$. 
We put $g:=(\dim_{\Q}H_{\Q})/2 = 2^{21}$. 
Let
$\mathscr{A}:=\mathscr{A}_{g, m, \mathrm{K}'}$
be the moduli space over $\Z_{(p)}$
of triples $(A, \lambda, \epsilon'^p)$ consisting 
an abelian scheme $A$ of dimension $g$, 
a polarization $\lambda \colon A \to A^{*}$ of degree $m$, 
and a $\mathrm{K}'^p$-level structure $\epsilon'^p$.
For a sufficiently small $\mathrm{K}'^p$, 
this is represented by a quasi-projective scheme over $\Z_{(p)}$.
We have a canonical open and closed immersion
\[
\mathrm{Sh}_{\mathrm{K}'}(\GSp, S^{\pm}) \hookrightarrow \mathscr{A}_{\Q}
\]
over $\Q$. 
Hence, we have a morphism 
$
\mathrm{Sh}_{\widetilde{\mathrm{K}}} \to \mathscr{A}_{\Q}.
$
over $\Q$.
By the construction of $\mathscr{S}_{\widetilde{\mathrm{K}}}$, 
this morphism extends to a morphism 
$\mathscr{S}_{\widetilde{\mathrm{K}}} \to \mathscr{A}$
over $\Z_{(p)}$; see \cite[(2.3.3)]{KisinIntModel}, \cite[Section 4.4]{KimMadapusiIntModel}.

In summary, we have the following diagram of schemes over $\Z_{(p)}$:
\[
\xymatrix{
\mathscr{S}_{\widetilde{\mathrm{K}}} \ar[d]^-{} \ar[r]^-{} &
\mathscr{A} \\
\mathscr{S}_{\widetilde{\mathrm{K}}_0}. &
}
\]

Let $\mathcal{A}_{\mathscr{S}_{\widetilde{\mathrm{K}}}} \to \mathscr{S}_{\widetilde{\mathrm{K}}}$ be the abelian scheme corresponding to the morphism $\mathscr{S}_{\widetilde{\mathrm{K}}} \to \mathscr{A}$.
The abelian scheme $\mathcal{A}_{\mathscr{S}_{\widetilde{\mathrm{K}}}}$ is called the \textit{Kuga-Satake abelian scheme}.
We often drop the subscript ${\mathscr{S}_{\widetilde{\mathrm{K}}}}$ in the notation.
For every $\mathscr{S}_{\widetilde{\mathrm{K}}}$-scheme $S$, 
we denote the pullback of $\mathcal{A}_{\mathscr{S}_{\widetilde{\mathrm{K}}}}$ to $S$ by $\mathcal{A}_{S}$.

\begin{rem}
\label{Remark:Hyperspecial}
If the discriminant of the quadratic space $L$ is divisible by $p$,
we do not yet have a satisfactory theory of
integral canonical models of the Shimura varieties
$\mathrm{Sh}_{\mathrm{K}}(\GSpin(L_{\Q}), X_{L})$ and
$\mathrm{Sh}_{\mathrm{K}_0}(\SO(L_{\Q}), X_{L})$
associated with $L$.
(The open compact subgroup
$\K_{0, p} \subset \SO(L_{\Q})(\Q_p)$
may not be hyperspecial.)
Following Madapusi Pera \cite{MadapusiIntModel, MadapusiTateConj, KimMadapusiIntModel},
we embed $L$ into $\widetilde{L}$ whose discriminant is
not divisible by $p$ as in Lemma \ref{LatticeEmbedding},
and use the integral canonical models
$\mathscr{S}_{\widetilde{\mathrm{K}}}$ and $\mathscr{S}_{\widetilde{\mathrm{K}}_0}$
associated with $\widetilde{L}$.
\end{rem}

\subsection{Local systems on Shimura varieties}\label{Subsection:Local Systems on Shimura Varieties}

In this subsection, we recall basic results on (complex analytic, $\ell$-adic, and $p$-adic) local systems on orthogonal Shimura varieties.
(For details, see \cite{MadapusiIntModel, KimMadapusiIntModel}.)

The $\widetilde{G}$-representation $\widetilde{L}_{\Z_{(p)}}$ and the $\widetilde{G}$-equivariant embedding 
\[
i \colon \widetilde{L}_{\Z_{(p)}} \hookrightarrow \End_{\Z_{(p)}}(H_{\Z_{(p)}})
\]
induce the following objects:
\begin{enumerate}
	\item A $\Q$-local system $\widetilde{\mathbb{V}}_{B}$ over the complex analytic space $\Sh^{\mathrm{an}}_{\widetilde{\K}, \C}$ and an embedding of $\Q$-local systems:
\[
i_{B} \colon \widetilde{\mathbb{V}}_{B} \hookrightarrow \underline{\End}(H^{\vee}_B).
\]
Here $H_B$ is the relative first singular cohomology with coefficients in $\Q$ of $\mathcal{A}^{\mathrm{an}}_{\Sh_{\widetilde{\K}, \C}}$ over $\Sh^{\mathrm{an}}_{\widetilde{\K}, \C}$, and $H^{\vee}_B$ is its dual.
	\item An $\A^p_f$-local system $\widetilde{\mathbb{V}}^p$ over the integral canonical model $\mathscr{S}_{\widetilde{\mathrm{K}}}$ and an embedding of $\A^p_f$-local systems:
\[
i^p \colon \widetilde{\mathbb{V}}^p \hookrightarrow \underline{\End}(V^p\mathcal{A}).
\]
Here we put
\[
V^p\mathcal{A}:= (\plim[p \nmid n]\mathcal{A}[n])\otimes_{\Z}\Q
\] 
and consider it as an $\A^p_f$-local system over $\mathscr{S}_{\widetilde{\mathrm{K}}}$. 
	\item A $\Z_p$-local system $\widetilde{\mathbb{L}}_{p}$ over the Shimura variety $\mathscr{S}_{\widetilde{\mathrm{K}}, \Q}=\Sh_{\widetilde{\mathrm{K}}}$ and an embedding of $\Z_p$-local systems:

\[
i_p \colon \widetilde{\mathbb{L}}_{p} \hookrightarrow \underline{\End}(T_p\mathcal{A}_{\Sh_{\widetilde{\mathrm{K}}}}).
\]
Here $T_p\mathcal{A}_{\Sh_{\widetilde{\mathrm{K}}}}$ is the $p$-adic Tate module of $\mathcal{A}_{\Sh_{\widetilde{\mathrm{K}}}}$ over $\Sh_{\widetilde{\mathrm{K}}}$.
\end{enumerate}

\subsection{Hodge tensors}\label{Subsection:Hodge tensors}

Recall that we fix tensors
$\{ s_{\alpha}\} $ of $H^{\otimes}_{\Z_{(p)}}$ defining the closed embedding $\widetilde{G} \hookrightarrow \GL(H_{\Z_{(p)}})$ over $\Z_{(p)}$; see Section \ref{Subsection: Representations of general spin groups and Hodge tensors}.
The tensors $\{ s_{\alpha}\}$ of $H^{\otimes}_{\Z_{(p)}}$ give rise to global sections $\{ s_{\alpha, B} \}$ of $H^{\otimes}_B$, global sections $\{s^p_{\alpha} \}$ of $(V^p\mathcal{A})^{\otimes}$,
and global sections $\{s_{\alpha, p}\}$ of $(T_{p}\mathcal{A})^{\otimes}$.

We recall properties of these tensors. (See \cite[(1.3.6)]{KisinModp}, \cite[Proposition 4.10]{KimMadapusiIntModel} for details.)
\begin{enumerate}
	\item 
Let $k$ be a field of characteristic $0$.
For every $x \in  \mathscr{S}_{\widetilde{\mathrm{K}}, \Q}(k)$ and a geometric point $\overline{x} \in \mathscr{S}_{\widetilde{\mathrm{K}}, \Q}(\overline{k})$ above $x$, the stalk $\widetilde{\mathbb{L}}_{p, \overline{x}}$ at $\overline{x}$ is equipped with an even perfect bilinear form $( \ , \ )$ over $\Z_p$.
The bilinear form is $\Gal(\overline{k}/k)$-invariant, i.e.\ we have
\[
(gy_1, gy_2)=(y_1, y_2) 
\]
for every $y_1, y_2 \in \widetilde{\mathbb{L}}_{p, \overline{x}}$ and every $g \in \Gal(\overline{k}/k)$.

We identify $T_p(\mathcal{A}_{\overline{x}})$ with
$H^1_{\mathrm{\acute{e}t}}(\mathcal{A}_{\overline{x}},\Z_p)^{\vee}$.
The $\Gal(\overline{k}/k)$-module $\widetilde{\mathbb{L}}_{p, \overline{x}}$ and the homomorphism $i_{p, \overline{x}}$ are characterized by the property that there is an isomorphism of $\Z_p$-modules
\[
H_{\Z_p} \cong H^1_{\mathrm{\acute{e}t}}(\mathcal{A}_{\overline{x}},\Z_p)^{\vee}
\]
which carries $\{ s_{\alpha}\}$ to $\{ s_{\alpha, p, \overline{x}} \}$ and induces the commutative diagram
\[
\xymatrix{
\widetilde{L}_{\Z_p} \ar[d]_-{\cong} \ar[r]^-{i} & \End_{\Z_p}(H_{\Z_p}) \ar[d]^-{\cong} \\
\widetilde{\mathbb{L}}_{p, \overline{x}} \ar[r]^-{i_{p, \overline{x}}} & \End_{\Z_p}(H^1_{\mathrm{\acute{e}t}}(\mathcal{A}_{\overline{x}},\Z_p)^{\vee}),
}
\]
where $\widetilde{L}_{\Z_p} \cong \widetilde{\mathbb{L}}_{p, \overline{x}}$ is an isometry over $\Z_p$.
(We will often drop the subscript $\overline{x}$ of $i_{p, \overline{x}}$.)

\item Let $k$ a field of characteristic $0$ or $p$. 
For every $x \in \mathscr{S}_{\widetilde{\mathrm{K}}}(k)$ and a geometric point $\overline{x} \in \mathscr{S}_{\widetilde{\mathrm{K}}}(\overline{k})$ above $x$, 
the stalk $\widetilde{\mathbb{V}}^p_{\overline{x}}$ at $\overline{x}$ has a bilinear form $( \ , \ )$ over $\A^p_f$ satisfying the same property as above with $\Z_p$ replaced by $\A^p_f$.

\item For every $x \in \mathscr{S}_{\widetilde{\mathrm{K}}}(\C)$, 
the stalk $\widetilde{\mathbb{V}}_{B, x}$ at $x$ has a bilinear form $( \ , \ )$ over $\Q$ satisfying the same property as above with $\Z_p$ replaced by $\Q$.
\end{enumerate}

\subsection{$F$-crystals and Breuil-Kisin modules}\label{Subsection: $F$-crystals and Breuil-Kisin modules}

In this subsection, let $k$ be a perfect field and $x \in \mathscr{S}_{\widetilde{\mathrm{K}}, \F_p}(k)$.
We recall basic results on $F$-crystals attached to $x \in \mathscr{S}_{\widetilde{\mathrm{K}}, \F_p}(k)$.
(For details, see \cite{KisinIntModel, KisinModp, MadapusiIntModel, KimMadapusiIntModel}.)

The $\widetilde{G}$-representation $\widetilde{L}_{\Z_{(p)}}$ and the $\widetilde{G}$-equivariant embedding
\[
i \colon \widetilde{L}_{\Z_{(p)}} \hookrightarrow \End_{\Z_{(p)}}(H_{\Z_{(p)}})
\]
induces a free $W$-module $\widetilde{L}_{\cris, x}$ of finite rank and an embedding
\[
i_{\cris} \colon 
\widetilde{L}_{\cris, x} \hookrightarrow \End_{W}(H^1_{\cris}(\mathcal{A}_{x}/W)^{\vee}).
\]
The $W[1/p]$-vector space 
$\widetilde{L}_{\cris, x}[1/p]$
has the structure of an $F$-isocrystal.
Namely, it is equipped with a Frobenius automorphism $\varphi$.
The embedding $i_{\cris}$ induces an embedding of $F$-isocrystals
\[
i_{\cris}[1/p] \colon \widetilde{L}_{\cris, x}
[1/p] \hookrightarrow \End_{W[1/p]}(H^1_{\cris}(\mathcal{A}_{x}/W[1/p])^{\vee}).
\]
There is an even perfect bilinear form on $\widetilde{L}_{\cris, x}$.
When $p$ is inverted, we have
\[ (\varphi(y_1), \varphi(y_2))=\sigma(y_1, y_2) \]
for every $y_1, y_2 \in \widetilde{L}_{\cris, x}[1/p]$.

We recall the definition of $\widetilde{L}_{\cris, x}$ and $i_{\cris}$ when $k$ is a finite field $\F_q$ or $\overline{\F}_q$.
Take a lift $\widetilde{x} \in \mathscr{S}_{\widetilde{\mathrm{K}}}(\O_K)$ of $x$, where $K$ is a finite totally ramified extension of $W[1/p]$ and $\O_K$ is its valuation ring.
Note that such a lift exists by \cite[Proposition 2.3.5]{KisinIntModel}, \cite[Proposition 4.6]{KimMadapusiIntModel}.
Let $\eta= \Spec K$ be the generic point of $\Spec \O_K$ and $\overline{\eta}$ a geometric point above $\eta$.

We fix a uniformizer $\varpi$ of $K$, and a system
$\{ \varpi^{1/{p^n}} \}_{n \geq 0}$ of $p^n$-th roots of $\varpi$ such that
$(\varpi^{1/{p^{n+1}}})^p=\varpi^{1/{p^n}}$.
Let 
\[
\mathfrak{M}(H^1_{\mathrm{\acute{e}t}}(\mathcal{A}_{\overline{\eta}},\Z_p))
\]
be the Breuil-Kisin module (over $\O_K$
with respect to $\{ \varpi^{1/{p^n}} \}_{n \geq 0}$)
associated with the $\Gal(\overline{K}/K)$-stable $\Z_p$-lattice  $H^1_{\mathrm{\acute{e}t}}(\mathcal{A}_{\overline{\eta}},\Z_p)$ in a crystalline representation.
(For the definition of Breuil-Kisin modules, see Section \ref{Subsection:Breuil-Kisin modules and crystalline Galois representations}. See also \cite[(1.2)]{KisinIntModel}.)

Since $\M(-)$ is a tensor functor, the tensors $\{ s_{\alpha, p, \overline{\eta}} \}$ of $(T_p\mathcal{A}_{\overline{\eta}})^{\otimes}$ give rise to Frobenius invariant tensors $\{ \mathfrak{M}(s_{\alpha, p, \overline{\eta}}) \}$ of $\mathfrak{M}(H^1_{\mathrm{\acute{e}t}}(\mathcal{A}_{\overline{\eta}},\Z_p))^{\otimes}$ and Frobenius invariant tensors $\{ \Mcris(s_{\alpha, p, \overline{\eta}}) \}$ of $\Mcris(H^1_{\mathrm{\acute{e}t}}(\mathcal{A}_{\overline{\eta}},\Z_p))^{\otimes}$.
By \cite[Corollary 1.4.3]{KisinIntModel}, \cite[Theorem 2.12]{KimMadapusiIntModel}, we have a canonical isomorphism
\[
\Mcris(H^1_{\mathrm{\acute{e}t}}(\mathcal{A}_{\overline{\eta}},\Z_p)) \cong H^1_{\cris}(\mathcal{A}_{x}/W).
\]
(See also Theorem \ref{Theorem:p-adic Hodge theorem}.)
Using this canonical isomorphism, the tensors $\{ \Mcris(s_{\alpha, p, \overline{\eta}}) \}$ give rise to Frobenius invariant tensors $\{ s_{\alpha, \cris, x} \}$ of $H^1_{\cris}(\mathcal{A}_{x}/W)^{\otimes}$.
The tensors $\{ s_{\alpha, \cris, x} \}$ do not depend on the choice of the lift $\widetilde{x} \in \mathscr{S}_{\widetilde{\mathrm{K}}}(\O_K)$ of $x$; see \cite[Proposition 1.3.9]{KisinModp}.
(See also the last paragraph in the proof of \cite[Proposition 4.6]{KimMadapusiIntModel}.)

Kisin proved the following results.

\begin{prop}[Kisin {\cite[Proposition 1.3.4, Corollary 1.3.5]{KisinIntModel}}]\label{Proposition:LcrisBKmodule}
There is an isomorphism of $\mathfrak{S}$-modules
\[
H^1_{\mathrm{\acute{e}t}}(\mathcal{A}_{\overline{\eta}},\Z_p)^{\vee}
\otimes_{\Z_p}\mathfrak{S} \cong \mathfrak{M}(H^1_{\mathrm{\acute{e}t}}(\mathcal{A}_{\overline{\eta}},\Z_p)^{\vee})
\]
which carries $\{ s_{\alpha, p, \overline{\eta}} \}$ to $\{ \mathfrak{M}(s_{\alpha, p, \overline{\eta}}) \}$ and induces the following commutative diagram
\[
\xymatrix{
\widetilde{\mathbb{L}}_{p, \overline{\eta}}\otimes_{\Z_p}\mathfrak{S} \ar[d]_-{\cong} \ar[r]^-{i_p} & \End_{\mathfrak{S}}(H^1_{\mathrm{\acute{e}t}}(\mathcal{A}_{\overline{\eta}},\Z_p)^{\vee}
\otimes_{\Z_p}\mathfrak{S}) \ar[d]^-{\cong} \\
 \mathfrak{M}(\widetilde{\mathbb{L}}_{p, \overline{\eta}}) \ar[r]^-{\mathfrak{M}(i_{p})} & \End_{\mathfrak{S}}(\mathfrak{M}(H^1_{\mathrm{\acute{e}t}}(\mathcal{A}_{\overline{\eta}},\Z_p)^{\vee})),
 }
\]
where $\widetilde{\mathbb{L}}_{p, \overline{\eta}}\otimes_{\Z_p}\mathfrak{S} \cong\mathfrak{M}(\widetilde{\mathbb{L}}_{p, \overline{\eta}})$ is an isometry over $\mathfrak{S}$.
\end{prop}
\begin{proof}
The assertion follows from \cite[Proposition 1.3.4, Corollary 1.3.5]{KisinIntModel}.
Precisely, the statements in \cite[Proposition 1.3.4, Corollary 1.3.5]{KisinIntModel} do not claim the existence of the commutative diagram above, but the same argument works.
\end{proof}

We define
\[
\widetilde{L}_{\cris, x}:= \Mcris(\widetilde{\mathbb{L}}_{p, \overline{\eta}}) \quad \mathrm{and} \quad i_{\cris}:=\Mcris(i_p).
\]
The even perfect bilinear form on
$
\widetilde{\mathbb{L}}_{p, \overline{\eta}}
$
induces an even perfect bilinear form on $\widetilde{L}_{\cris, x}$.
By (1) in Section \ref{Subsection:Hodge tensors} and Proposition \ref{Proposition:LcrisBKmodule}, there is an isomorphism of $W$-modules
\[
	H_{W} \cong H^1_{\cris}(\mathcal{A}_{x}/W)^{\vee}
	\]
	which carries $\{ s_{\alpha}\}$ to $\{ s_{\alpha, \cris, x} \}$ and induces the following commutative diagram:
	\[
\xymatrix{
\widetilde{L}_{W} \ar[d]_-{\cong} \ar[r]^-{i} & \End_{W}(H_{W}) \ar[d]^-{\cong} \\
  \widetilde{L}_{\cris, x}\ar[r]^-{i_{\cris}} & \End_{W}(H^1_{\cris}(\mathcal{A}_{x}/W)^{\vee}),
}
\]
where $\widetilde{L}_{W} \cong \widetilde{L}_{\cris, x}$ is an isometry over $W$.
It follows that $\widetilde{L}_{\cris, x}$ and $i_{\cris}$ do not depend on the choice of the lift $\widetilde{x} \in \mathscr{S}_{\widetilde{\mathrm{K}}}(\O_K)$ of $x$.

\subsection{$\Lambda$-structures for integral canonical models}
\label{Subsection:LambdaStructures}

Recall that we have fixed an embedding of quadratic spaces $L \hookrightarrow \widetilde{L}$.
Let
$\Lambda := L^{\perp} \subset \widetilde{L}$
be the orthogonal complement of $L$ in $\widetilde{L}$, and $\iota \colon \Lambda \hookrightarrow \widetilde{L}$ the natural inclusion.

We recall the definition of $\Lambda$-structures from 
\cite{MadapusiIntModel}.

\begin{defn}[see {\cite[Definition 6.11]{MadapusiIntModel}}]
\label{Definition:Lambda structure}
 A \textit{$\Lambda$-structure} for an $\mathscr{S}_{\widetilde{\mathrm{K}}}$-scheme $S$
 is a homomorphism of $\Z_{(p)}$-modules
\[
\iota_S \colon \Lambda_{\Z_{(p)}} \to \End_{S}(\mathcal{A}_S)_{\Z_{(p)}}
\]
satisfying the following properties:
\begin{enumerate}
	\item For any algebraically closed field $\overline{K}$ of characteristic $0$ and $x \in  S(\overline{K})$,
	there is an isometry $\iota_{p} \colon \Lambda_{\Z_p} \to \widetilde{\mathbb{L}}_{p, x}$ over $\Z_p$ such that the homomorphism induced by $\iota_S$
\[\Lambda_{\Z_p} \to \End_{\Z_p}(T_p\mathcal{A}_{x})
\]
factors as
\[
\Lambda_{\Z_p} \overset{\iota_{p}}{\longrightarrow} \widetilde{\mathbb{L}}_{p, x} \overset{i_{p}}{\longrightarrow} \End_{\Z_p}(T_p\mathcal{A}_{x}).
\]
\item For any perfect field $k$ of characteristic $p$ and $x \in S(k)$, there is an isometry
$\iota_{\cris} \colon \Lambda_{W} \to \widetilde{L}_{\cris, x}$
over $W$
such that the homomorphism induced by $\iota_S$
\[
\Lambda_{W} \to \End_{W}(H^1_{\cris}(\mathcal{A}_{x}/W)^{\vee})
\]
factors as
\[
\xymatrix{\Lambda_{W} \ar[r]^-{\iota_{\cris}} & \widetilde{L}_{\cris, x}  \ar[r]^-{i_{\cris}} & \End_{W}(H^1_{\cris}(\mathcal{A}_{x}/W)^{\vee})}.
\]
\end{enumerate}
\end{defn}

It turns out that these conditions imply the following:
\begin{enumerate}
	\item By \cite[Corollary 5.22]{MadapusiIntModel}, for every geometric point $x \to S$, there is an isometry
	$\iota^{p} \colon \Lambda_{\A^p_f} \to \widetilde{\mathbb{V}}^p_{x}$
	over $\A^p_f$ such that
the homomorphism induced by $\iota_S$
\[
\Lambda_{\A^p_f} \to \End_{\A^p_f}(V^p(\mathcal{A}_{x}))
\]
factors as
\[
\Lambda_{\A^p_f} \overset{\iota^{p}}{\longrightarrow} \widetilde{\mathbb{V}}^p_{x} \overset{i^p}{\longrightarrow} \End_{\A^p_f}(V^p(\mathcal{A}_{x})).
\]
\item For every $\C$-valued point $x \in S(\C)$,
there is an isometry
$\iota_B \colon \Lambda_{\Q} \to \widetilde{\mathbb{V}}_{B, x}$
over $\Q$ such that
the homomorphism induced by $\iota_S$ 
\[
\Lambda_{\Q} \to \End_{\Q}(H^1_B(\mathcal{A}_x, \Q)^{\vee})
\]
factors as
\[
\Lambda_{\Q} \overset{\iota_B}{\longrightarrow} \widetilde{\mathbb{V}}_{B, x} \overset{i_B}{\longrightarrow} \End_{\Q}(H^1_B(\mathcal{A}_x, \Q)^{\vee}).
\]
\end{enumerate}

We recall the definition of a $\mathrm{K}^p$-level structure.
Here $\K^p \subset \GSpin(L_{\Q})(\A^p_f)$ is an open compact subgroup whose image under the homomorphism 
\[
\GSpin(L_{\Q})(\A^p_f) \to \SO(L_{\Q})(\A^p_f)
\]
is $\K^p_0$.

Let $S$ be an $\mathscr{S}_{\widetilde{\mathrm{K}}}$-scheme.
For simplicity, we assume $S$ is locally noetherian and connected.
Let $\epsilon'$ be the corresponding $\mathrm{K}'^p$-level structure on $\mathcal{A}_S$;
as in \cite[(3.2.4)]{KisinIntModel}, for a geometric point $s \to S$, the $\mathrm{K}'^p$-level structure $\epsilon'$ is induced by a $\widetilde{\mathrm{K}}^p$-orbit $\widetilde{\epsilon}$ of an isometry
$H_{\A^p_f}\cong V^p(\mathcal{A}_s)$
over $\A^p_f$ which carries $\{ s_{\alpha}\}$ to $\{ s^p_{\alpha, s}\}$ and carries $\widetilde{L}_{\A^p_f}$ to $\widetilde{\mathbb{V}}^p_{s}$ such that
 $\widetilde{\epsilon}$ is $\pi_1(S, s)$-invariant.
Here $\pi_1(S, s)$ denotes the \'etale fundamental group of $S$, and the Tate module $V^p(\mathcal{A}_s)$ over $\A^p_f$ has a natural action of $\pi_1(S, s)$.

\begin{defn}\label{Definition:Kp level structure}
Let $S$ be a locally noetherian connected scheme over $\mathscr{S}_{\widetilde{\mathrm{K}}}$.
Let $s \to S$ be a geometric point.
A \textit{$\mathrm{K}^p$-level structure} on $(S, \iota_S)$ is a $\pi_1(S, s)$-invariant $\mathrm{K}^p$-orbit $\epsilon_{\iota}$ of an isometry of $\A^p_f$-modules
\[
H_{\A^p_f}\cong V^p(\mathcal{A}_s)
\]
satisfying the following properties:
\begin{enumerate}
	\item It carries $\{ s_{\alpha}\}$ to $\{ s^p_{\alpha, s}\}$.
	\item The following diagram is commutative:
\[
\xymatrix{\Lambda_{\A^p_f}  \ar^-{\iota}[r] \ar_-{\iota^p}[rd] &
\widetilde{L}_{\A^p_f} \ar[d]^-{\cong} \ar[r]^-{i} & \End_{\A^p_f}(H_{\A^p_f}) \ar[d]^-{\cong} \\
& \widetilde{\mathbb{V}}^p_{s} \ar[r]^-{i^p} & \End_{\A^p_f}(V^p(\mathcal{A}_s)).
}
\]
\item The $\mathrm{K}^p$-orbit $\epsilon_{\iota}$ induces the $\widetilde{\mathrm{K}}^p$-orbit $\widetilde{\epsilon}$ on $S$.
\end{enumerate}
\end{defn}

\begin{defn}\label{Definition:Lambda structure functor}
Let $Z_{\K^p}(\Lambda)$ be the functor on $\mathscr{S}_{\widetilde{\mathrm{K}}}$-schemes defined by
\[
Z_{\K^p}(\Lambda)(S):= \{\, (\iota_S, \epsilon_{\iota}) \mid \iota_S \, \, \text{is a}\,\, \Lambda\text{-structure and}\, \, \epsilon_{\iota}\,\, \text{is a}\,\, \mathrm{K}^p\text{-level structure on}\, (S, \iota_S) \,\}
\]
for an $\mathscr{S}_{\widetilde{\mathrm{K}}}$-scheme $S$.
\end{defn}

Similarly, we can define a $\Lambda$-structure $\iota_{0, S}$ for an $\mathscr{S}_{\widetilde{\mathrm{K}}_0}$-scheme $S$, a $\mathrm{K}^p_0$-level structure on $(S, \iota_{0, S})$ and a functor $Z_{\mathrm{K}^p_0}(\Lambda)$.
(See \cite[Definition 6.11]{MadapusiIntModel} for details.)

The following result was proved by Madapusi Pera.

\begin{prop}[Madapusi Pera {\cite{MadapusiIntModel}}]\label{Proposition:LambdaFiniteEtale}
The functor $Z_{\K^p}(\Lambda)$ (resp.\ $Z_{\K^p_0}(\Lambda)$) is represented by a scheme which is finite and unramified over $\mathscr{S}_{\widetilde{\mathrm{K}}}$ (resp.\ $\mathscr{S}_{\widetilde{\mathrm{K}}_0}$).
Moreover there is a natural morphism 
\[
Z_{\K^p}(\Lambda) \to Z_{\K^p_0}(\Lambda),
\]
which is finite and \'etale.
\end{prop}
\begin{proof}
See \cite[Proposition 6.13]{MadapusiIntModel}. 
\end{proof}

\section{Moduli spaces of $K3$ surfaces and the Kuga-Satake morphism}\label{Section:Moduli spaces of K3 surfaces and the Kuga-Satake morphism} 

In this section, we recall definitions and basic properties of the moduli space of $K3$ surfaces and the level structure.
Then we recall definitions and basic results on the Kuga-Satake morphism
over $\Z_{(p)}$ introduced by Madapusi Pera
\cite{MadapusiTateConj, KimMadapusiIntModel}.

\subsection{Moduli spaces of $K3$ surfaces}\label{Subsection:Moduli spaces of K3 surfaces}

Recall that we say $f\colon \mathscr{X} \to S$ is
a \textit{$K3$ surface over $S$} if $S$ is a scheme, $\mathscr{X}$ is an algebraic space, and $f$ is a proper smooth morphism whose geometric fibers are $K3$ surfaces. 

A \textit{quasi-polarization} of $f\colon \mathscr{X} \to S$ is a section
${\xi}\in\underline{\Pic}(\mathscr{X}/S)(S)$
of the relative Picard functor whose fiber $\xi(s)$ at every geometric point $s \to S$ is a line bundle on the $K3$ surface $\mathscr{X}_s$ which is nef and big. We say ${\xi}\in\underline{\Pic}(\mathscr{X}/S)(S)$ is \textit{primitive} if, for every geometric point $s\to{S}$, the cokernel of the inclusion 
$\langle\xi(s)\rangle\hookrightarrow\Pic(\mathscr{X}_s)$
is torsion-free.
We say $\xi$ has degree $2d$ if,  for every geometric point $s \to {S}$, we have
$({\xi(s)},{\xi(s)})=2d$,
where $(\ ,\ )$ denotes the intersection pairing on $\mathscr{X}_s$.
We say a pair 
$(f\colon \mathscr{X} \to S, \xi)$
is a \textit{quasi-polarized $K3$ surface over $S$ of degree $2d$} if
$f\colon \mathscr{X} \to S$
is a $K3$ surface over $S$ and $\xi\in\underline{\Pic}(\mathscr{X}/S)(S)$ is a primitive quasi-polarization of degree $2d$.

  Let $M_{2d}$ be the moduli functor that sends a $\Z$-scheme $S$ to the groupoid consists of quasi-polarized $K3$ surfaces over $S$ of degree $2d$.
 The moduli functor $M_{2d}$ is a Deligne-Mumford stack of finite type over $\Z$; see \cite[Theorem 4.3.4]{Rizov06} and \cite[Proposition 2.1]{Maulik}.

We put
$M_{2d, \Z_{(p)}}:=M_{2d}\otimes_{\Z}\Z_{(p)}.$
Let $S$ be an $M_{2d, \Z_{(p)}}$-scheme.
For the quasi-polarized $K3$ surface 
$(f\colon \mathscr{X} \to S, \xi)$ associated with the structure morphism $S \to M_{2d, \Z_{(p)}}$ and a prime number $\ell \neq p$,
we equipped $R^2f_* \Z_{\ell}(1)$ with the \textit{negative} of the cup product pairing. 
Let
\[
P^2f_* \Z_{\ell}(1) := \ch_\ell(\xi)^{\perp} \subset R^2f_* \Z_{\ell}(1) 
\]
be the orthogonal complement of the $\ell$-adic Chern class $\ch_\ell(\xi) \in R^2f_* \Z_{\ell}(1)(S)$ with respect to the pairing.
We set
\[
P^2f_*{\widehat{\Z}^p}(1) := \prod_{\ell \neq p}P^2f_* \Z_{\ell}(1).
\]
The stalk of $P^2f_* \Z_{\ell}(1)$ (resp.\  $P^2f_*{\widehat{\Z}^p}(1)$) at a geometric point $s \to S$ will be denoted by $P^2_{\et}(\mathscr{X}_{s}, {\Z}_{\ell}(1))$ (resp.\ 
$P^2_{\et}(\mathscr{X}_{s}, \widehat{\Z}^p(1))$).

Let $M^{\mathrm{sm}}_{2d, \Z_{(p)}}$ be the smooth locus  of $M_{2d, \Z_{(p)}}$ over $\Z_{(p)}$.
Madapusi Pera constructed a twofold finite \'etale cover $\widetilde{M}^{\mathrm{sm}}_{2d, \Z_{(p)}} \to M^{\mathrm{sm}}_{2d, \Z_{(p)}}$ parameterizing orientations of $P^2f_*{\widehat{\Z}^p}(1)$, which satisfies the following property.
For every morphism $S \to \widetilde{M}^{\mathrm{sm}}_{2d, \Z_{(p)}}$, there is a natural isometry of ${\widehat{\Z}^p}$-local systems on $S$
\[
\nu \colon \underline{\det L \otimes_\Z \widehat{\Z}^p} \cong \det P^2f_*{\widehat{\Z}^p}(1)
\]
such that, for every $s \in S(\C)$, the isometry $\nu$ restricts to an isometry over $\Z$
\[
\nu_s \colon \det L \cong \det P^2_B(\mathscr{X}_s, \Z(1))
\]
under the canonical isomorphism
\[
P^2_B(\mathscr{X}_s, \Z(1)) \otimes_\Z \widehat{\Z}^p \cong P^2_{\et}(\mathscr{X}_{s}, \widehat{\Z}^p(1)),
\]
where we put $P^2_B(\mathscr{X}_s, \Z(1)):= \ch_B(\xi_s)^{\perp} \subset H^2_B(\mathscr{X}_s, \Z(1))$.
See \cite[Section 5]{MadapusiTateConj} for details.

For an open compact subgroup $\K^p_0 \subset \SO(L_{\Q})(\A^p_f)$ as in Section \ref{Subsection:Orthogonal Shimura Varieties over Q}, we recall the notion of (oriented) $\K^p_0$-level structures from \cite[Section 3]{MadapusiTateConj}.
For simplicity, we only consider the case $S$ is a locally noetherian connected $\widetilde{M}^{\mathrm{sm}}_{2d, \Z_{(p)}}$-scheme.
Let $s \to S$ be a geometric point and $\pi_1(S, s)$ the \'etale fundamental group of $S$.
 A \textit{$\K^p_0$-level structure} on $(f\colon \mathscr{X} \to S, \xi)$ is a $\pi_1(S, s)$-invariant $\K^p_0$-orbit $\eta$ of an isometry over $\widehat{\Z}^p$
 \[
 \Lambda_{K3}\otimes_{\Z}\widehat{\Z}^p \cong H^2_{\mathrm{\acute{e}t}}(\mathscr{X}_{s}, \widehat{\Z}^p(1))
 \]
which carries $e-df$ to $\mathrm{ch}_{\widehat{\Z}^p}(\xi(s))$ such that the induced isometry 
 \[
 \det L \otimes_\Z \widehat{\Z}^p \cong \det P^2_{\et}(\mathscr{X}_{s}, \widehat{\Z}^p(1))
 \]
coincides with $\nu_s$.
Here the \'etale cohomology $H^2_{\mathrm{\acute{e}t}}(\mathscr{X}_{s},\widehat{\Z}^p(1))$ has a natural action of $\pi_1(S, s)$, and $\Lambda_{K3}\otimes_{\Z}\widehat{\Z}^p$ has a natural action of $\K^p_0$; see \cite[Lemma 2.6]{MadapusiIntModel}.

Let $M^{\mathrm{sm}}_{2d, \K^p_0, \Z_{(p)}}$ be the moduli functor over $\widetilde{M}^{\mathrm{sm}}_{2d, \Z_{(p)}}$ which sends an $\widetilde{M}^{\mathrm{sm}}_{2d, \Z_{(p)}}$-scheme $S$ to the set of (oriented) $\K^p_0$-level structures on the quasi-polarized $K3$ surface $(f\colon \mathscr{X} \to S, \xi)$.

The following result is well-known.
\begin{prop}\label{Proposition:K3moduliLevel}
If $\K^p_0 \subset \SO(L_{\Q})(\A^p_f)$ is small enough, $M^{\mathrm{sm}}_{2d, \K^p_0, \Z_{(p)}}$ is an algebraic space over $\Z_{(p)}$ which is finite, \'etale, and faithfully flat over $M^{\mathrm{sm}}_{2d, \Z_{(p)}}$.
\end{prop}
\begin{proof}
This result was essentially proved by Rizov \cite[Theorem 6.2.2]{Rizov06}, Maulik \cite[Proposition 2.8]{Maulik} and Madapusi Pera \cite[Proposition 3.11]{MadapusiTateConj}.
Note that their proofs work in every characteristic $p$, including $p=2$.
	Their proofs rely on the injectivity of the map
	\[ \Aut(X) \to \GL(H^{2}_{\mathrm{\acute{e}t}}( X,\Q_{\ell} )) \]
	for every $\ell \neq p$, where $X$ is a $K3$ surface over an algebraically closed field of characteristic $p > 0$.
	The injectivity was proved by Ogus when $p>2$; see \cite[Corollary 2.5]{Ogus}.
	(Precisely, Ogus proved it for the crystalline cohomology.
The injectivity for the $\ell$-adic cohomology follows from Ogus' results; see \cite[Proposition 3.4.2]{Rizov06}.)
	Recently, Keum proved that the injectivity holds also when $p=2$; see \cite[Theorem 1.4]{Keum}.
\end{proof}

We will assume an open compact subgroup $\K^p_0 \subset \SO(L_{\Q})(\A^p_f)$ as in Section \ref{Subsection:Orthogonal Shimura Varieties over Q} is small enough so that Proposition \ref{Proposition:K3moduliLevel} can be applied. 

\subsection{The Kuga-Satake morphism}\label{Subsection:Kuga-Satake morphism}

Rizov and Madapusi Pera defined the following \'etale morphism over $\Q$:
\[
M^{\mathrm{sm}}_{2d, \K^p_0, \Q} \to \mathrm{Sh}_{\mathrm{K}_0}(\SO(L_{\Q}), X_{L}).
\]
It is called the \textit{Kuga-Satake morphism} over $\Q$.
(See \cite[Theorem 3.9.1]{Rizov10}, \cite[Corollary 5.4]{MadapusiTateConj} for details.)

Since $M^{\mathrm{sm}}_{2d, \K^p_0, \Z_{(p)}}$ is smooth over $\Z_{(p)}$,
the composite of the following morphisms over $\Q$
\[
M^{\mathrm{sm}}_{2d, \K^p_0, \Q} \to \mathrm{Sh}_{\mathrm{K}_0}(\SO(L_{\Q}), X_{L}) \to \mathrm{Sh}_{\widetilde{\mathrm{K}}_0}
\]
extends to a morphism over $\Z_{(p)}$
\[
M^{\mathrm{sm}}_{2d, \K^p_0, \Z_{(p)}} \to \mathscr{S}_{\widetilde{\mathrm{K}}_0}
\]
by the extension properties of the integral canonical model $\mathscr{S}_{\widetilde{\mathrm{K}}_0}$; see \cite[(2.3.7)]{KisinIntModel}, \cite[Proposition 5.7]{MadapusiTateConj}. 

The following results are proved by Madapusi Pera.

\begin{prop}[Madapusi Pera]
\label{Proposition:KugaSatakeMap}
There is a natural \'etale $\mathscr{S}_{\widetilde{\mathrm{K}}_0}$-morphism
\[
\KS \colon M^{\mathrm{sm}}_{2d, \K^p_0, \Z_{(p)}}\to Z_{\K^p_0}(\Lambda).
\]
\end{prop}

\begin{proof}
The morphism over $\Q$
\[
\mathrm{Sh}_{\mathrm{K}_0}(\SO(L_{\Q}), X_{L}) \to \mathrm{Sh}_{\widetilde{\mathrm{K}}_0}
\]
factors through the generic fiber $Z_{\K^p_0}(\Lambda)_{\Q}$ of $Z_{\K^p_0}(\Lambda)$; see \cite[6.15]{MadapusiIntModel} for details.
Hence we have a morphism over $\Q$
\[
M^{\mathrm{sm}}_{2d, \K^p_0, \Q} \to Z_{\K^p_0}(\Lambda)_\Q.
\]
This morphism extends to a $\mathscr{S}_{\widetilde{\mathrm{K}}_0}$-morphism
\[
\KS \colon M^{\mathrm{sm}}_{2d, \K^p_0, \Z_{(p)}}\to Z_{\K^p_0}(\Lambda)
\]
by \cite[Chapter I, Proposition 2.7]{Faltings-Chai}.

For the \'etaleness of the morphism $\KS$, see the proof of \cite[Proposition A.12]{KimMadapusiIntModel}.
See also Remark \ref{Remark: Kim-Madapusi Pera gap}.
\end{proof}

One usually calls the morphism $\KS$ in
Proposition \ref{Proposition:KugaSatakeMap} a period map, and calls a morphism from the moduli space of $K3$ surfaces to the moduli space of abelian varieties a Kuga-Satake morphism.
In this paper, following Madapusi Pera, we call the morphism $\KS$ a $\textit{Kuga-Satake morphism}$ for convenience.

Summarizing the above, we have the following commutative diagram of algebraic spaces over $\Z_{(p)}$:
\[
\xymatrix{ & Z_{\K^p}(\Lambda) \ar[r]^-{} \ar[d]^-{} & \mathscr{S}_{\widetilde{\mathrm{K}}} \ar[d]^-{} \\
M^{\mathrm{sm}}_{2d, \K^p_0, \Z_{(p)}} \ar[r]^-{\KS} & Z_{\K^p_0}(\Lambda) \ar[r]^-{} \ar[r]^-{} & \mathscr{S}_{\widetilde{\mathrm{K}}_0}.
}
\]
Here
$\mathscr{S}_{\widetilde{\mathrm{K}}}$ (resp.\ 
$\mathscr{S}_{\widetilde{\mathrm{K}}_0}$)
is the integral canonical model of the Shimura variety associated with
$\widetilde{G}=\GSpin(\widetilde{L}_{\Z_{(p)}})$
(resp.\ $\widetilde{G}_0=\SO(\widetilde{L}_{\Z_{(p)}})$),
and $Z_{\K^p}(\Lambda)$
(resp.\ $Z_{\K^p_0}(\Lambda)$)
is the scheme over
$\mathscr{S}_{\widetilde{\mathrm{K}}}$
(resp.\ $\mathscr{S}_{\widetilde{\mathrm{K}}_0}$)
as in Definition \ref{Definition:Lambda structure functor} and Proposition \ref{Proposition:LambdaFiniteEtale}.

\section{$F$-crystals on Shimura varieties}\label{Section:$F$-crystals on Shimura varieties}

In this section, let $k$ be a perfect field of characteristic $p>0$.
We shall study the $W$-module
$\widetilde{L}_{\cris, s}$
associated with a $k$-valued point $s \in \mathscr{S}_{\widetilde{\mathrm{K}}}(k)$ of the integral canonical model $\mathscr{S}_{\widetilde{\mathrm{K}}}$.

We use the same notation as in Section \ref{Subsection: Preliminaries}.
In particular, $W := W(k)$, $\mathfrak{S}:=W[[u]]$,
$K$ is a finite totally ramified extension of $W[1/p]$,
and $\overline{K}$ is an algebraic closure of $K$.
We fix a uniformizer $\varpi$ of $K$,
and a system $\{ \varpi^{1/{p^n}} \}_{n \geq 0}$ of
$p^n$-th roots of $\varpi$ such that
$(\varpi^{1/{p^{n+1}}})^p=\varpi^{1/{p^n}}$.
Let $E(u) \in W[u]$ be the (monic) Eisenstein polynomial of $\varpi$.

\subsection{The primitive cohomology of quasi-polarized $K3$ surfaces}\label{Subsection:The primitive cohomology of quasi-polarized $K3$ surfaces}

Let $t \in M^{\mathrm{sm}}_{2d, \K^p_0, \Z_{(p)}}(\O_K)$ be an $\O_K$-valued point,
and $(\mathscr{Y}, \xi)$ a quasi-polarized $K3$ surface over $\O_K$ of degree $2d$
associated with $t$.
We assume $\xi$ is a line bundle whose restriction to every geometric fiber is big and nef.

We denote the orthogonal complement with respect to the cup product of the first Chern class of $\xi$ in the de Rham cohomology and the $p$-adic \'etale cohomology by
\begin{align*}
P^2_{\dR}(\mathscr{Y}/\O_K) &:= \mathrm{ch}_{\dR}(\xi)^{\perp} \subset H^2_{\dR}(\mathscr{Y}/\O_K), \\
P^2_{\mathrm{\acute{e}t}}(\mathscr{Y}_{\overline{K}}, \Z_p(1)) &:= \mathrm{ch}_{p}(\xi)^{\perp} \subset H^2_{\mathrm{\acute{e}t}}(\mathscr{Y}_{\overline{K}}, \Z_p(1)).
\end{align*}

We shall recall some well-known properties of the de Rham cohomology of $K3$ surfaces; see \cite{Deligne:LiftingK3} for details.
The de Rham cohomology $H^2_{\dR}(\mathscr{Y}/\O_K)$ admits the Hodge filtration $\Fil^i$:
\[
0=\Fil^3_{\mathrm{Hdg}} \subset \Fil^2_{\mathrm{Hdg}} \subset \Fil^1_{\mathrm{Hdg}} \subset \Fil^0_{\mathrm{Hdg}} = H^2_{\dR}(\mathscr{Y}/\O_K).
\]
The graded piece
$\Gr^i_{\mathrm{Hdg}} := \Fil^i_{\mathrm{Hdg}}/\Fil^{i+1}_{\mathrm{Hdg}}$
is a free $\O_K$-module for every $i$.
The Hodge filtration $\Fil^i_{\mathrm{Hdg}}$ on $H^2_{\dR}(\mathscr{Y}/\O_K)$
is mapped onto the Hodge filtration $\overline{\Fil}^i_{\mathrm{Hdg}}$ on
$H^2_{\dR}(\mathscr{Y}_{k}/k)$
under the canonical isomorphism
\[
H^2_{\dR}(\mathscr{Y}/\O_K)\otimes_{\O_K} k \cong H^2_{\dR}(\mathscr{Y}_{k}/k).
\]
Moreover, $\Fil^2_{\mathrm{Hdg}}$ is the orthogonal complement of
$\Fil^1_{\mathrm{Hdg}}$
with respect to the cup product on $H^2_{\dR}(\mathscr{Y}/\O_K)$.
In other words, the cup product induces an isomorphism of $\O_K$-modules
\[
\Gr^2_{\mathrm{Hdg}} \cong (\Gr^0_{\mathrm{Hdg}})^{\vee}.
\]

We define a decreasing filtration $\Fil^i_{\text{pr}}$ on $P^2_{\dR}(\mathscr{Y}/\O_K)$
by
\[
\Fil^i_{\text{pr}} :=\Fil^i_{\mathrm{Hdg}} \cap P^2_{\dR}(\mathscr{Y}/\O_K).
\]
We put
$\Gr^i_{\text{pr}} := \Fil^i_{\text{pr}}/\Fil^{i+1}_{\text{pr}}$.

\begin{lem}\label{Lemma:Duality filtration}
The natural homomorphisms
$
\Gr^0_{\text{pr}} \to \Gr^0_{\mathrm{Hdg}}
$
and
$
\Gr^2_{\text{pr}} \to \Gr^2_{\mathrm{Hdg}}
$
are isomorphisms.
In particular, the cup product induces an isomorphism of $\O_K$-modules:
\[
\Gr^2_{\text{pr}} \cong (\Gr^0_{\text{pr}})^{\vee}.
\]
\end{lem}
\begin{proof}
	Since the first Chern class
	$\mathrm{ch}_{\dR}(\xi)$ is contained in $\Fil^1_{\mathrm{Hdg}}$, we have $\Fil^2_{\mathrm{Hdg}} \subset P^2_{\dR}(\mathscr{Y}/\O_K)$
	and
	$\Gr^2_{\text{pr}} \cong \Gr^2_{\mathrm{Hdg}}$.
	
	We shall show 
	$\Gr^0_{\text{pr}} \cong \Gr^0_{\mathrm{Hdg}}$.
	By the definition of $\Fil^i_{\text{pr}}$,
	the homomorphism $\Gr^0_{\text{pr}} \to \Gr^0_{\mathrm{Hdg}}$ is injective.
	To prove the surjectivity, it suffices to prove that it is surjective after taking the reduction modulo $p$.
	Since both $\Gr^0_{\text{pr}}$ and $\Gr^0_{\mathrm{Hdg}}$ are free $\O_K$-modules of rank $1$,
	it suffices to show the following homomorphism is injective
	\[
	\Gr^0_{\text{pr}}\otimes_{\O_K} k \to \Gr^0_{\mathrm{Hdg}} \otimes_{\O_K} k.
	\]
	
	Since $\xi \in \underline{\Pic}(\mathscr{Y}/\O_K)(\O_K)$ is primitive, 
	the cokernel of the homomorphism
	\[
	\langle \mathrm{ch}_{\dR}(\xi) \rangle \hookrightarrow H^2_{\dR}(\mathscr{Y}/\O_K)
	\]
	is $p$-torsion-free by \cite[Corollary 1.4]{Ogus}, and the following composite
	\[
	H^2_{\dR}(\mathscr{Y}/\O_K) \cong H^2_{\dR}(\mathscr{Y}/\O_K)^{\vee} \to \langle \mathrm{ch}_{\dR}(\xi) \rangle^{\vee}
	\]
	is surjective, where the first isomorphism is obtained by the Poincar\'e duality and the second homomorphism is the restriction map.
	Thus we have the following split exact sequence:
	\[
	0 \to P^2_{\dR}(\mathscr{Y}/\O_K) \to H^2_{\dR}(\mathscr{Y}/\O_K) \to \langle \mathrm{ch}_{\dR}(\xi) \rangle^{\vee} \to 0.
	\]
	It follows that 
	\[
	P^2_{\dR}(\mathscr{Y}/\O_K) \otimes_{\O_K} k
	\cong P^2_{\dR}(\mathscr{Y}_{k}/k) := \mathrm{ch}_{\dR}(\xi_{k})^{\perp} \subset H^2_{\dR}(\mathscr{Y}_{k}/k).
	\]
	
	Hence it is enough to show the injection
	\[
	\Fil^1_{\text{pr}}\otimes_{\O_K} k
	\hookrightarrow \overline{\Fil}^1_{\mathrm{Hdg}} \cap P^2_{\dR}(\mathscr{Y}_{k}/k)
	\]
	is an isomorphism of $k$-vector spaces.
	It suffices to show the intersection
	$\overline{\Fil}^1_{\mathrm{Hdg}} \cap P^2_{\dR}(\mathscr{Y}_{k}/k)$ is a $20$-dimensional $k$-vector space.
	If the dimension was greater than or equal to $21$,
	we would have 
	\[
	\overline{\Fil}^1_{\mathrm{Hdg}} \cap P^2_{\dR}(\mathscr{Y}_{k}/k) = \overline{\Fil}^1_{\mathrm{Hdg}}
	\]
	and
	\[
	\mathrm{ch}_{\dR}(\xi_{k}) \in (\overline{\Fil}^1_{\mathrm{Hdg}})^{\perp} = \overline{\Fil}^2_{\mathrm{Hdg}}.
	\]
	By the proof of \cite[Proposition 2.2.1]{Ogus},
	the versal deformation space of $(\mathscr{Y}_{k}, \xi_{k})$ is not regular.
	This contradicts that $(\mathscr{Y}_k, \xi_{k})$ lies on the smooth locus $M^{\mathrm{sm}}_{2d, \K^p_0, \Z_{(p)}}$.
\end{proof}

We consider the Breuil-Kisin module 
\[
\mathfrak{M}(P^2_{\mathrm{\acute{e}t}}(\mathscr{Y}_{\overline{K}}, \Z_p))
\]
(over $\O_K$ with respect to $\{ \varpi^{1/{p^n}} \}_{n \geq 0}$)
associated with   
$P^2_{\mathrm{\acute{e}t}}(\mathscr{Y}_{\overline{K}}, \Z_p)$.

We define a decreasing filtration $\Fil^i(\mathfrak{M}_{\dR}(P^2_{\mathrm{\acute{e}t}}(\mathscr{Y}_{\overline{K}}, \Z_p)))$ on
$\mathfrak{M}_{\dR}(P^2_{\mathrm{\acute{e}t}}(\mathscr{Y}_{\overline{K}}, \Z_p))$
as in Section \ref{Subsection:Breuil-Kisin modules and crystalline Galois representations}.
Let
\begin{align*}
	c_{\dR, \mathscr{Y}_K}\colon D_{\dR}(H^2_{\mathrm{\acute{e}t}}(\mathscr{Y}_{\overline{K}}, \Z_p)[1/p]) &\cong H^2_{\dR}(\mathscr{Y}/\O_K)[1/p],\\
	c_{\cris, \mathscr{Y}} \colon D_{\cris}(H^2_{\mathrm{\acute{e}t}}(\mathscr{Y}_{\overline{K}}, \Z_p)[1/p]) &\cong H^2_{\cris}(\mathscr{Y}_{k}/W)[1/p]
\end{align*}
be the isomorphisms in Section \ref{Subsection:The integral p-adic Hodge theory}.
These isomorphisms are compatible with Chern classes,  cup products, and trace maps; see Section \ref{Subsection:Scholze's de Rham comparison map} and
Section \ref{Subsection:Crystalline comparison map of Bhatt-Morrow-Scholze}.
Hence we have canonical isomorphisms
of filtered $W[1/p]$-modules
\[
\mathfrak{M}_{\dR}(P^2_{\mathrm{\acute{e}t}}(\mathscr{Y}_{\overline{K}}, \Z_p))[1/p] \cong D_{\dR}(P^2_{\mathrm{\acute{e}t}}(\mathscr{Y}_{\overline{K}}, \Z_p)[1/p])
 \cong P^2_{\dR}(\mathscr{Y}/\O_K)[1/p].
\]
We also have canonical isomorphisms
\[
\Mcris(P^2_{\mathrm{\acute{e}t}}(\mathscr{Y}_{\overline{K}}, \Z_p))[1/p] \cong D_{\cris}(P^2_{\mathrm{\acute{e}t}}(\mathscr{Y}_{\overline{K}}, \Z_p)[1/p])
 \cong P^2_{\cris}(\mathscr{Y}_{k}/W)[1/p],
\]
which are compatible with Frobenius endomorphisms.

\begin{lem}\label{Lemma:Primitive Breuil-Kisin module (de Rham)}
	The canonical isomorphism
	\[
	\mathfrak{M}_{\dR}(P^2_{\mathrm{\acute{e}t}}(\mathscr{Y}_{\overline{K}}, \Z_p))[1/p] \cong P^2_{\dR}(\mathscr{Y}/\O_K)[1/p]
	\]
	maps
		$\mathfrak{M}_{\dR}(P^2_{\mathrm{\acute{e}t}}(\mathscr{Y}_{\overline{K}}, \Z_p))$
		onto
	$P^2_{\dR}(\mathscr{Y}/\O_K)$.
	Moreover, it maps $\Fil^i(\mathfrak{M}_{\dR}(P^2_{\mathrm{\acute{e}t}}(\mathscr{Y}_{\overline{K}}, \Z_p)))$  onto
        $\Fil^i_{\text{pr}}$ for every $i \in \Z$.
\end{lem}
\begin{proof}
We shall show the first assertion.
The cup product on $H^2_{\mathrm{\acute{e}t}}(\mathscr{Y}_{\overline{K}}, \Z_p)$ induces a perfect pairing on the Breuil-Kisin module
$
\mathfrak{M}(H^2_{\mathrm{\acute{e}t}}(\mathscr{Y}_{\overline{K}}, \Z_p)).
$
Since the kernel of a homomorphism of free $\mathfrak{S}$-modules of finite rank is free, the orthogonal complement
\[
\mathfrak{M}(\langle \mathrm{ch}_{p}(\xi) \rangle(-1))^{\perp}
\subset \mathfrak{M}(H^2_{\mathrm{\acute{e}t}}(\mathscr{Y}_{\overline{K}}, \Z_p))
\]
is a free $\mathfrak{S}$-module of finite rank.
Hence it is a Breuil-Kisin module (of height $\leq 2$).
By the characterization of $\mathfrak{M}(P^2_{\mathrm{\acute{e}t}}(\mathscr{Y}_{\overline{K}}, \Z_p))$ in \cite[Theorem 4.4]{BMS},
we have 
\[
\mathfrak{M}(P^2_{\mathrm{\acute{e}t}}(\mathscr{Y}_{\overline{K}}, \Z_p)) \cong \mathfrak{M}(\langle \mathrm{ch}_{p}(\xi) \rangle(-1))^{\perp}.
\]
By Theorem \ref{Theorem:p-adic Hodge theorem}, we have
\[
\MdR(H^2_{\mathrm{\acute{e}t}}(\mathscr{Y}_{\overline{K}}, \Z_p)) \cong H^2_{\dR}(\mathscr{Y}/\O_K).
\]
Under this isomorphism, the homomorphism
\[
\MdR(\langle \mathrm{ch}_{p}(\xi) \rangle(-1))
\hookrightarrow \MdR(H^2_{\mathrm{\acute{e}t}}(\mathscr{Y}_{\overline{K}}, \Z_p))
\]
is identified with the inclusion
$
\langle \mathrm{ch}_{\dR}(\xi) \rangle \hookrightarrow H^2_{\dR}(\mathscr{Y}/\O_K)
$
by Proposition \ref{Proposition:de Rham comparison and Chern class}.
Thus it is enough to show the cokernel of the injective homomorphism
\[
j \colon \varphi^{*}\mathfrak{M}(\langle \mathrm{ch}_{p}(\xi) \rangle(-1))
\hookrightarrow \varphi^{*}\mathfrak{M}(H^2_{\mathrm{\acute{e}t}}(\mathscr{Y}_{\overline{K}}, \Z_p))
\]
is a free $\mathfrak{S}$-module.
As above, we have
\[
\mathfrak{M}(\langle \mathrm{ch}_{p}(\xi) \rangle(-1))\cong \mathfrak{M}(P^2_{\mathrm{\acute{e}t}}(\mathscr{Y}_{\overline{K}}, \Z_p))^{\perp}.
\]
Hence the cokernel of $j$ is a torsion-free $\mathfrak{S}$-module.
Therefore it suffices to show the cokernel of
$
\langle \mathrm{ch}_{\dR}(\xi) \rangle \hookrightarrow H^2_{\dR}(\mathscr{Y}/\O_K)
$
is $p$-torsion-free.
This follows from \cite[Corollary 1.4]{Ogus}.

The second assertion follows from the fact that $\Gr^{i}(\mathfrak{M}_{\dR}(P^2_{\mathrm{\acute{e}t}}(\mathscr{Y}_{\overline{K}}, \Z_p)))$ and $\Gr^i_{\text{pr}}$ are free $\O_K$-modules.
\end{proof}

\begin{lem}\label{Lemma:Primitive Breuil-Kisin module (crystalline)}
The canonical isomorphism
	\[
	\mathfrak{M}_{\cris}(P^2_{\mathrm{\acute{e}t}}(\mathscr{Y}_{\overline{K}}, \Z_p))[1/p] \cong P^2_{\cris}(\mathscr{Y}_{k}/W)[1/p]
	\]
maps $\mathfrak{M}_{\cris}(P^2_{\mathrm{\acute{e}t}}(\mathscr{Y}_{\overline{K}}, \Z_p))$ onto
$P^2_{\cris}(\mathscr{Y}_{k}/W)$.
\end{lem}
\begin{proof}
	As in the proof of Lemma \ref{Lemma:Primitive Breuil-Kisin module (de Rham)}, the assertion follows from the fact
	that the cokernel of $j$ is a free $\mathfrak{S}$-module
	by using Theorem \ref{Theorem:p-adic Hodge theorem} and Corollary \ref{Corollary:Crystalline comparison and Chern class}.
\end{proof}

\subsection{$F$-crystals on Shimura varieties and the cohomology of $K3$ surfaces}\label{Subsection:$F$-crystals on Shimura varieties and the cohomology of K3 surfaces}
In this subsection, we assume $k$ is a finite field $\F_q$ or $\overline{\F}_q$.
We shall compare the $W$-module
$\widetilde{L}_{\cris, s}$
associated with a $k$-valued point $s \in \mathscr{S}_{\widetilde{\mathrm{K}}}(k)$ of the integral canonical model $\mathscr{S}_{\widetilde{\mathrm{K}}}$
with the crystalline cohomology of $K3$ surfaces.

We consider the following situation.
\begin{enumerate}
	\item 
Let $s \in M^{\mathrm{sm}}_{2d, \K^p_0, \Z_{(p)}}(k)$ be a $k$-valued point on the smooth locus and $(X,\mathscr{L})$ a quasi-polarized $K3$ surface over $k$ of degree $2d$ associated with $s \in M^{\mathrm{sm}}_{2d, \K^p_0, \Z_{(p)}}(k)$.

\item 
The image of $s$ under the Kuga-Satake morphism $\KS$ is denoted by
the same notation $s \in Z_{\K^p_0}(\Lambda)(k)$.
After replacing $k$ by a finite extension of it,
there is a $k$-valued point of $Z_{\K^p}(\Lambda)(k)$ mapped to $s$.
We fix such a point, and denote it also by $s \in Z_{\K^p}(\Lambda)(k)$.
Let $\mathcal{A}_s$ be the Kuga-Satake abelian variety over $k$
associated with the point 
$s \in Z_{\mathrm{K}^p}(k)$.

\item
We take an $\O_K$-valued point
$t \in M^{\mathrm{sm}}_{2d, \K^p_0, \Z_{(p)}}(\O_K)$
which is a lift of
$s \in M^{\mathrm{sm}}_{2d, \K^p_0, \Z_{(p)}}(k)$.
The morphism
\[
Z_{\mathrm{K}^p}(\Lambda) \to Z_{\mathrm{K}^p_0}(\Lambda)
\]
is \'etale by Proposition \ref{Proposition:LambdaFiniteEtale}.
Hence the image of the $\O_K$-valued point $t$
under the Kuga-Satake morphism lifts a unique $\O_K$-valued point on
$Z_{\mathrm{K}^p}(\Lambda)$
which is a lift of
$s \in Z_{\mathrm{K}^p}(\Lambda)$.
We also denote it by $t \in Z_{\K^p}(\Lambda)(\O_K)$.

\item
Let $(\mathscr{Y}, \xi)$ be a quasi-polarized $K3$ surface over $\O_K$
of degree $2d$ associated with
$t \in M^{\mathrm{sm}}_{2d, \K^p_0, \Z_{(p)}}(\O_K)$.
We assume $\xi$ is a line bundle whose restriction to every geometric fiber is big and nef.

\item
Let $\overline{t}$ be a geometric point of
$Z_{\K^p}(\Lambda)$ above the generic point of
$t \in Z_{\K^p}(\Lambda)(\O_K)$.
\end{enumerate}

Then the
$\Gal(\overline{K}/K)$-stable $\Z_p$-lattice in a crystalline representation
\[
\widetilde{L}'_{p}:= \widetilde{\mathbb{L}}_{p, \overline{t}}
\]
satisfies the following properties:
	\begin{enumerate}
	\item $\widetilde{L}'_{p}$ admits an even perfect bilinear form $( \ , \ )$ over $\Z_p$ which is compatible with the action of $\Gal(\overline{K}/K)$. 
	\item There is a $\Gal(\overline{K}/K)$-equivariant homomorphism
	$\iota_p \colon \Lambda_{\Z_p} \to \widetilde{L}'_{p}$
	preserving the bilinear forms.
	\item There is a $\Gal(\overline{K}/K)$-equivariant isometry over $\Z_p$
     \[
    P^2_{\mathrm{\acute{e}t}}(\mathscr{Y}_{\overline{K}},\Z_p(1)) \cong \iota_p(\Lambda_{\Z_p})^{\perp}.
    \]
(See \cite[Proposition 5.6]{MadapusiTateConj}.) 
	\item There is an isometry over $W$
	\[ \Mcris(\widetilde{L}'_{p}) \cong \widetilde{L}_{\cris, s} \]
	inducing an isomorphism of $F$-isocrystals after inverting $p$.
	(See Section \ref{Subsection: $F$-crystals and Breuil-Kisin modules}.)
	\end{enumerate}

We use the following notation on twists of $\varphi$-modules.
Let $(N, \varphi)$ be a pair of  a free $W$-module of finite rank and
a $\sigma$-linear map $\varphi$ on $N[1/p]$.
We denote the pair $(N, p^{-i}\varphi)$ by $N(i)$.

We put 
\[
\widetilde{L}_{\cris}:=\widetilde{L}_{\cris, s}.
\]
The Frobenius automorphism of $\widetilde{L}_{\cris}(-1)[1/p]$ maps the $W$-module $\widetilde{L}_{\cris}(-1)$ into itself.
Therefore $\widetilde{L}_{\cris}(-1)$ is an $F$-crystal.

\begin{prop}\label{Proposition:CrisOrthogonal}
	We have an isomorphism of $F$-crystals
	\[
	P^2_{\cris}(X/W) \cong \iota_{\cris}(\Lambda_{W})^{\perp}(-1)\subset \widetilde{L}_{\cris}(-1).
	\]
\end{prop}
\begin{proof}
	This proposition was proved in \cite[Corollary 5.14]{MadapusiTateConj} when $p \neq 2$.
	Here we give a proof using the integral comparison theorem
	of Bhatt-Morrow-Scholze \cite{BMS}.
	(Our proof works for any $p$, including $p=2$.)
	As in the proof of Lemma \ref{Lemma:Primitive Breuil-Kisin module (de Rham)}, we have
	\[
	\M(P^2_{\mathrm{\acute{e}t}}(\mathscr{Y}_{\overline{K}},\Z_p)) \cong \M(\iota_p(\Lambda_{\Z_p})(-1))^{\perp}.
	\]
	Since 
	\[
	\Mcris(P^2_{\mathrm{\acute{e}t}}(\mathscr{Y}_{\overline{K}},\Z_p))\cong P^2_{\cris}(X/W)
	\]
	by Lemma \ref{Lemma:Primitive Breuil-Kisin module (crystalline)}, 
	it suffices to show the cokernel of the homomorphism
	\[
\varphi^{*}\M(\iota_p(\Lambda_{\Z_p})(-1)) \to \varphi^{*}\M(\widetilde{L}'_{p}(-1))
	\]
	is a free $\mathfrak{S}$-modules in the proof of Lemma \ref{Lemma:Primitive Breuil-Kisin module (de Rham)}.
	This follows from Proposition \ref{Proposition:Flatness of cokernel} below.
\end{proof}

\begin{prop}\label{Proposition:Flatness of cokernel}
The cokernel of the homomorphism
\[
\M(\iota_p(\Lambda_{\Z_p})^{\perp}) \to \M(\widetilde{L}'_{p})
\]
is a free $\mathfrak{S}$-module.
\end{prop}
\begin{proof}
Let $\M'$ be the cokernel of  the homomorphism
\[
\M(\iota_p(\Lambda_{\Z_p})(-1)^{\perp}) \to
\M(\widetilde{L}'_{p}(-1)).
\]
As in the proof of Lemma \ref{Lemma:Primitive Breuil-Kisin module (de Rham)},
we have
\[
\M(\iota_p(\Lambda_{\Z_p})(-1)^{\perp}) \cong \M(\iota_p(\Lambda_{\Z_p})(-1))^{\perp}.
\]
Hence the cokernel $\M'$ is a torsion-free $\mathfrak{S}$-module.
Note that $\M(\widetilde{L}'_{p}(-1))$ is an effective Breuil-Kisin module of height $ \leq 2$ and it induces a Frobenius
\[
1\otimes \varphi \colon \varphi^{*}\M' \to \M'.
\]

It is enough to show the Frobenius of $\M'$ satisfies the following properties:
\begin{enumerate}
	\item For every $x \in \varphi^{*}\M'$, the image $(1\otimes \varphi)(x)$ is divisible by $E(u)$.
	\item The cokernel of $1\otimes \varphi \colon \varphi^{*}\M' \to \M'$ is killed by $E(u)$.
\end{enumerate}
In fact, if $(1), (2)$ are proved, the homomorphism $1/E(u)(1\otimes \varphi)$ makes $\M'$ a torsion-free $\varphi$-module of height $0$ in the sense of \cite{Liu12}.
Then, we see that $\M'$ is a free $\mathfrak{S}$-module by \cite[Lemma 2.18 (2)]{Liu12}.

We shall prove $(1)$.
For simplicity, we put
\begin{align*}
	\M &:=\M(\iota_p(\Lambda_{\Z_p})(-1)^{\perp}) \\
\widetilde{\M} &:=\M(\widetilde{L}'_{p}(-1))\\
\MdR &:= \MdR(\iota_p(\Lambda_{\Z_p})(-1)^{\perp})\\
\widetilde{\M}_{\dR} &:=\MdR(\widetilde{L}'_{p}(-1)).
\end{align*}

The perfect bilinear form on $\widetilde{\M}_{\dR}$
induces
$
\widetilde{\M}_{\dR} \cong \widetilde{\M}_{\dR}^{\vee}.
$
We shall show that this isomorphism induces
\[
\Gr^2(\widetilde{\M}_{\dR}) \cong \Gr^0(\widetilde{\M}_{\dR})^{\vee}.
\]
Since the cokernel of $\Gr^2(\widetilde{\M}_{\dR}) \hookrightarrow \widetilde{\M}_{\dR}$
(resp.\ $\Gr^0(\widetilde{\M}_{\dR})^{\vee} \hookrightarrow \widetilde{\M}_{\dR}^{\vee}$) is $p$-torsion-free, it suffices to show that we have the desired isomorphism after inverting $p$.
This follows the isomorphism
\begin{align*}
    \widetilde{L}'_{p}(-1)[1/p] &\cong \iota_p(\Lambda_{\Z_p})(-1)[1/p] \oplus \iota_p(\Lambda_{\Z_p})(-1)^{\perp}[1/p] \\
    & \cong \iota_p(\Lambda_{\Z_p})(-1)[1/p] \oplus P^2_{\mathrm{\acute{e}t}}(\mathscr{Y}_{\overline{K}},\Z_p)[1/p],
\end{align*}
where we use $P^2_{\mathrm{\acute{e}t}}(\mathscr{Y}_{\overline{K}},\Z_p) \cong \iota_p(\Lambda_{\Z_p})(-1)^{\perp}$.

Since $P^2_{\mathrm{\acute{e}t}}(\mathscr{Y}_{\overline{K}},\Z_p) \cong \iota_p(\Lambda_{\Z_p})(-1)^{\perp}$, the bilinear form on $\MdR$ induces 
\[
\Gr^2(\MdR) \cong \Gr^0(\MdR)^{\vee}
\]
by Lemma \ref{Lemma:Duality filtration} and Lemma \ref{Lemma:Primitive Breuil-Kisin module (de Rham)}.

The homomorphism
\[
\MdR \to \widetilde{\M}_{\dR}
\]
induces the following commutative diagram:
\[
\xymatrix{ \Gr^2(\MdR) \ar[r]^-{\cong}\ar[d]^-{}  & \Gr^0(\MdR)^{\vee}  \\
\Gr^2(\widetilde{\M}_{\dR}) \ar[r]^-{\cong} & \Gr^0(\widetilde{\M}_{\dR})^{\vee}. \ar[u]^-{}
}
\]
Hence $\Gr^0(\widetilde{\M}_{\dR})^{\vee} \to \Gr^0(\MdR)^{\vee}$ is surjective.
Since both $\Gr^0(\MdR)^{\vee}$ and $\Gr^0(\widetilde{\M}_{\dR})^{\vee}$
are free $\O_K$-modules of rank $1$, we have
$
\Gr^0(\MdR)^{\vee} \cong \Gr^0(\widetilde{\M}_{\dR})^{\vee}$, and hence
\[
\Gr^0(\MdR) \cong \Gr^0(\widetilde{\M}_{\dR}).
\]
We also have
\[
\Gr^2(\MdR) \cong \Gr^2(\widetilde{\M}_{\dR}).
\]
By $\Gr^0(\MdR) \cong \Gr^0(\widetilde{\M}_{\dR})$ and Lemma \ref{Lemma: 1-st Filtration}, we have
\[
\Fil^1(\varphi^{*}\widetilde{\M})+\varphi^{*}\M=\varphi^{*}\widetilde{\M}.
\]
The assertion $(1)$ follows from this equality.

We shall prove $(2)$.
By $\Gr^2(\MdR) \cong \Gr^2(\widetilde{\M}_{\dR})$ and Lemma \ref{Lemma: 1-st Filtration}, we have
\[
\Fil^2(\varphi^{*}\widetilde{\M})\subset \Fil^1(\varphi^{*}\M) + E(u)\varphi^{*}\widetilde{\M}.
\]
We take an element $x \in \widetilde{\M}$.
Since $\widetilde{\M}$ is of height $\leq 2$, there is an element $y \in \varphi^{*}\widetilde{\M}$ such that
\[ (1\otimes \varphi)(y)=E(u)^2x. \]
By the above inclusion, there is an element
$z \in \Fil^1(\varphi^{*}\M)$ and
$w \in \varphi^{*}\widetilde{\M}$
such that
\[
y = z + E(u)w.
\]
We put $(1\otimes \varphi)(z):= E(u)z'$ for $z' \in \M$.
Then we have
\[
E(u)x=z' + (1\otimes \varphi)(w).
\]
This shows $E(u)x$ is zero in the cokernel of $\varphi^{*}\M' \to \M'$ and the proof of the assertion $(2)$ is complete.

The proof of Proposition \ref{Proposition:Flatness of cokernel} is complete.
\end{proof}

\subsection{Formal Brauer groups}\label{Subsection: Formal Brauer groups}

We consider the situation as in Section \ref{Subsection:$F$-crystals on Shimura varieties and the cohomology of K3 surfaces}.

Let $\widehat{\Br}:=\widehat{\Br}(X)$ be the \textit{formal Brauer group} associated with the $K3$ surface $X$.
Recall that $\widehat{\Br}$ is a one-dimensional smooth formal group scheme
pro-representing the functor
\[ 
 \Phi^2_{X} \colon {\mathrm{Art}_k} \to \text{(Abelian  groups)}
\]
defined by
\[
R \mapsto \Ker({H^2_{\rm{\acute{e}t}}(X_{R},\mathbb{G}_m)}\rightarrow{H^2_{\rm{\acute{e}t}}(X,\mathbb{G}_m)}),
\]
where $\mathrm{Art}_k$ is the category of local artinian $k$-algebras with residue field $k$, and (Abelian  groups) is the category of abelian groups;
see \cite[Chapter II, Corollary 2.12]{Artin-Mazur}.
(For basic properties of the formal Brauer group,
see also \cite[Section 6]{Liedtke}.)
The \textit{height} $h$ of the  $K3$ surface $X$ is defined to be the height of $\widehat{\Br}$.
We have
$1\leq h \leq10$ or $h = \infty$.

There is a natural equivalence from the category of one-dimensional smooth formal group schemes of finite height over $k$ to the category of one-dimensional connected $p$-divisible groups over $k$.
If the height of $X$ is finite, we identify the formal Brauer group $\widehat{\Br}$ with the corresponding connected $p$-divisible group over $k$, and let $\widehat{\Br}^{*}$ be the Cartier dual of $\widehat{\Br}$.

For a crystal $\mathscr{E}$ over $\mathrm{CRIS}(k/\Z_p)$, we will denote simply by the same letter $\mathscr{E}$ the value
$\mathscr{E}_{W \twoheadrightarrow k}$ in ($\Spec k \hookrightarrow \Spec W$).
By \cite[(5.3.3.1)]{BBM}, we have a canonical perfect bilinear form
\[
\mathbb{D}(\widehat{\Br}^{*}) \times \mathbb{D}(\widehat{\Br})(-1) \to W(-2).
\]

\begin{prop}\label{Proposition:CrysDecomposition}
	Assume the height of $X$ is finite.
	The following assertions hold:
	\begin{enumerate}
	\item There is an isomorphism of $F$-crystals
	\[
	\widetilde{L}_{\cris}(-1) \cong \mathbb{D}(\widehat{\Br}^{*})\oplus \mathbb{D}(D)(-1)\oplus\mathbb{D}(\widehat{\Br})(-1),
	\]
	where $D$ is an \'etale $p$-divisible group over $k$.
\item Under this isomorphism, the bilinear form on $\widetilde{L}_{\cris}(-1)$ is the direct sum of a perfect bilinear form
\[
\mathbb{D}(D)(-1) \times \mathbb{D}(D)(-1) \to W(-2)
\]
and the canonical perfect bilinear form
\[
\mathbb{D}(\widehat{\Br}^{*}) \times \mathbb{D}(\widehat{\Br})(-1) \to W(-2).
\]
\end{enumerate}
	\end{prop}

\begin{proof}
The breaking points of the Newton polygon of $\widetilde{L}_{\cris}(-1)$ lie on the Hodge polygon of it; see Lemma \ref{Lemma:HodgePolygon} and its proof below.

By the Hodge-Newton decomposition \cite[Theorem 1.6.1]{Katz},
there is a decomposition as an $F$-crystal over $W$
\[
\widetilde{L}_{\cris}(-1) \cong \widetilde{L}_{1-1/h} \oplus \widetilde{L}_{1} \oplus \widetilde{L}_{1+1/h},
\]
where
$\widetilde{L}_{\lambda}$ is an $F$-crystal over $W$
has a single slope $\lambda$ for each $\lambda \in \{\, 1-1/h,\, 1,\, 1+1/h \,\}$. 

Via this decomposition, the bilinear form $( \ , \ )$ is the direct sum of 
a perfect bilinear form
\[
\widetilde{L}_{1} \times \widetilde{L}_{1} \to W(-2)
\]
and a perfect bilinear form
\[
\widetilde{L}_{1-1/h} \times \widetilde{L}_{1+1/h} \to W(-2).
\]
Similarly, we have a decomposition 
\[
P^2_{\cris}(X/W) \cong L_{1-1/h} \oplus L_{1} \oplus L_{1+1/h}.
\]

By Proposition \ref{Proposition:CrisOrthogonal}, 
we have $P^2_{\cris}(X/W)=\iota_{\cris}(\Lambda_{W})(-1)^{\perp}$.
Since $\iota_{\cris}(\Lambda_{W})(-1)$ is contained in $\widetilde{L}_{1}$,
we have
\[
L_{1-1/h}=\widetilde{L}_{1-1/h} \quad \mathrm{and} \quad L_{1+1/h}=\widetilde{L}_{1+1/h}.
\]
We have a natural isomorphism of $F$-crystals over $W$
\[
\mathbb{D}(\widehat{\Br}^{*}) \cong L_{1-1/h}.
\]
(See \cite[Proposition 7]{Taelman} for example.
See also Remark \ref{Remark:Cartier module}.)
Using the perfect bilinear form
\[
\widetilde{L}_{1-1/h} \times \widetilde{L}_{1+1/h} \to W(-2),
\]
we identify $L_{1+1/h}$ with $L^{\vee}_{1-1/h}(-2)\cong \mathbb{D}(\widehat{\Br})(-1)$.

Since we have $p\widetilde{L}_{1} = \varphi(\widetilde{L}_{1})$,
there is an \'etale $p$-divisible group $D$ over $k$ such that $\mathbb{D}(D)(-1)\cong \widetilde{L}_{1}$.
\end{proof}

\begin{rem}\label{Remark:Cartier module}
We assume the height of $X$ is finite.
The proof of \cite[Proposition 7]{Taelman} shows that there is a natural isomorphism of $F$-crystals over $W$
\[
\mathrm{TC}(\widehat{\Br}) \cong L_{1-1/h}.
\]
Here $\mathrm{TC}(\widehat{\Br})$ is the Cartier-Dieudonn\'e module of typical curves of $\widehat{\Br}$; see \cite[I, Section 3]{Artin-Mazur} for example.
The $F$-crystal $\mathrm{TC}(\widehat{\Br})$
is naturally isomorphic to the $F$-crystal $\mathbb{D}(\widehat{\Br}^{*})$
by \cite[(5.8)]{Breen}, \cite[Th\'eor\`eme 4.2.14, (5.3.3.1)]{BBM}.
\end{rem}

The following result is used in the proof of Proposition \ref{Proposition:CrysDecomposition}.

\begin{lem}\label{Lemma:HodgePolygon}
We have an isomorphism of $W$-modules
\[
\widetilde{L}_{\cris}(-1)/\varphi(\widetilde{L}_{\cris}(-1)) \cong (W/p)^{\oplus 20} \oplus W/p^2.
\]
\end{lem}
\begin{proof}
	
We put
\[
M:=\Mcris(\widetilde{L}'_{p})(-1) \cong \widetilde{L}_{\cris}(-1).
\]
Since we have
\[
P^2_{\cris}(X/W)[1/p]\oplus \Lambda_{W}[1/p](-1) \cong M[1/p],
\]
the Newton polygon of $M$ has three slopes and multiplicities described as follows:
\begin{center}
\begin{tabular}{|c|c|c|c|} \hline
\ slope \ \    &   $1-1/h$  \ \  & \ $1$  \  \ & \ $1+1/h$ \ \  \\ 
\hline
 \  multiplicity \ \    &  \  $h$ \ \ & \ $22-2h$ \ \ & \ $h$ \ \  \\ \hline
  \end{tabular}
\end{center}

The Newton polygon of $M$ is above the Hodge polygon of $M$ and both polygons have the same initial point and the end point; see \cite[Theorem 1.4.1]{Katz}.
Therefore, it is enough to show the largest slope of the Hodge polygon of $M$ is $2$, and its multiplicity is less than or equal to $1$.
We put
\[
M':= \{\, x \in M \mid px \in \Im(\varphi) \,\}.
\]
Then it is enough to show there is a surjection $W/p \to M/M'$.

We put $\mathfrak{M}:=\mathfrak{M}(\widetilde{L}'_{p}(-1))$ and define
\[
\mathfrak{M}':= \{\, x \in \mathfrak{M} \mid E(u)x \in \Im(1\otimes \varphi \colon \varphi^{*}\mathfrak{M} \to \mathfrak{M}) \,\}.
\]
Since the Frobenius
$\varphi \colon \mathfrak{S} \to \mathfrak{S}$
is faithfully flat,
we have
\[
\varphi^{*}\mathfrak{M}' = \{\, x \in \varphi^{*}\mathfrak{M} \mid \varphi(E(u))x \in \Im(1\otimes \varphi \colon \varphi^{*}(\varphi^{*}\mathfrak{M}) \to \varphi^{*}\mathfrak{M}) \,\}.
\]
Via the isomorphism 
$
\varphi^{*}\mathfrak{M}\otimes_{\mathfrak{S}}W \cong M,
$
we have a surjection
\[
(\varphi^{*}\mathfrak{M}/\varphi^{*}\mathfrak{M}')\otimes_{\mathfrak{S}}W \to M/M'.
\]
We shall show $(\varphi^{*}\mathfrak{M}/\varphi^{*}\mathfrak{M}')\otimes_{\mathfrak{S}}W \cong W/p$.

Since we have a perfect bilinear form over $\mathfrak{S}$
\[
\mathfrak{M} \times \mathfrak{M} \to \mathfrak{S}(-2)
\]
which is compatible with the Frobenius endomorphisms, 
we have
\[
E(u)^2\mathfrak{M} \subset \Im(\varphi^{*}\mathfrak{M} \to \mathfrak{M}).
\] 
Hence $\mathfrak{M}/\mathfrak{M}'$ is killed by $E(u)$ and $\mathfrak{M}/\mathfrak{M}'$ is a finite $\mathfrak{S}/E(u)$-module.
Using the perfect bilinear form of $\mathfrak{M}$ again, we see that $\mathfrak{M}/\mathfrak{M}'$ is $p$-torsion-free.
Since $\mathfrak{S}/E(u)\cong \O_K$ is a discrete valuation ring, we see that $\mathfrak{M}/\mathfrak{M}'$ is a free $\mathfrak{S}/E(u)$-module of finite rank.

Since we have
\[
P^2_{\mathrm{\acute{e}t}}(\mathscr{Y}_{\overline{K}},\Q_p) \oplus \iota_p(\Lambda_{\Q_p})(-1) \cong \widetilde{L}'_{p}(-1)[1/p],
\]
the dimension of the $K$-vector space
$\mathrm{gr}^iD_{\dR}(\widetilde{L}'_{p}(-1)[1/p])$ is as follows:
\[
\dim_{K}\mathrm{gr}^iD_{\dR}(\widetilde{L}'_{p}(-1)[1/p])=
\begin{cases}
	0 & i \neq 1, 2, \\
	20 & i = 1, \\
	1 & i=2.
\end{cases}
\]
It follows that $\mathfrak{M}/\mathfrak{M}'$ is a free $\mathfrak{S}/E(u)$-module of rank $1$ by \cite[Lemma 1.2.2]{KisinCris}.

Therefore we have 
\[
\varphi^{*}\mathfrak{M}/\varphi^{*}\mathfrak{M}'\cong \varphi^{*}(\mathfrak{M}/\mathfrak{M}') \cong \mathfrak{S}/\varphi(E(u))
\]
and
\[
(\varphi^{*}\mathfrak{M}/\varphi^{*}\mathfrak{M}')\otimes_{\mathfrak{S}}W \cong (\mathfrak{S}/\varphi(E(u)))\otimes_{\mathfrak{S}}W \cong  W/p.
\]
\end{proof}

\subsection{Remarks on the \'etaleness of the Kuga-Satake morphism}

The \'etaleness of the Kuga-Satake morphism
\[ \KS \colon M^{\mathrm{sm}}_{2d, \K^p_0, \Z_{(p)}}\to Z_{\K^p_0}(\Lambda) \]
plays an important role in this paper.
It also plays an important role
in Madapusi Pera's proof of the Tate conjecture
for $K3$ surfaces \cite{MadapusiTateConj, KimMadapusiIntModel}.
The proof of the \'etaleness was given by Madapusi Pera
in  \cite[Theorem 5.8]{MadapusiTateConj} when $p$ is odd,
and in \cite[Proposition A.12]{KimMadapusiIntModel} when $p=2$.

In the course of writing this paper,
we found some issues on the proof of the \'etaleness of $\KS$.
We can avoid these issues using our results in this section.
See Remark \ref{Remark: Kim-Madapusi Pera gap} and
Remark \ref{Remark:Bloch-Kato and BMS} below for details.
(See also \cite{MadapusiErratum}, where Madapusi Pera gave a somewhat different argument.)

\begin{rem}\label{Remark: Kim-Madapusi Pera gap}
When $p=2$, the proof of the \'etaleness of $\KS$ in
\cite[Proposition A.12]{KimMadapusiIntModel}
seems to rely on an incorrect statement on the relation between endomorphisms of
an abelian scheme and the crystalline cohomology.
Here we explain how to avoid this issue using
Proposition \ref{Proposition:Flatness of cokernel}.
First, we briefly recall Madapusi Pera's proof.
Let $\widetilde{\mathbb{L}}_{\dR}$ be the filtered vector bundle with integrable connection on $\mathscr{S}_{\widetilde{\mathrm{K}}_0}$ associated with the $\widetilde{G}_0$-representation $\widetilde{L}_{\Z_{(p)}}$.
The pullback of $\widetilde{\mathbb{L}}_{\dR}$ by $\KS$
is also denoted by the same symbol $\widetilde{\mathbb{L}}_{\dR}$.
The $\Lambda$-structure for $M^{\mathrm{sm}}_{2d, \K^p_0, \Z_{(p)}}$ induces a homomorphism of vector bundles
\[
\iota_{\dR} \colon \Lambda\otimes_{\Z}\mathcal{O}_{M^{\mathrm{sm}}_{2d, \K^p_0, \Z_{(p)}}} \to \widetilde{\mathbb{L}}_{\dR}.
\]
In order to prove the \'etaleness of $\KS$
(see \cite[Theorem 5.8]{MadapusiTateConj}, \cite[Proposition A.12]{KimMadapusiIntModel}),
Madapusi Pera showed the orthogonal complement
\[
    \mathbb{L}_{\dR}:=\iota_{\dR}(\Lambda\otimes_{\Z}\mathcal{O}_{M^{\mathrm{sm}}_{2d, \K^p_0, \Z_{(p)}}})^{\perp} \subset \widetilde{\mathbb{L}}_{\dR}
\]
with respect to the canonical bilinear form on $\widetilde{\mathbb{L}}_{\dR}$ is isomorphic to
$\mathbb{P}^2_{\dR}$ (up to twist):
\[
\alpha_{\dR} \colon \mathbb{P}^2_{\dR} \cong \mathbb{L}_{\dR}(-1).
\]
Here $\mathbb{P}^2_{\dR}$ is the primitive part of the relative de Rham cohomology of the universal family on $M^{\mathrm{sm}}_{2d, \K^p_0, \Z_{(p)}}$;
see \cite[Proposition 5.11]{MadapusiTateConj} for $p \geq 3$ and the proof of \cite[Proposition A.12]{KimMadapusiIntModel} for $p=2$.
In order to prove $\mathbb{P}^2_{\dR}$ and $\mathbb{L}_{\dR}(-1)$ are isomorphic,
he proved that the cokernel of
$\iota_{\dR}$ is a \textit{vector bundle}
on $M^{\mathrm{sm}}_{2d, \K^p_0, \Z_{(p)}}$;
see the proof of \cite[Proposition A.12]{KimMadapusiIntModel}.
To prove this, it suffices to show the cokernel of $\iota_{\dR}$ is
a free $W(\overline{\F}_q)$-module at every $W(\overline{\F}_q)$-valued point of $M^{\mathrm{sm}}_{2d, \K^p_0, \Z_{(p)}}$;
see the proof of \cite[Lemma 6.16 (iv)]{MadapusiIntModel}.
When $p \geq 3$, this freeness property was proved
by applying \cite[Lemma 6.14]{MadapusiIntModel} for $e=1$.
However, when $p = 2$, we cannot apply \cite[Lemma 6.14]{MadapusiIntModel}
because the statement of \cite[Lemma 6.14]{MadapusiIntModel} is \textit{false} when $e=p-1$.
(There are counter-examples when $p=2$ and $e=1$; see \cite[Example 3.18]{BO}.)
To avoid this issue, we can use Proposition \ref{Proposition:Flatness of cokernel}
to prove the required freeness property.
Note that we essentially used the integral comparison theorem of
Bhatt-Morrow-Scholze \cite{BMS}
in the proof of this proposition.
\end{rem}

\begin{rem}\label{Remark:Bloch-Kato and BMS}
There is another issue on the integral comparison map used in the proof of the isomorphism
$\alpha_{\dR} \colon \mathbb{P}^2_{\dR} \cong \mathbb{L}_{\dR}(-1)$.
This issue exists for every $p$ (including odd $p$).
In \cite[Theorem 5.8]{MadapusiTateConj} and
\cite[Proposition A.12]{KimMadapusiIntModel},
Madapusi Pera used the integral comparison map
for varieties with ordinary reduction proved by
Bloch-Kato \cite[Theorem 9.6]{Bloch-Kato},
and the density of ordinary locus in the special fiber of $M^{\mathrm{sm}}_{2d, \K^p_0, \Z_{(p)}}$; see the proof of \cite[Lemma 5.10]{MadapusiTateConj}.
Madapusi Pera used the compatibility
between the integral comparison map of Bloch-Kato with other comparison maps used in \cite[Section 2]{MadapusiTateConj}.
However, we could not find appropriate references for the compatibility
used in \cite{MadapusiTateConj, KimMadapusiIntModel}
at least for small $p$.
We can avoid this issue by using the de Rham comparison map $c_{\dR, \mathscr{Y}_K}$ of Scholze \cite{ScholzeHodge} and the crystalline comparison map $c_{\cris, \mathscr{Y}}$ of Bhatt-Morrow-Scholze \cite{BMS}.
The integral comparison map $c_{\cris, \mathscr{Y}}$ is a substitute for the integral comparison map of Bloch-Kato; see also Proposition \ref{Proposition:CrisOrthogonal} in this paper.
The results of Blasius-Wintenberger \cite{Blasius}, which were used in \cite{MadapusiTateConj} and \cite[Appendix A]{KimMadapusiIntModel},
also hold using $c_{\dR, \mathscr{Y}_K}$ and $c_{\cris, \mathscr{Y}}$; see Section \ref{Subsection:Blasius-Wintenberger}.
Finally, we remark that the integral comparison map $c_{\cris, \mathscr{Y}}$
can be applied to all $K3$ surfaces including non-ordinary ones.
\end{rem}

\section{Construction of liftings of points on orthogonal Shimura varieties}
\label{Section:Lifting of a special point}

In this section, we consider a point $s \in Z_{\mathrm{K}^p}(\Lambda)(\F_q)$
which is the image of the point corresponding to
a quasi-polarized $K3$ surface $(X,\mathscr{L})$ of finite height over $\F_q$.
Since the Brauer group of $\F_q$ is trivial,
$\mathscr{L}$ is a line bundle on $X$.
(See \cite[Chapter 8, Section 1, Proposition 4]{BoschLuetkebohmertRaynaud}.)

We shall construct characteristic $0$ liftings of the point
$s \in Z_{\mathrm{K}^p}(\Lambda)(\F_q)$
over a finite extension of $W(\overline{\F}_q)[1/p]$
corresponding to characteristic $0$ liftings of
the formal Brauer group $\widehat{\Br}$ of $X$.
We construct such liftings using integral $p$-adic Hodge theory and
our results on $F$-crystals on Shimura varieties in
Section \ref{Section:$F$-crystals on Shimura varieties}.

Note that, when $p \geq 5$, a stronger result can be obtained by
the method of Nygaard-Ogus \cite{NygaardOgus};
see Remark \ref{Remark:NygaardOgus}.
But, when $p \leq 3$,
it seems difficult to apply their methods.
(The deformation theory of $K3$ crystals developed in \cite{NygaardOgus}
does not work in characteristic $p = 2$ or $3$;
see \cite[p.498]{NygaardOgus}.)
The method of this paper can be applied to
$K3$ surfaces of finite height in any characteristic.

\subsection{Liftings of points with additional properties}
\label{Subsection:Liftings of points with additional properties}

In this subsection, we shall state our results on characteristic $0$ liftings of points on $Z_{\K^p}(\Lambda)$ with additional properties.

We put $k := \F_q$, and consider the situation as in Section \ref{Subsection:$F$-crystals on Shimura varieties and the cohomology of K3 surfaces}.
We assume the height $h$ of the $K3$ surface $X$ is finite.

As in the proof of Proposition \ref{Proposition:CrysDecomposition},
we have a natural embedding
\[
\mathbb{D}(\widehat{\Br})\cong L_{1+1/h}(1) \hookrightarrow P^2_{\cris}(X/W)(1).
\]
By Proposition \ref{Proposition:CrisOrthogonal}, we have an embedding
\[
\mathbb{D}(\widehat{\Br}) \hookrightarrow \widetilde{L}_{\cris}.
\]

Let $E$ be a finite totally ramified extension of $K_0=W[1/p]$ and $\mathcal{G}$ a one-dimensional smooth formal group over $\O_{E}$ whose special fiber is isomorphic to $\widehat{\Br}$.
We shall construct characteristic $0$ liftings of the point
$s \in Z_{\mathrm{K}^p}(\Lambda)(\F_q)$
over a finite extension of $W(\overline{\F}_q)[1/p]$
corresponding to $\mathcal{G}$.

Let
	\[
	\mathrm{Fil}(\mathcal{G}) \hookrightarrow \mathbb{D}(\widehat{\Br}) \otimes_{W} E\hookrightarrow \widetilde{L}_{\cris}\otimes_{W}E
	\]
be the filtration associated with $\mathcal{G}$, i.e. the $E$-vector subspace generated by the inverse image of the Hodge filtration $\Fil^1\mathbb{D}(\mathcal{G}_{{\O_E/p}})(\O_E)$ under the isomorphism
\[
\mathbb{D}(\widehat{\Br}) \otimes_{W} E \cong \mathbb{D}(\mathcal{G}_{{\O_E/p}})(\O_E) \otimes_{\O_E} E.
\]
Take a generator $e$ of $\mathrm{Fil}(\mathcal{G})$,
	and consider 
\[
i(e):=(i_{\cris}\otimes_{W}{E})(e),
\]
the image of $e$ under the embedding 
\[
i_{\cris}\otimes_{W}E \colon \widetilde{L}_{\cris}\otimes_{W}E \hookrightarrow \End_{E}(H^1_{\cris}(\mathcal{A}_s/W)^{\vee}\otimes_{W}E).
\]
Let $i(e)(H^1_{\cris}(\mathcal{A}_s/W)^{\vee}\otimes_{W}E)$ be the image of
\[
i(e) \colon H^1_{\cris}(\mathcal{A}_s/W)^{\vee}\otimes_{W}E \to H^1_{\cris}(\mathcal{A}_s/W)^{\vee}\otimes_{W}E.
\]

We define a decreasing filtration
$\{\Fil^i(\widetilde{L}_{\cris}\otimes_{W}E)\}_i$
on $\widetilde{L}_{\cris}\otimes_{W}E$ by
\[
\mathrm{Fil}^i(\widetilde{L}_{\cris}\otimes_{W}E):=
\begin{cases}
0 & i \geq 2, \\
\mathrm{Fil}(\mathcal{G}) & i = 1, \\
\mathrm{Fil}(\mathcal{G})^{\perp} & i=0, \\
\widetilde{L}_{\cris}\otimes_{W}E & i \leq -1.
\end{cases}
\]
We also define a decreasing filtration
$\{ \Fil^i(H^1_{\cris}(\mathcal{A}_s/W)^{\vee}\otimes_{W}E) \}_i$
on $H^1_{\cris}(\mathcal{A}_s/W)^{\vee}\otimes_{W}E$ by
\[
\mathrm{Fil}^i(H^1_{\cris}(\mathcal{A}_s/W)^{\vee}\otimes_{W}E):=\begin{cases}
0 & i \geq 1, \\
i(e)(H^1_{\cris}(\mathcal{A}_s/W)^{\vee}\otimes_{W}E) & i=0, \\
H^1_{\cris}(\mathcal{A}_s/W)^{\vee}\otimes_{W}E & i \leq -1.
\end{cases}
\]

The  following theorem is the main result in this section.

\begin{thm}\label{Theorem:CMliftingAbelianVariety}
Let $K$ be the composite of $E$ and $W(\overline{\F}_q)[1/p]$.
There is an $\O_{K}$-valued point
$\widetilde{s} \in Z_{\mathrm{K}^p}(\Lambda)(\O_{K})$
satisfying the following properties:
	\begin{enumerate}
		\item $\widetilde{s}$ is a lift of $\overline{s} \in Z_{\mathrm{K}^p}(\Lambda)(\overline{\F}_q)$, where $\overline{s}\in Z_{\mathrm{K}^p}(\Lambda)(\overline{\F}_q)$ 
is a geometric point above $s$.
		\item The Hodge filtration on
		$H^1_{\cris}(\mathcal{A}_s/W)^{\vee}\otimes_{W}K$
		corresponding to $\mathcal{A}_{\widetilde{s}}$ over $\O_K$ coincides with the filtration defined by
        \[
        \mathrm{Fil}^i(H^1_{\cris}(\mathcal{A}_s/W)^{\vee}\otimes_{W}K) :=
        \mathrm{Fil}^i(H^1_{\cris}(\mathcal{A}_s/W)^{\vee}\otimes_{W}E) \otimes_E K.
        \]
	\end{enumerate} 
\end{thm}

\begin{rem}
\label{Remark:NygaardOgus}
Recall that the $K3$ surface $X$ over $\F_q$ comes from the $\F_q$-valued point $s \in M^{\mathrm{sm}}_{2d, \K^p_0, \Z_{(p)}}(\F_q)$ satisfying the conditions as in Section \ref{Subsection:$F$-crystals on Shimura varieties and the cohomology of K3 surfaces}.
In particular, the Kuga-Satake abelian variety $\mathcal{A}_s$ is defined over $\F_q$.
When $p \geq 5$, a result stronger than Theorem \ref{Theorem:CMliftingAbelianVariety}
can be obtained by the method of Nygaard-Ogus in \cite{NygaardOgus}.
In fact, when $p \geq 5$,
Nygaard-Ogus constructed a characteristic $0$ lifting of
the $K3$ surface $X$ over $\O_E$ corresponding to the one-dimensional smooth formal group $\mathcal{G}$.
By the Kuga-Satake morphism, we find an $\O_E$-valued point
$\widetilde{s} \in Z_{\mathrm{K}^p}(\Lambda)(\O_E)$
lifting $s$ such that the Hodge filtration on
$H^1_{\cris}(\mathcal{A}_s/W)^{\vee}\otimes_{W}E$
corresponding to $\mathcal{A}_{\widetilde{s}}$ over $\O_E$ coincides with
the filtration
$\{ \Fil^i(H^1_{\cris}(\mathcal{A}_s/W)^{\vee}\otimes_{W}E) \}_i$
defined as above.
Hence we do not need to take the composite with $W(\overline{\F}_q)[1/p]$ when $p \geq 5$.
Note that, in this paper,
we construct characteristic $0$ liftings over $\O_K$, not over $\O_E$.
Currently, we do not know how to obtain an $\O_E$-valued point
by the methods of this paper.
\end{rem}

\subsection{Some lemmas}\label{Subsection:Some lemmas}
In this subsection, we give some lemmas which will be used in the proof of Theorem \ref{Theorem:CMliftingAbelianVariety}.
 
 By Proposition \ref{Proposition:CrysDecomposition}, we have the following isomorphism of $W$-modules inducing an isomorphism of $F$-isocrystals after inverting $p$:
\[
\widetilde{L}_{\cris} \cong \mathbb{D}(\widehat{\Br}^{*})(1)\oplus \mathbb{D}(D)\oplus\mathbb{D}(\widehat{\Br}).
\]
	
Since $D$ is an \'etale $p$-divisible group over $k$, it canonically extends over $\O_E$ and we also denote it by $D$.
Let $\mathcal{G}^*$ be the Cartier dual of $\mathcal{G}$.
We define a $\Gal(\overline{E}/E)$-stable $\Z_p$-lattice $\widetilde{L}_p$ in a crystalline representation by
\[
\widetilde{L}_p:=(T_p\mathcal{G}^*)^{\vee}(1)\oplus (T_pD)^{\vee}\oplus (T_p\mathcal{G})^{\vee}.
\]

Recall that, for every $p$-divisible group $\mathscr{G}$ over $\O_K$, there is a comparison isomorphism of filtered $\varphi$-modules
\[
c_{\mathscr{G}} \colon D_{\cris}((T_{p}\mathscr{G})^{\vee}[1/p])  \cong \mathbb{D}(\mathscr{G}_k)(W)[1/p].
\]
Moreover the composite 
\[\Mcris((T_{p}\mathscr{G})^{\vee})[1/p] \cong D_{\cris}((T_{p}\mathscr{G})^{\vee}[1/p])  \overset{c_{\mathscr{G}}}{\cong} \mathbb{D}(\mathscr{G}_k)(W)[1/p]
	\]
maps $\Mcris((T_{p}\mathscr{G})^{\vee})$ onto $\mathbb{D}(\mathscr{G}_k)(W)$.
See Section \ref{Subsection:Breuil-Kisin modules and p-divisible groups} and Section \ref{Subsection:Comparison isomorphisms for $p$-divisible groups} for details.

\begin{lem}\label{Lemma:Lifting-LCris}
	The $\Gal(\overline{E}/E)$-module $\widetilde{L}_p$ satisfies the following properties:
	\begin{enumerate}
	\item $\widetilde{L}_p$ admits an even perfect bilinear form $( \ , \ )$ which is $\Gal(\overline{E}/E)$-invariant.
		\item There is a $\Gal(\overline{E}/E)$-equivariant homomorphism
		$\iota_p: \Lambda_{\Z_p} \to \widetilde{L}_p$
		preserving the bilinear forms.
		\item There is an isometry of $F$-isocrystals
		    \[D_{\cris}(\widetilde{L}_p[1/p])\cong \widetilde{L}_{\cris}[1/p]\]
		    such that the following composite
		    \[
		    \Mcris(\widetilde{L}_p)[1/p]\cong D_{\cris}(\widetilde{L}_p[1/p])\cong \widetilde{L}_{\cris}[1/p]
		    \]
		    maps $\Mcris(\widetilde{L}_p)$ isomorphically onto $\widetilde{L}_{\cris}$.
		\item The following diagram is commutative:
            \[
            \xymatrix{ \Lambda_{W}\ar[r]^-{\iota_p} \ar[rd]_-{\iota_{\cris}} & \Mcris(\widetilde{L}_p) \ar[d]^-{\cong} \\
             & \widetilde{L}_{\cris}.
            }
            \]
        \item The following composite
        \[
        D_{\dR}(\widetilde{L}_p[1/p])\cong D_{\cris}(\widetilde{L}_p[1/p])\otimes_{K_0}E \cong \widetilde{L}_{\cris}[1/p] \otimes_{K_0}E
        \]
        preserves the filtrations.
	\end{enumerate}
\end{lem}	

\begin{proof}
    We equip the $\Z_p$-module
	\[(T_p\mathcal{G}^*)^{\vee}(1)\oplus (T_p\mathcal{G})^{\vee}
	\]
	with a natural bilinear form, which is even and perfect.
	Since $D$ is an \'etale $p$-divisible group,
	the even perfect bilinear form
	\[
	\mathbb{D}(D) \times \mathbb{D}(D) \to W
	\]
induces an even perfect bilinear form
	\[
	(T_pD)^{\vee} \times (T_pD)^{\vee} \to \Z_p
	\]
	which is $\Gal(\overline{E}/E)$-invariant.
	
	Let $\underline{\Lambda}^{\vee}_{\Z_p}$ be the $p$-divisible group over $\O_E$ associated with the $\Z_p$-module $\Lambda^{\vee}_{\Z_p}$ with the trivial $\Gal(\overline{E}/E)$-action.
	So we have an isomorphism of $F$-crystals
	$\mathbb{D}(\underline{\Lambda}^{\vee}_{\Z_p})\cong \Lambda_W$.
	
	The image of the homomorphism $\iota_{\cris} \colon \Lambda_W \to \widetilde{L}_{\cris}$ is contained in $\mathbb{D}(D)$, hence we have a homomorphism $\Lambda_W \to \mathbb{D}(D)$.
	Therefore, we have a morphism $D \to \underline{\Lambda}^{\vee}_{\Z_p}$ of \'etale $p$-divisible groups over $k$.
	 This extends over $\O_E$.
	Then we have a $\Gal(\overline{E}/E)$-equivariant homomorphism 
	$\Lambda_{\Z_p} \to (T_pD)^{\vee}$.
	This homomorphism preserves the bilinear forms by construction.  
	
	The other properties follow from \cite[Theorem 2.12]{KimMadapusiIntModel}.
	See also Section \ref{Subsection:Breuil-Kisin modules and p-divisible groups} and Section \ref{Subsection:Comparison isomorphisms for $p$-divisible groups}.
\end{proof}

\begin{lem}\label{Lemma:admissible}
There is a crystalline $\Gal(\overline{E}/E)$-representation $H_{\mathrm{\acute{e}t}, \Q_p}$ over $\Q_p$ such that 
\[
D_{\cris}(H_{\mathrm{\acute{e}t}, \Q_p}) \cong H^1_{\cris}(\mathcal{A}_s/W)^{\vee}[1/p]
\]
as filtered $\varphi$-modules.

\end{lem}

\begin{proof}
It is a theorem of Colmez-Fontaine that any weakly admissible filtered
$\varphi$-module is admissible; see \cite[Th\'eor{\`e}me A]{Colmez-Fontaine}.
(Today, there are several alternative proofs of this theorem.
For example, see \cite[Proposition 2.1.5]{KisinCris}.)
Hence it is enough to prove that the filtered $\varphi$-module $H^1_{\cris}(\mathcal{A}_s/W)^{\vee}[1/p]$ is weakly admissible.

As in Section \ref{Subsection:Hodge tensors}, we fix an isomorphism of $W$-modules
	\[
	H_{W} \cong H^1_{\cris}(\mathcal{A}_{s}/W)^{\vee}
	\]
	which carries $\{ s_{\alpha}\}$ to $\{ s_{\alpha, \cris, s} \}$ and makes the following diagram commutative:
	\[
	\xymatrix{
\widetilde{L}_{W} \ar[d]_-{\cong} \ar[r]^-{i} & \End_{W}(H_{W}) \ar[d]^-{\cong} \\
\widetilde{L}_{\cris} \ar[r]^-{i_{\cris}} & \End_{W}(H^1_{\cris}(\mathcal{A}_{s}/W)^{\vee}).
}
	\]
	Using this isomorphism, we equip $H^1_{\cris}(\mathcal{A}_{s}/W)^{\vee}$ with a right action of the Clifford algebra $\Cl_{W}:=\Cl\otimes_{\Z}W$ using the natural right action of $\Cl_{W}$ on $H_W$.

	The Clifford algebra $\Cl_{\F_p}$ is a central simple algebra over the finite field $\F_p$ by \cite[\S 9.4, Corollaire]{Bourbaki}.
	Hence it is isomorphic to $\mathrm{M}_n(\F_p)$ as a $\F_p$-algebra where $n=2^{11}$.
	Then $\Cl_{\Z_p}$ is isomorphic to $\mathrm{M}_n(\Z_p)$ as a $\Z_p$-algebra by \cite[Lemma 5.1.16]{Knus}.
	We fix an isomorphism $\Cl_{\Z_p} \cong \mathrm{M}_n(\Z_p)$.
	Using this isomorphism, we equip $H^1_{\cris}(\mathcal{A}_{s}/W)^{\vee}$ with a right $\mathrm{M}_n(W)$-action.
	
	The right $\mathrm{M}_n(\Z_p)$-action on $H^1_{\cris}(\mathcal{A}_{s}/W)^{\vee}[1/p]$ is compatible with the Frobenius automorphisms,
	and $\Fil^0$ is an $\mathrm{M}_n(E)$-submodule of
	$H^1_{\cris}(\mathcal{A}_{s}/W)^{\vee}\otimes_W E$.
	Therefore a $\sigma$-linear endomorphism $\varphi \otimes \sigma$ on
	\[
	H^1_{\cris}(\mathcal{A}_{s}/W)^{\vee}[1/p] \otimes_{\mathrm{M}_n(K_0)}({K_0}^{n})
	\]
	is well-defined, and
	we consider the following $E$-vector subspace
	\[
	\mathrm{Fil}^0 \otimes_{\mathrm{M}_n(E)} (E^{n}) \subset (H^1_{\cris}(\mathcal{A}_{s}/W)^{\vee} \otimes_{W}E) \otimes_{\mathrm{M}_n(E)} (E^{n}).
	\]
	These determine the structure of a filtered $\varphi$-module on 
	\[
	B:= H^1_{\cris}(\mathcal{A}_{s}/W)^{\vee}[1/p]\otimes_{\mathrm{M}_n(K_0)}({K_0}^{n}).
	\]
	Since we have an isomorphism of filtered $\varphi$-modules
	\[
	B^{n} \cong H^1_{\cris}(\mathcal{A}_{s}/W)^{\vee}[1/p],
	\]
	it suffices to show that $B$ is weakly admissible.
	
	The embedding $i_{\cris}$ induces an isomorphism of $W$-modules
	\[
	\Cl(\widetilde{L}_{\cris})\cong \End_{\Cl_W}(H^1_{\cris}(\mathcal{A}_{s}/W)^{\vee}),
	\]
	 which is an isomorphism of filtered $\varphi$-modules after inverting $p$.
	By the Morita equivalence, we have
	\[
	\End_{\Cl_{K_0}}(H^1_{\cris}(\mathcal{A}_{s}/W)^{\vee}[1/p])\cong \End_{K_0}(B).
	\]
	Hence we have an isomorphism of filtered $\varphi$-modules
	\[
	\Cl(\widetilde{L}_{\cris}[1/p])\cong \End_{K_0}(B).
	\]
	Since $\widetilde{L}_{\cris}[1/p]$ is weakly admissible by Lemma \ref{Lemma:Lifting-LCris},
	the filtered $\varphi$-module
	$\Cl(\widetilde{L}_{\cris}[1/p])$ is weakly admissible.
    Therefore,
    \[ B \otimes_{K_0}B^{\vee} \cong \End_{K_0}(B) \]
	is also weakly admissible.
	(Recall that the tensor product of two weakly admissible $\varphi$-modules
	is weakly admissible; see \cite[Corollaire 1]{Colmez-Fontaine} for example.)
	
	In order to show $B$ is weakly admissible,
	we have to show $t_{H}(B')\leq t_{N}(B')$ for every filtered $\varphi$-submodule $B' \subset B$.
	(For the definition of the functions $t_{H}$ and $t_{N}$, see \cite[1.1.3]{KisinCris} for example.)
	Note that we have
	\[
	t_{H}(B)= t_{N}(B)=-2^{10}.
	\]
	Since
	$B \otimes_{K_0}B^{\vee}$
	is weakly admissible, we have
	\[
	t_{H}(B'\otimes B^{\vee})\leq t_{N}(B'\otimes B^{\vee}).
	\]
	Each side of the inequality may be computed as 
	\begin{align*}
		t_{H}(B'\otimes B^{\vee}) & = \dim_{K_0}(B')t_{H}(B')+ \dim_{K_0}(B^{\vee})t_{H}(B^{\vee})\\
         & =\dim_{K_0}(B')t_{H}(B')- \dim_{K_0}(B)t_{H}(B)
	\end{align*}
		and
	\begin{align*}
		t_{N}(B'\otimes B^{\vee}) & = \dim_{K_0}(B')t_{N}(B')+ \dim_{K_0}(B^{\vee})t_{N}(B^{\vee})\\
         & =\dim_{K_0}(B')t_{N}(B')- \dim_{K_0}(B)t_{N}(B).
	\end{align*}
    It follows that $t_{H}(B')\leq t_{N}(B')$.
    
    Therefore, $B$ is weakly admissible, and the proof of Lemma \ref{Lemma:admissible} is complete.
\end{proof}

\subsection{$p$-divisible groups adapted to general spin groups}
\label{Subsection:p-divisible groups adapted to general spin groups}

By Lemma \ref{Lemma:admissible},
there is a crystalline $\Gal(\overline{E}/E)$-representation $H_{\mathrm{\acute{e}t}, \Q_p}$ over $\Q_p$ such that 
\[
D_{\cris}(H_{\mathrm{\acute{e}t}, \Q_p}) \cong H^1_{\cris}(\mathcal{A}_s/W)^{\vee}\otimes_{W}E
\]
as filtered $\varphi$-modules.

The tensors $\{ s_{\beta, \cris, s} \} \subset \{ s_{\alpha, \cris, s}\}$ corresponding to $p^{\pm}$, $\{ r_{e_i} \}_{1 \leq i \leq 2^{22}}$, and the endomorphism $\pi'$
preserve the filtration on
\[ (H^1_{\cris}(\mathcal{A}_s/W)^{\vee}\otimes_{W}E)^{\otimes} \]
by Proposition \ref{Proposition: Filtrations on Clifford algebras}.
(For the tensors $p^{\pm}$, $\{ r_{e_i} \}_{1 \leq i \leq 2^{22}}$, and $\pi'$, see Section \ref{Subsection: Representations of general spin groups and Hodge tensors} for details.)
Therefore the tensors $\{ s_{\beta, \cris, s} \}$ induce tensors $\{ s'_{\beta, p} \}$ of $H^{\otimes}_{\mathrm{\acute{e}t}, \Q_p}$.
By Proposition \ref{Proposition: Filtrations on Clifford algebras} again, the injection
\[
i_{\cris}\otimes_{W}E \colon \widetilde{L}_{\cris}\otimes_{W}E \hookrightarrow \End_{W}(H^1_{\cris}(\mathcal{A}_{s}/W)^{\vee}\otimes_{W}E)
\]
preserves the filtrations.
Therefore $i_{\cris}\otimes_{W}E$ induces an inclusion of crystalline $\Gal(\overline{E}/E)$-representations
\[
\widetilde{L}_p[1/p] \hookrightarrow \End_{\Q_p}(H_{\mathrm{\acute{e}t}, \Q_p}).
\]

The derived group $\widetilde{G}^{\mathrm{der}}_{\Q_{p}}$
of $\widetilde{G}_{\Q_{p}}$ is the spin double cover of
$\widetilde{G}_{0, \Q_p}=\SO(\widetilde{L}_{\Q_p})$,
and a simply connected semisimple algebraic group over $\Q_p$.
So, we have $H^1(\Q_p, \widetilde{G}^{\mathrm{der}}_{\Q_{p}})=0$ by \cite[Theorem 6.4]{Platonov-Rapinchuk}.
This implies that
$
H^1(\Q_p, \widetilde{G}_{\Q_{p}})=0
$
and every $\widetilde{G}_{\Q_{p}}$-torsor over $\Q_{p}$ is trivial.
Hence there is an isomorphism of $\Q_p$-vector spaces
\[
H_{\Q_p} \cong H_{\mathrm{\acute{e}t}, \Q_p}
\]
which carries $\{ s_{\beta} \}$ to $\{ s'_{\beta, p} \}$ and induces the following commutative diagram:
\[
	\xymatrix{
 \widetilde{L}_{\Q_p} \ar[d]_-{\cong} \ar[r]^-{i} & \End_{\Q_p}(H_{\Q_p}) \ar[d]^-{\cong} \\
\widetilde{L}_p[1/p] \ar[r]^-{} & \End_{\Q_p}(H_{\mathrm{\acute{e}t}, \Q_p}),
}
\]
where $\widetilde{L}_{\Q_p} \cong \widetilde{L}_p[1/p]$
is an isometry over $\Q_p$.

Since the tensors $\{ s'_{\beta, p} \}$ are fixed by $\Gal(\overline{E}/E)$
and $\widetilde{G}(\Q_p)$ is the stabilizer of the tensors $\{ s_{\beta} \}$,
we have the following commutative diagram:
\[
\xymatrix{
&\GSpin(\widetilde{L}_p)(\Q_p) \ar[r]^-{\cong} \ar[d]^-{} & \widetilde{G}(\Q_p) \ar[d]^-{} \\
\Gal(\overline{E}/E) \ar[r]^-{\rho_0} \ar[ru]^-{\rho} &\SO(\widetilde{L}_p)(\Q_p) \ar[r]^-{\cong}  & \widetilde{G}_0(\Q_p). 
}
\]
Here $\rho_0$ is the continuous homomorphism corresponding to the crystalline $\Gal(\overline{E}/E)$-representation $\widetilde{L}_p[1/p]$.
(See Lemma \ref{Lemma:Lifting-LCris} for the properties of the $\Gal(\overline{E}/E)$-stable $\Z_p$-lattice $\widetilde{L}_p$
in $\widetilde{L}_p[1/p]$.)

\begin{lem}\label{Lemma: Galois-stable}
We have
$\rho(\Gal(\overline{E}/E)) \subset \GSpin(\widetilde{L}_p)(\Z_p).$
\end{lem}

\begin{proof}
Since we have an exact sequence of group schemes over $\Z_{p}$
\[
1 \to
\G_{m, \Z_{p}} \to \GSpin(\widetilde{L}_p) \to \SO(\widetilde{L}_p) \to 1,
\]
it follows that
\[
\GSpin(\widetilde{L}_p)(\Z_p) \to \SO(\widetilde{L}_p)(\Z_p)
\]
is surjective by Hilbert's theorem 90 and the smoothness of $\G_{m, \Z_{p}}$.
Moreover we have
$
\rho_0(\Gal(\overline{E}/E)) \subset \SO(\widetilde{L}_p)(\Z_p).
$
Thus, for every $g \in \Gal(\overline{E}/E)$, there is an element $g' \in \GSpin(\widetilde{L}_p)(\Z_p)$ and $a \in \Q^{\times}_p$ such that $\rho(g)=ag'$.
Let $\nu \colon \GSpin(\widetilde{L}_p) \to \G_m$ be the spinor norm.
Since $\Gal(\overline{E}/E))$ is compact, we have $\nu(ag')=\nu(\rho(g)) \in \Z^{\times}_p$.
Hence we have $a \in \Z^{\times}_p$ and
$\rho(g)=ag' \in \GSpin(\widetilde{L}_p)(\Z_p)$.
In conclusion, we have
$
\rho(\Gal(\overline{E}/E)) \subset \GSpin(\widetilde{L}_p)(\Z_p).
$
\end{proof}

By Lemma \ref{Lemma: Galois-stable}, we see that $\Cl(\widetilde{L}_p)$ is a $\Gal(\overline{E}/E)$-stable $\Z_p$-lattice in the crystalline representation $H_{\mathrm{\acute{e}t}, \Q_p}$, which will be denoted by $H_{\mathrm{\acute{e}t}}$.

Let $K$ be the composite of $E$ and $W(\overline{\F}_q)[1/p]$. 
In Proposition \ref{Proposition:LiftingCrisLattice} below, we shall show that $H_{\mathrm{\acute{e}t}}$ satisfies properties which should be satisfied when $H_{\mathrm{\acute{e}t}}$ is the $p$-adic Tate module of the abelian scheme $\mathcal{A}_{\widetilde{s}}$ associated with a desired lift $\widetilde{s} \in Z_{\mathrm{K}^p}(\Lambda)(\O_{K})$.

The next lemma will be used in the proof of Proposition \ref{Proposition:LiftingCrisLattice}.

\begin{lem}\label{Lemma: isometry}
There is an isometry $\widetilde{L}_p \cong \widetilde{L}_{\Z_p}$ over $\Z_p$.
\end{lem}

\begin{proof}
    There is an isometry $\widetilde{L}_{\cris} \cong \widetilde{L}_W$ over $W$; see Section \ref{Subsection:Hodge tensors}.
    We have an isometry $\Mcris(\widetilde{L}_p) \cong \widetilde{L}_{\cris}$  over $W$ by Lemma \ref{Lemma:Lifting-LCris}.
    By using \cite[Corollary 1.3.5]{KisinIntModel}, we see that there is an isometry $\widetilde{L}_p\otimes_{\Z_p}W \cong \Mcris(\widetilde{L}_p)$ over $W$.
    So, the functor $\mathrm{Isom}(\widetilde{L}_p, \widetilde{L}_{\Z_p})$ on $\Z_p$-algebras which sends a $\Z_p$-algebra $R$
    to the set of isometries over $R$ from $(\widetilde{L}_p)_R$ to $\widetilde{L}_{R}$ is represented by an $\mathrm{O}(\widetilde{L}_{\Z_p})$-torsor, which is also denoted by $\mathrm{Isom}(\widetilde{L}_p, \widetilde{L}_{\Z_p})$.
    Here $\mathrm{O}(\widetilde{L}_{\Z_p})$ is the orthogonal group over $\Z_p$.
    This $\mathrm{O}(\widetilde{L}_{\Z_p})$-torsor corresponds to an element $x \in H^1(\Z_p, \mathrm{O}(\widetilde{L}_{\Z_p}))$.
    Since there is an isometry $\widetilde{L}_p[1/p] \cong \widetilde{L}_{\Q_p}$ over $\Q_p$, it follows that $x$ comes from an element of $H^1(\Z_p, \mathrm{SO}(\widetilde{L}_{\Z_p}))$.
    By Lang's theorem and the smoothness of $\mathrm{SO}(\widetilde{L}_{\Z_p})$, we have $H^1(\Z_p, \mathrm{SO}(\widetilde{L}_{\Z_p}))=0$.
Thus, we see that the $\mathrm{O}(\widetilde{L}_{\Z_p})$-torsor $\mathrm{Isom}(\widetilde{L}_p, \widetilde{L}_{\Z_p})$ is trivial.
\end{proof}

\begin{prop}\label{Proposition:LiftingCrisLattice}
	There are $\Gal(\overline{E}/E)$-invariant tensors $\{ s_{\alpha, p}\}$ of $H^{\otimes}_{\mathrm{\acute{e}t}}$ and an isomorphism of $W(\overline{\F}_q)$-modules
\[
\Phi \colon \Mcris(H_{\mathrm{\acute{e}t}})\otimes_{W}W(\overline{\F}_q) \overset{\cong}{\longrightarrow} H^1_{\cris}(\mathcal{A}_s/W)^{\vee}\otimes_{W} W(\overline{\F}_q)
\]
	satisfying the following properties:
	\begin{enumerate}
	\item There is an isomorphism of $\Z_p$-modules
	    \[
	    H_{\Z_p}\cong H_{\mathrm{\acute{e}t}}
	    \]
	    which carries $\{ s_{\alpha} \}$ to $\{ s_{\alpha, p} \}$ and induces the following commutative diagram:
\[
	\xymatrix{
 \widetilde{L}_{\Z_p} \ar[d]_-{\cong} \ar[r]^-{i} & \End_{\Z_p}(H_{\Z_p}) \ar[d]^-{\cong} \\
\widetilde{L}_p \ar[r]^-{} & \End_{\Z_p}(H_{\mathrm{\acute{e}t}}),
}
\]
where $\widetilde{L}_{\Z_p} \cong \widetilde{L}_p$
is an isometry over $\Z_p$.
	\item The isomorphism $\Phi$ is an isomorphism of $F$-isocrystals after inverting $p$.
	\item The isomorphism $\Phi$ carries $\{ \Mcris(s_{\alpha, p}) \}$ to $\{ s_{\alpha, \cris, s} \}$.
	\item The following diagram is commutative:
		\[
		\xymatrix{\Mcris(\widetilde{L}_p)\otimes_{W}W(\overline{\F}_q) \ar@{=}[d]^-{} \ar[r]^-{} & \End_{W(\overline{\F}_q)}(\Mcris(H_{\mathrm{\acute{e}t}})\otimes_{W}W(\overline{\F}_q)) \ar[d]^-{\cong} \\
 \widetilde{L}_{\cris}\otimes_{W} W(\overline{\F}_q) \ar[r]^-{i_{\cris}} & \End_{W(\overline{\F}_q)}(H^1_{\cris}(\mathcal{A}_{s}/W)^{\vee}\otimes_{W} W(\overline{\F}_q)),
}
\]
where we identify
		$
		\Mcris(\widetilde{L}_p)\otimes_{W}W(\overline{\F}_q)
		$
		with
		$\widetilde{L}_{\cris}\otimes_{W}W(\overline{\F}_q)$
		using the isomorphism
$\Mcris(\widetilde{L}_p)\cong \widetilde{L}_{\cris}$
in Lemma \ref{Lemma:Lifting-LCris}.
	\end{enumerate}
\end{prop}
\begin{proof}
There is an isometry $f \colon \widetilde{L}_p \cong \widetilde{L}_{\Z_p}$ over $\Z_p$ by Lemma \ref{Lemma: isometry}.
We choose such an isometry.

Let $\{ f^{*}(s_{\alpha}) \}$ be the tensors on $H_{\mathrm{\acute{e}t}}=\Cl(\widetilde{L}_p)$ corresponding to the tensors $\{ s_{\alpha} \}$ under $f$.
Since we have
\[ \rho(\Gal(\overline{E}/E)) \subset \GSpin(\widetilde{L}_p)(\Z_p), \]
these tensors are $\Gal(\overline{E}/E)$-invariant.
Let $\{ \mathfrak{M}(f^{*}(s_{\alpha})) \}$ be the induced tensors of $\mathfrak{M}(H_{\mathrm{\acute{e}t}})$.

As in the proofs of \cite[Proposition 1.3.4, Corollary 1.3.5]{KisinIntModel},
there is an isomorphism of $\mathfrak{S}$-modules
\[
\Phi_f \colon \mathfrak{M}(H_{\mathrm{\acute{e}t}}) \cong H_{\mathfrak{S}}
\]
 which carries $\{ \mathfrak{M}(f^{*}(s_{\alpha})) \}$ to $\{ s_{\alpha} \}$ and induces an isometry over $\mathfrak{S}$
\[
\mathfrak{M}(\widetilde{L}_p) \cong \widetilde{L}_{\mathfrak{S}}.
\]
We denote it also by $\Phi_f$.
Hence we have an isomorphism of $W$-modules
\[
\Phi_f \colon \Mcris(H_{\mathrm{\acute{e}t}}) \cong H_{W}
\]
which carries $\{ \Mcris(f^{*}(s_{\alpha})) \}$ to $\{ s_{\alpha} \}$ and induces an isometry over $W$
\[
\Phi_f \colon \Mcris(\widetilde{L}_p) \cong \widetilde{L}_{W}.
\]
We also choose an isomorphism
\[
	\Phi' \colon H_{W} \cong H^1_{\cris}(\mathcal{A}_{s}/W)^{\vee}
	\]
	which carries $\{ s_{\alpha}\}$ to $\{ s_{\alpha, \cris, s} \}$ which induces an isometry $\Phi' \colon \widetilde{L}_{W} \cong \widetilde{L}_{\cris}$.
	
Consider the composite of the following isomorphisms
\[
\psi \colon \widetilde{L}_{\cris}=\Mcris(\widetilde{L}_p) \overset{\Phi_f}{\longrightarrow}   \widetilde{L}_{W} \overset{\Phi'}{\longrightarrow} \widetilde{L}_{\cris}.
\]

We may assume $\det(\psi)=1$ as follows:
Suppose $\det(\psi)=-1$.
We take an isometry $h \colon \widetilde{L}_{\Z_p} \cong \widetilde{L}_{\Z_p}$ with $\det(h)=-1$.
(For example, $h$ can be constructed using a decomposition $\widetilde{L}_{\Z_p}\cong U_{\Z_p} \oplus U^{\perp}_{\Z_p}$ of quadratic spaces.)
We consider the tensors $\{ (h\circ f)^{*}(s_{\alpha}) \}$ on $H^{\otimes}_{\mathrm{\acute{e}t}}$ and the tensors $\{ \mathfrak{M}((h\circ f)^{*}(s_{\alpha})) \}$ on $\mathfrak{M}(H_{\mathrm{\acute{e}t}})^{\otimes}$.
Then the isomorphism
\[
\mathfrak{M}(H_{\mathrm{\acute{e}t}}) \overset{\Phi_f}{\longrightarrow} H_{\mathfrak{S}}  \overset{h}{\longrightarrow} H_{\mathfrak{S}}.
\]
carries $\{ \mathfrak{M}((h\circ f)^{*}(s_{\alpha})) \}$ to $\{ s_{\alpha} \}$
and the isomorphism 
\[
\xymatrix{\widetilde{L}_{\cris}=\Mcris(\widetilde{L}_p)\ar[r]^-{h \circ \Phi_f} &   \widetilde{L}_{W} \ar[r]^-{\Phi'}& \widetilde{L}_{\cris}
}
\]
has determinant $1$.
Therefore, after replacing $\{ f^{*}(s_{\alpha}) \}$ by $\{ (h\circ f)^{*}(s_{\alpha}) \}$ and replacing $\Phi_f$ by $h \circ \Phi_f$, we may assume $\det(\psi)=1$.

Then there is an element $g \in \widetilde{G}(W)$ whose image under the surjection
$\widetilde{G}(W) \twoheadrightarrow \widetilde{G}_0(W)$
is $(\Phi_f \circ \Phi')^{-1} \in \widetilde{G}_0(W)$. 
After replacing $\Phi'$ by $\Phi'\circ g$, we may assume $\psi=\id$.

Under the isomorphism $\Phi_f$, the Frobenius on $\Mcris(H_{\mathrm{\acute{e}t}})[1/p]$ has of the form $b\sigma$ for some $b \in \widetilde{G}(K_0)$. 
Similarly, under the isomorphism $\Phi'$, the Frobenius on $H^1_{\cris}(\mathcal{A}_{s}/W)^{\vee}[1/p]$ has of the form $b'\sigma$ for some $b' \in \widetilde{G}(K_0)$.
Since $\psi=\id$, the elements $b$ and $b'$ have the same image under the surjection
$\widetilde{G}(K_0) \twoheadrightarrow \widetilde{G}_0(K_0)$.
Hence there is an element $u \in K^{\times}_0$ such that $b=ub'$.

We shall show $u \in W^{\times}$.
Since the Hodge-Tate weights of $H_{\mathrm{\acute{e}t}, \Q_p}$ are in $\{ 0, 1 \}$
(i.e.\ $\Gr^{i}(D_{\dR}(H_{\mathrm{\acute{e}t}, \Q_p})^{\vee})=0$ if $i \neq 0, 1$),
the effective Breuil-Kisin module $\M(H^{\vee}_{\mathrm{\acute{e}t}})$ is of height $\leq 1$.
Hence the cokernel of the Frobenius of $\Mcris(H_{\mathrm{\acute{e}t}})^{\vee}$ is killed by $p$.
Moreover, since $\mathcal{A}_{s}$ is an abelian variety of dimension $2^{21}$,
the cokernel of the Frobenius of 
$
H^1_{\cris}(\mathcal{A}_{s}/W)
$
is isomorphic to
$
(W/p)^{\oplus 2^{21}}
$
as a $W$-module.
From these facts, we conclude $u \in W^{\times}$.

Take an element $v \in W(\overline{\F}_q)^{\times}$ such that $\sigma(v)/v=u$.
Then the isomorphism 
\[
\Mcris(H_{\mathrm{\acute{e}t}})\otimes_{W}{W(\overline{\F}_q)} \overset{\Phi_f}{\longrightarrow} H_{W(\overline{\F}_q)} \overset{\times v}{\longrightarrow} H_{W(\overline{\F}_q)} \overset{\Phi'}{\longrightarrow} H^1_{\cris}(\mathcal{A}_{s}/W)^{\vee}\otimes_{W}W(\overline{\F}_q)
\]
is an isomorphism of $F$-isocrystals after inverting $p$.

The proof of Proposition \ref{Proposition:LiftingCrisLattice} is complete. 
\end{proof}

Recall that there is an equivalence of categories between
the category of $p$-divisible groups and
the category of $\mathrm{Gal}(\overline{E}/E)$-stable $\Z_p$-lattices in crystalline representations
whose Hodge-Tate weights are in $\{ 0, 1 \}$;
see \cite[Corollary 2.2.6]{KisinCris}, \cite[Theorem 2.2.1]{Liu13}.
Therefore, there is a (unique) $p$-divisible group $\mathcal{H}$ over $\O_E$ whose $p$-adic Tate module is
isomorphic to $H_{\mathrm{\acute{e}t}}$.

By Proposition \ref{Proposition:LiftingCrisLattice},
the base change $\mathcal{H}_{\O_K}$ is a lift of the $p$-divisible group $\mathcal{A}_{\overline{s}}[p^{\infty}]$ associated with
$\mathcal{A}_{\overline{s}}$.
As a corollary of Proposition \ref{Proposition:LiftingCrisLattice}, the $p$-divisible group $\mathcal{H}_{\O_K}$ is $\widetilde{G}$-adapted to $\mathcal{A}_{\overline{s}}[p^{\infty}]$
in the sense of \cite[Definition 3.3]{KimMadapusiIntModel}:

\begin{cor}
\label{Corollary:adapted}
The $p$-divisible group $\mathcal{H}_{\O_K}$ is \textit{$\widetilde{G}$-adapted} to $\mathcal{A}_{\overline{s}}[p^{\infty}]$.
\end{cor}

\begin{proof}
The assertion follows from Proposition
\ref{Proposition:LiftingCrisLattice}
and \cite[Theorem 2.5]{KimMadapusiIntModel}.
\end{proof}

\subsection{Proof of Theorem \ref{Theorem:CMliftingAbelianVariety}}\label{Subsection:ProofMainTheorem}
In this subsection, we shall complete the proof of Theorem \ref{Theorem:CMliftingAbelianVariety}.

By Corollary \ref{Corollary:adapted},
there is a lift $\widetilde{s} \in \mathscr{S}_{\widetilde{\mathrm{K}}}(\O_K)$ of $\overline{s} \in \mathscr{S}_{\widetilde{\mathrm{K}}}(\overline{\F}_q)$
such that the $p$-divisible group associated with $\mathcal{A}_{\widetilde{s}}$
is isomorphic to $\mathcal{H}_{\O_K}$.
See \cite[Lemma 3.8]{KimMadapusiIntModel} and the proof of \cite[Proposition 4.6]{KimMadapusiIntModel} for details.
Although the residue field is assumed to be finite in \cite[Lemma 3.8]{KimMadapusiIntModel}, the same result holds for $\O_K$ since the $p$-divisible group $\mathcal{A}_{\overline{s}}[p^{\infty}]$ is defined over the finite field $k$.

By Proposition \ref{Proposition:LiftingCrisLattice} (4), the $0$-th piece of the Hodge filtration on
$H^1_{\cris}(\mathcal{A}_s/W)^{\vee}\otimes_{W}K$
coincides with $\Fil^0(H^1_{\cris}(\mathcal{A}_s/W)^{\vee}\otimes_{W}K)$ defined in Section \ref{Subsection:Liftings of points with additional properties}.
        
To prove Theorem \ref{Theorem:CMliftingAbelianVariety}, it remains to show $\widetilde{s} \in \mathscr{S}_{\widetilde{\mathrm{K}}}(\O_K)$ lifts to an
$\O_K$-valued point of
$ Z_{\mathrm{K}^p}(\Lambda)$.
By the construction of $H_{\mathrm{\acute{e}t}}$,
we have a $\Gal(\overline{K}/K)$-equivariant homomorphism
\[
\Lambda_{\Z_p}\subset\widetilde{L}_p \to \End_{\Z_p}(H_{\mathrm{\acute{e}t}}).
\]
Therefore, we get
\[ \Lambda_{\Z_{(p)}} \to \End_{\O_K}(\mathcal{H}_{\O_K}) \]
lifting a homomorphism
\[ \Lambda_{\Z_{(p)}} \to \End_{\overline{\F}_q}(\mathcal{A}_{\overline{s}}[p^{\infty}])
\]
induced by $\iota$.
By the Serre-Tate theorem, the $\Lambda$-structure $\iota$ lifts to a homomorphism
\[ \Lambda_{\Z_{(p)}} \to \End_{\O_K}(\mathcal{A}_{\widetilde{s}})_{\Z_{(p)}}. \]

The proof of Theorem \ref{Theorem:CMliftingAbelianVariety} is complete.
\hfill $\qed$

\section{Kisin's algebraic groups associated with Kuga-Satake abelian varieties over finite fields} \label{Section:Kisin's algebraic groups}

In this section, we attach an algebraic group $I$ over $\Q$ to
a quasi-polarized $K3$ surface of finite height over $\overline{\F}_q$.
It is a subgroup of the multiplicative group of the endomorphism algebra of the Kuga-Satake abelian variety.
Then we study its action on the formal Brauer group of
the $K3$ surface.

\subsection{Kisin's algebraic groups}\label{Subsection:Kisin's algebraic groups}

Let $s \in Z_{\mathrm{K}^p}(\F_q)$ be an $\F_q$-valued point and $\overline{s}\in Z_{\mathrm{K}^p}(\overline{\F}_q)$ a geometric point above $s$.
Let $\Aut_{\Q}(\mathcal{A}_{\overline{s}})$ be the algebraic group over $\Q$ defined by
\[
\Aut_{\Q}(\mathcal{A}_{\overline{s}})(R):= (\End_{\overline{\F}_q}(\mathcal{A}_{\overline{s}}) \otimes_{\Z} R)^{\times}
\]
for every $\Q$-algebra $R$.

After replacing $\F_q$ by a finite extension of it, we may assume that all endomorphisms of $\mathcal{A}_{\overline{s}}$ are defined over $\F_q$.
Namely we have
\[
\End_{\F_q}(\mathcal{A}_{s}) \otimes_{\Z} \Q \cong
\End_{\overline{\F}_q}(\mathcal{A}_{\overline{s}}) \otimes_{\Z} \Q.
\]

The global sections $\{ s^p_{\alpha} \}$ of ${V}^p(\mathcal{A})^{\otimes}$ give rise to global sections $\{ s_{\alpha, \ell} \}$ of $V_{\ell}(\mathcal{A})^{\otimes}$.
The $\A^p_f$-local system
$\widetilde{\mathbb{V}}^p$ and the embedding $i^p$
determine a $\Q_{\ell}$-local system
$\widetilde{\mathbb{V}}_{\ell}$ and an embedding
$
\widetilde{\mathbb{V}}_{\ell} \hookrightarrow \underline{\End}(V_{\ell}(\mathcal{A}))
$
of $\Q_{\ell}$-local systems for every $\ell \neq p$.
The homomorphism $\iota^p$ induced by the $\Lambda$-structure for $s \in Z_{\mathrm{K}^p}(\F_q)$
determines a homomorphism
$
\iota_{\ell} \colon \Lambda_{\Q_{\ell}} \to \widetilde{\mathbb{V}}_{\ell, \overline{s}}.
$

We fix an isomorphism 
$
H_{\Q_{\ell}} \cong V_{\ell}(\mathcal{A}_{\overline{s}})
$
of $\Q_{\ell}$-vector spaces which carries $\{ s_{\alpha}\}$ to $\{ s_{\alpha, \ell, \overline{s}} \}$ and induces the following commutative diagram:
\[
\xymatrix{\Lambda_{\Q_{\ell}}  \ar^-{}[r] \ar_-{\iota_{\ell}}[rd] &
\widetilde{L}_{\Q_{\ell}} \ar[d]^-{\cong} \ar[r]^-{i} & \End_{\Q_{\ell}}(H_{\Q_{\ell}}) \ar[d]_-{\cong} \\
& \widetilde{\mathbb{V}}_{\ell, \overline{s}} \ar[r]^-{i_{\ell}} & \End_{\Q_{\ell}}(V_{\ell}(\mathcal{A}_{\overline{s}})),
}
\]
where $\widetilde{L}_{\Q_{\ell}} \cong \widetilde{\mathbb{V}}_{\ell, \overline{s}}$ is an isometry over $\Q_{\ell}$. 

Let
$
H_{W} \cong H^1_{\cris}(\mathcal{A}_{s}/W)^{\vee}
$
be an isomorphism as in Section \ref{Subsection:Hodge tensors}.
After inverting $p$ and composing an element of $\widetilde{G}(W[1/p])$,
we can find an isomorphism
\[ H_{W[1/p]} \cong H^1_{\cris}(\mathcal{A}_{s}/W)^{\vee}[1/p] \]
which carries $\{ s_{\alpha}\}$ to $\{ s_{\alpha, \cris, s} \}$
and induces the same diagram as above by \cite[Lemma 2.8]{MadapusiIntModel} and the fact that every $\GSpin(L_{W[1/p]})$-torsor over $W[1/p]$ is trivial (see \cite[Theorem 6.4]{Platonov-Rapinchuk} and the arguments in Section \ref{Subsection:p-divisible groups adapted to general spin groups}).
We fix such an isomorphism.

Kisin introduced an algebraic group
$\widetilde{I}_{\ell}$ over $\Q_{\ell}$
for every prime number $\ell$ (including $\ell = p$),
and an algebraic group $\widetilde{I}$ over $\Q$
as follows; see \cite[(2.1.2)]{KisinModp} for details.
\begin{enumerate}
\item For a prime number $\ell \neq p$,
let $\Frob_q \in \End_{\Q_{\ell}}(H_{\Q_{\ell}})$ 
be the $\Q_{\ell}$-endomorphism of $H_{\Q_{\ell}}$ induced by the $q$-th power Frobenius of $\mathcal{A}_{\overline{s}}$.
Since $\Frob_q$ fixes the tensors $\{ s_{\alpha, \ell, \overline{s}} \}$, we have
$
\Frob_q \in \widetilde{G}(\Q_{\ell}).
$
For every integer $m \geq 1$, we define an algebraic $\Q_{\ell}$-subgroup
$
\widetilde{I}_{\ell, m}
$
of $\widetilde{G}_{\Q_{\ell}}$ by
\[ \widetilde{I}_{\ell, m}(R) := \{\, g \in \widetilde{G}(R) \mid g \Frob^{m}_q = \Frob^{m}_q g \,\} \]
for every $\Q_{\ell}$-algebra $R$.
For sufficiently divisible $m \geq 1$, 
the algebraic group $\widetilde{I}_{\ell, m}$ does not depend on $m$,
and we denote it by $\widetilde{I}_{\ell}$. 

\item For $\ell = p$, we define an algebraic group $\widetilde{I}_{p, m}$ over $\Q_p$ by 
\[
\widetilde{I}_{p, m}(R) := \{\, g \in \widetilde{G}(R\otimes_{\Q_p}W(\F_{q^m})[1/p]) \mid g \varphi = \varphi g \,\}.
\]
For sufficiently divisible $m \geq 1$, 
the algebraic group $\widetilde{I}_{p, m}$ does not depend on $m$,
and we denote it by $\widetilde{I}_{p}$.

\item Let
$\widetilde{I} \subset \Aut_{\Q}(\mathcal{A}_{\overline{s}})$
be the largest closed $\Q$-subgroup of 
$\Aut_{\Q}(\mathcal{A}_{\overline{s}})$
mapped into $\widetilde{I}_{\ell}$ for every $\ell$ (including $\ell=p$).
\end{enumerate}
Replacing $\F_q$ be a finite extension of it, we may assume $\widetilde{I}_{\ell, 1}=\widetilde{I}_{\ell}$ and
$\widetilde{I}_{p, 1}=\widetilde{I}_{p}$.

Kisin proved that the natural map
\[
\widetilde{I}_{\Q_{\ell}} \to \widetilde{I}_{\ell}
\]
is an isomorphism of algebraic groups over $\Q_{\ell}$
for every $\ell$ (including $\ell=p$)
\cite[Corollary 2.3.2]{KisinModp}.

For our purpose, we need
an algebraic subgroup $I \subset \widetilde{I}$ over $\Q$ defined using the $\Lambda$-structure; see Definition \ref{Definition:Lambda structure}.
If $L$ is self-dual at $p$, it coincides with Kisin's algebraic group associated with an $\F_q$-valued point of the integral canonical model of $\mathrm{Sh}_{\mathrm{K}_0}(\SO(L_{\Q}), X_{L})$ taken in a similar way as in Section \ref{Subsection:$F$-crystals on Shimura varieties and the cohomology of K3 surfaces}.

\begin{defn}
\label{Definition: algebraic group I}
\begin{enumerate}
\item Let 
$I \subset \widetilde{I}$
be the algebraic subgroup over $\Q$ defined by
\[
I(R):=\{\, g \in \widetilde{I}(R) \mid ghg^{-1} = h \,\, {\rm{in}} \, \End_{\overline{\F}_q}(\mathcal{A}_{\overline{s}})_R \,\, {\rm{for \,\, every}} \, h \in \iota(\Lambda_{\Q})  \,\}
\]
for every $\Q$-algebra $R$.

\item For a prime number $\ell \neq p$, let 
$I_{\ell} \subset \widetilde{I}_{\ell}$
be the algebraic subgroup over $\Q_{\ell}$ defined by
\[
I_{\ell}(R):=\{\, g \in \widetilde{I}_{\ell}(R) \mid ghg^{-1} = h \,\, {\rm{in}} \, \End_{\Frob_{q}}(H_{\Q_{\ell}})_R \,\, {\rm{for \,\, every}} \, h \in i(\Lambda_{\Q_{\ell}})  \,\}
\]
for every $\Q_{\ell}$-algebra $R$.

\item For $\ell = p$, let 
$I_{p} \subset \widetilde{I}_p$
be the algebraic subgroup over $\Q_p$ defined in a similar way as above.
\end{enumerate}
\end{defn}

As in Kisin's paper \cite{KisinModp},
we shall prove that the natural map
\[
I_{\Q_{\ell}} \to I_{\ell}
\]
is an isomorphism of algebraic groups over $\Q_{\ell}$
for every $\ell$ (including $\ell=p$).
Here we prove it for some $\ell \neq p$.
The case of general $\ell$ will be proved later;
see Corollary \ref{Corollary:KisinGroup(II)}.

\begin{prop}\label{Proposition:KisinGroup(I)}
\begin{enumerate}
\item For some prime number $\ell \neq p$,
the natural map
\[
I_{\Q_{\ell}} \to I_{\ell}
\]
is an isomorphism of algebraic groups over $\Q_{\ell}$.

\item The algebraic groups $I$ and $\GSpin(L_{\Q})$
over $\Q$ have the same rank.
(Recall that the rank of an algebraic group over a field $k$ is the dimension of a maximal $k$-torus of it.)
\end{enumerate}
\end{prop}

\begin{proof}
$(1)$ We fix a prime number $\ell \neq p$ such that
$\GSpin(L_{\Q})$ and $\widetilde{G}_{\Q}$ are split at $\ell$, and
all the eigenvalues of $\Frob_q$ acting on  $H_{\Q_{\ell}}$
are contained in $\Q_{\ell}$. 
We shall show the assertion $(1)$ holds for such $\ell$.
By the proof of \cite[Corollary 2.1.7]{KisinModp}, 
the homomorphism 
\[
\widetilde{I}_{\Q_{\ell}} \to \widetilde{I}_{\ell} 
\]
is an isomorphism. 
(Precisely, Kisin proved it in \cite{KisinModp} assuming $p \geq 3$
and the restriction of $\psi$ to $H_{\Z_{(p)}}$ is perfect.
These assumptions are unnecessary; see the proof of \cite[Theorem A.8]{KimMadapusiIntModel}.)
By Tate's theorem, we have 
\[
\End_{\F_q}(\mathcal{A}_{s})_{\Q_{\ell}}
\cong \End_{\Frob_{q}}(H_{\Q_{\ell}}).
\] 
Now, the assertion (1) follows from 
the definitions of $I$ and $I_{\ell}$.

$(2)$
We follow Kisin's proof of \cite[Corollary 2.1.7]{KisinModp}.
Since $\Frob_q$ and $I_{\ell}$
act trivially on $i(\Lambda_{\Q_{\ell}})$,
we have $\Frob_q \in \GSpin(L_{\Q_{\ell}})$ and
$I_{\ell} \subset \GSpin(L_{\Q_{\ell}})$; see \cite[(2.6.1)]{MadapusiIntModel}.
The element $\Frob_q \in \GSpin(L_{\Q_{\ell}})$ is semisimple since the action of $\Frob_q$ on $V_{\ell}(\mathcal{A}_{\overline{s}})$ is semisimple.
Thus the connected component $S$ of the Zariski closure of the group
$\langle \Frob_q \rangle$ generated by $\Frob_q$
is a split torus in $\GSpin(L_{\Q_{\ell}})$ by the hypotheses on $\ell$. 
Since $I_{\ell}$ is the same as the centralizer of $\Frob^m_q$
in $\GSpin(L_{\Q_{\ell}})$ for a sufficiently divisible $m$, 
it follows that $I_{\ell}$ coincides with the centralizer of $S$.
Therefore, $I_{\ell}$ contains a split maximal torus of $\GSpin(L_{\Q_{\ell}})$.
(Hence $I_{\ell}$ is a connected split reductive group over $\Q_{\ell}$.)
In particular, the rank of $I_{\ell}$ as an algebraic group over $\Q_{\ell}$ is equal to the rank of  $\GSpin(L_{\Q_{\ell}})$ as an algebraic group over $\Q_{\ell}$.
Since we have $I_{\Q_{\ell}} \cong I_{\ell}$, 
the ranks of the algebraic groups $I$ and $\GSpin(L_{\Q})$
over $\Q$ are equal.
\end{proof}

\subsection{The action of Kisin's groups on the formal Brauer groups of $K3$ surfaces}\label{Subsection:The action of Kisin's group on the formal Brauer groups}

We consider the situation as in Section \ref{Section:Lifting of a special point} and keep the notation.
In particular, we assume the height $h$ of $X$ is finite.

As in the proof of Proposition \ref{Proposition:CrysDecomposition}, we have a decomposition
\[
P^2_{\cris}(X/W) \cong L_{1-1/h} \oplus L_{1} \oplus L_{1+1/h}.
\]
Here, $L_{\lambda}$ is an $F$-crystal over $W$ which has a single slope $\lambda$
for each $\lambda \in \{ 1-1/h,\, 1,\, 1+1/h \}$.
Moreover, there is an isomorphism of $F$-crystals over $W$:
	\[
	\mathbb{D}(\widehat{\Br})
	\cong L_{1+1/h}(1).
	\]

By Proposition \ref{Proposition:CrisOrthogonal}, the algebraic group $I$ acts on $P^2_{\cris}(X/W)(1)[1/p]$.
Hence we have the following homomorphism of algebraic groups over $\Q_p$:
\[
I_{\Q_p} \to \Res_{K_0/\Q_p}\GL(P^2_{\cris}(X/W)(1)[1/p]).
\]

For a one-dimensional smooth formal group $\mathscr{G}$ over a ring $A$,
let $\Aut_{\Q_p}(\mathscr{G})$ be the $\Q_p$-group such that
\[
\Aut_{\Q_p}(\mathscr{G})(R):=(\End_{A}(\mathscr{G})\otimes_{\Z_p}R)^{\times}
\]
for every $\Q_p$-algebra $R$.

\begin{lem}\label{Lemma:ActionFormalBrauerGroup}
There is a homomorphism
\[
I_{\Q_p} \to (\Aut_{\Q_p}(\widehat{\Br}))^{\mathrm{op}}
\]
which is compatible with
\[
I_{\Q_p} \to \Res_{K_0/\Q_p}\GL(P^2_{\cris}(X/W)(1)[1/p])
\]
via the composite
$
P^2_{\cris}(X/W)(1) \to  L_{1+1/h}(1) \cong \mathbb{D}(\widehat{\Br}),
$
where the first map is the projection.
\end{lem}
\begin{proof}
We put 
$V_{\lambda}:=L_{1+\lambda}(1)[1/p]$
for each $\lambda \in \{ -1/h,\, 0,\, 1/h \}$.
For an $F$-isocrystal $M$ over $K_0$,
	let $\GL_{\varphi}(M)$ be the algebraic group over $\Q_p$
	defined by
	\[
	\GL_{\varphi}(M)(R):=\{\, g \in \GL_{K_0\otimes_{\Q_p} R}(M\otimes_{\Q_p} R) \mid g \varphi = \varphi g \,\}
	\]
	for every $\Q_p$-algebra $R$, where $\varphi$ is the Frobenius of $M$.
We have an isomorphism of algebraic groups over $\Q_p$:
	\[
	\GL_{\varphi}(P^2_{\cris}(X/W)(1)[1/p])\cong \GL_{\varphi}(V_{-1/h})\times \GL_{\varphi}(V_{0}) \times \GL_{\varphi}(V_{1/h}).
	\]
Let
\[ I_{\Q_p} \to \GL_{\varphi}(V_{1/h}). \]
be the composite of the map $I_{\Q_p} \to  \GL_{\varphi}(P^2_{\cris}(X/W)(1)[1/p])$
with the projection $\GL_{\varphi}(P^2_{\cris}(X/W)(1)[1/p])\to \GL_{\varphi}(V_{1/h})$.

Note that we have an isomorphism
	\[
	(\Aut_{\Q_p}(\widehat{\Br}))^{\mathrm{op}} \cong \GL_{\varphi}(V_{1/h}).
	\]
Hence we have a homomorphism of algebraic groups over $\Q_p$
	\[
	I_{\Q_p} \to (\Aut_{\Q_p}(\widehat{\Br}))^{\mathrm{op}}.
	\]
\end{proof}

\section{Lifting of $K3$ surfaces over finite fields with actions of tori}
\label{Section:Lifting of K3 surfaces over finite fields with actions of tori}

\subsection{$K3$ surfaces with complex multiplication}\label{Subsection:K3 surfaces with complex multiplication}

In this subsection, we recall the definition and basic properties of
$K3$ surfaces with complex multiplication over $\C$.

Let $Y$ be a projective $K3$ surface over $\C$.
Let
\[
T_{Y}:=\Pic(Y)^{\perp}_{\Q} \subset H^{2}_B(Y, \Q(1))\]
be the transcendental part of the singular cohomology, which has the $\Q$-Hodge structure coming from $H^{2}(Y, \Q(1))$.

Let
\[
E_{Y}:={\rm{End}}_{\rm{Hdg}}(T_{Y})
\]
be the $\Q$-algebra of $\Q$-linear endomorphisms on $T_{Y}$ preserving the $\Q$-Hodge structure on it. 
We say $Y$ has \textit{complex multiplication (CM)} if $E_{Y}$ is a CM field and ${\rm{dim}}_{E_{Y}}(T_{Y})=1$.
Here a number field is called CM if it is a purely imaginary quadratic extension of a totally real number field.

Let ${\mathrm{MT}}(T_{Y})$ be the \textit{Mumford-Tate group} of $T_{Y}$.
By the definition, it is the smallest algebraic $\Q$-subgroup  of
$\SO(T_{Y})$ such that $h_Y(\mathbb{S}(\R)) \subset {\mathrm{MT}}(T_{Y})(\R)$,
where 
\[
h_Y \colon \mathbb{S} \to \SO(T_{Y})_{\R}
\]
is the homomorphism over $\R$ corresponding to the $\Q$-Hodge structure of $T_Y$.
By the results of Zarhin \cite[Section 2]{Zarhin}, the $K3$ surface $Y$ has CM if and only if the Mumford-Tate group
${\mathrm{MT}}(T_{Y})$ is commutative; see \cite[Proposition 8]{Taelman} for example.

In the rest of this subsection, we fix a $\C$-valued point $t \in M^{\mathrm{sm}}_{2d, \K^p_0, \Q}(\C)$.
Let $(Y, \xi)$ be a quasi-polarized $K3$ surface over $\C$ associated with $t$.
The image of $t$ under the Kuga-Satake morphism $\mathrm{KS}$
is also denoted by $t \in Z_{\K^p_0}(\Lambda)(\C)$.

\begin{prop}\label{Proposition:CM K3 is defined over a number field}
	Assume $Y$ is a $K3$ surface with CM over $\C$.
	Then there exist a number field $F \subset \C$ and an $F$-valued point $t_0 \in M_{2d, \K^p_0, \Q}(F)$ such that the morphism $t \colon \Spec \C \to Z_{\K^p_0}(\Lambda)$ factors through the image of $t_0$ under $\KS$.
\end{prop}
\begin{proof}
This proposition follows from Rizov's result \cite[Corollary 3.9.4]{Rizov10}
as follows.
The image of $t \in Z_{\K^p_0}(\Lambda)(\C)$ under the morphism
$Z_{\K^p_0}(\Lambda) \to \mathscr{S}_{\widetilde{\mathrm{K}}_0}$
is denoted by the same symbol $t$.
If $Y$ has CM, then the residue field of the image $t \in \mathscr{S}_{\widetilde{\mathrm{K}}_0}(\C)$ is a number field by the definition of the canonical model
$\mathscr{S}_{\widetilde{\mathrm{K}}_0, \Q}= \mathrm{Sh}_{\widetilde{\mathrm{K}}_0}$
over $\Q$.
Since the morphism
$Z_{\K^p_0}(\Lambda) \to \mathscr{S}_{\widetilde{\mathrm{K}}_0}$
is finite by Proposition \ref{Proposition:LambdaFiniteEtale}, the residue field of $t \in Z_{\K^p_0}(\Lambda)(\C)$ is a number field.
Now, the assertion follows from the \'etaleness of the Kuga-Satake morphism $\mathrm{KS}$ (in characteristic $0$).
\end{proof}

\begin{rem}
\label{Remark:CM K3 is defined over a number field}
Pjatecki{\u\i}-{\v{S}}apiro and {\v{S}}afarevi{\v{c}} also showed every $K3$ surface with CM is defined over a number field; see \cite[Theorem 4]{Shafarevich2}.
\end{rem}

For the quasi-polarized $K3$ surface $(Y, \xi)$ over $\C$, the primitive singular cohomology is defined by
\[
P^2_{B}(Y, \Q(1)):= \mathrm{ch}_B(\xi)^{\perp} \subset H^2_{B}(Y, \Q(1)).
\]

We fix a $\C$-valued point of $Z_{\K^p}(\Lambda)$ mapped to $t$, and also denote it by $t \in Z_{\K^p}(\Lambda)(\C)$.
We have the Kuga-Satake abelian variety $\mathcal{A}_t$ over $\C$ corresponding to $t \in Z_{\K^p}(\Lambda)(\C)$.
As in Section \ref{Subsection:$F$-crystals on Shimura varieties and the cohomology of K3 surfaces}, the stalk
\[
\widetilde{V}_{t}:=\widetilde{\mathbb{V}}_{B, t}
\]
satisfies the following properties:
\begin{enumerate}
	\item $\widetilde{V}_{t}$ admits a perfect bilinear form $( \ , \ )$ over $\Q$ which is a quasi-polarization.
	\item There is a homomorphism $\iota_B \colon \Lambda_{\Q} \to \widetilde{V}_{t}$ preserving the bilinear forms and the $\Q$-Hodge structures.
	\item There is an isometry over $\Q$
     \[
    P^2_{B}(Y,\Q(1)) \cong \iota_B(\Lambda_{\Q})^{\perp}
    \]
    preserving the $\Q$-Hodge structures.
    \item The following diagram commutes:
    \[
    \xymatrix{&\GSpin(\widetilde{V}_{t})_{\R}  \ar[d]^-{} \ar[r]^-{} & \GL(H^1_B(\mathcal{A}_t, \Q)^{\vee})_{\R} \\
\mathbb{S} \ar[r]_-{h_0} \ar[ru]^-{h} &\SO(\widetilde{V}_{t})_{\R}, 
    }
    \]
    where $h_0$ is the homomorphism of algebraic groups over $\R$ corresponding to the $\Q$-Hodge structure on $\widetilde{V}_{t}$ and the composite of the following homomorphisms
    \[ \mathbb{S} \overset{h}{\to} \GSpin(\widetilde{V}_{t})_{\R}  \to \GL(H^1_B(\mathcal{A}_t, \Q)^{\vee})_{\R}
    \]
    is the homomorphism corresponding to the $\Q$-Hodge structure on $H^1_B(\mathcal{A}_t, \Q)^{\vee}$.
\end{enumerate}

\begin{prop}\label{Proposition:CM K3 and CM AV}
	The $K3$ surface $Y$ has CM if and only if the Kuga-Satake abelian variety $\mathcal{A}_t$ has CM.
\end{prop}
\begin{proof}
    This proposition was essentially proved by Tretkoff; see \cite[Corollary 3.2]{Tretkoff}.
    We give an argument from the point of view of algebraic groups.
	Since $h_0$ is the composite of $h_Y$ with the following inclusions
	\[
	\SO(T_{Y})_{\R} \hookrightarrow \SO(P^2_{B}(Y,\Q(1)))_{\R} \hookrightarrow \SO(\widetilde{V}_{t})_{\R},
	\]
	we have
	$\mathrm{MT}(T_{Y})\cong\mathrm{MT}(\widetilde{V}_{t})$.
	We shall show $\mathrm{MT}(\widetilde{V}_{t})$ is commutative if and only if $\mathrm{MT}(H^1_B(\mathcal{A}_t, \Q)^{\vee})$ is commutative.
	Since $\mathrm{MT}(H^1_B(\mathcal{A}_t, \Q)^{\vee})$ is contained in $\GSpin(\widetilde{V}_{t})$ and $\mathrm{MT}(\widetilde{V}_{t})$ is contained in the image of $\mathrm{MT}(H^1_B(\mathcal{A}_t, \Q)^{\vee})$, we only have to show $\mathrm{MT}(H^1_B(\mathcal{A}_t, \Q)^{\vee})$ is commutative if $\mathrm{MT}(\widetilde{V}_{t})$ is so.
	
	We assume $\mathrm{MT}(\widetilde{V}_{t})$ is commutative.
	Then the inverse image of $\mathrm{MT}(\widetilde{V}_{t})$ under the homomorphism $\GSpin(\widetilde{V}_{t}) \to \SO(\widetilde{V}_{t})$ is a solvable algebraic group and contains $\mathrm{MT}(H^1_B(\mathcal{A}_t, \Q)^{\vee})$.
	Since $\mathrm{MT}(H^1_B(\mathcal{A}_t, \Q)^{\vee})$ is a reductive group (see \cite[Proposition 3.6]{Deligne900}),
	it is commutative.
\end{proof}

\begin{cor}\label{Corollary:CM for every embedding}
Let $F$ be a field which can be embedded in $\C$.
Let $Z$ be a $K3$ surface over $F$.
 If $Z \otimes_{F, j} \C$ has CM for an embedding $j \colon F \hookrightarrow \C$, then $Z \otimes_{F, j'} \C$ has CM for every embedding $j' \colon F \hookrightarrow \C$.
\end{cor}

\begin{proof}
The assertion follows from Proposition \ref{Proposition:CM K3 and CM AV} and the fact that, for an abelian variety $A$ over $\C$ with CM and every automorphism $f \colon \C \cong \C$, the abelian variety $A \otimes_{\C, f}\C$ has CM.
\end{proof}

\begin{rem}
Let $F$ be a field of characteristic $0$ which can be embedded into $\C$,
and $Z$ a $K3$ surface over $F$.
We say $Z$ has CM if $Z \otimes_{F, j}{\C}$ has CM for
some embedding $j \colon F \hookrightarrow \C$,
in which case $Z \otimes_{F, j'}{\C}$ has CM for \textit{every} embedding $j' \colon F \hookrightarrow \C$ by Corollary \ref{Corollary:CM for every embedding}.
\end{rem}

\subsection{A lemma on liftings of formal groups with action of tori}
\label{Subsection:A lemma on liftings of formal groups with action of tori}

The following result on characteristic $0$ liftings of one-dimensional smooth formal groups is presumably well-known. (For the definition of $\Aut_{\Q_p}(\mathcal{G}_0)$ and $\Aut_{\Q_p}(\mathcal{G})$, see Section \ref{Subsection:The action of Kisin's group on the formal Brauer groups}.)

\begin{lem}\label{Lemma:FormalGroupLifting}
	Let $\mathcal{G}_0$ be a one-dimensional smooth formal group over $\F_q$. 
	Let $T_p$ be an algebraic torus over $\Q_p$ and 
	\[
	\rho \colon T_p \to \Aut_{\Q_p}(\mathcal{G}_0)
	\]
	a homomorphism of algebraic groups over $\Q_p$. 
	Assume that the height of $\mathcal{G}_0$ is finite,
	and the Frobenius $\Phi$ of $\mathcal{G}_0$ over $\F_q$ is contained in $\rho(T_p(\Q_p))$.
	Then, there exist a finite totally ramified extension $E$ of $K_0$ and a smooth formal group $\mathcal{G}$ over $\O_{E}$ satisfying the following properties:
	\begin{enumerate}
		\item The special fiber of $\mathcal{G}$ is isomorphic to $\mathcal{G}_0$.
		\item The homomorphism $\rho$ factors as 
	\[
	T_p \to \Aut_{\Q_p}(\mathcal{G})\to \Aut_{\Q_p}(\mathcal{G}_0).
	\]
	\end{enumerate}
\end{lem}
\begin{proof}
We fix an isomorphism $\mathcal{G}_0 \cong \Spf \F_q[[x]]$ and consider $\mathcal{G}_0$ as a formal group law in $\F_q[[x, y]]$.

	The composite of the following homomorphism
	\[
T_p \overset{\rho}{\longrightarrow}\Aut_{\Q_p}(\mathcal{G}_0) \longrightarrow \Aut_{\Q_p}(\mathcal{G}_{0, \overline{\F}_q})
	\]
	is also denoted by $\rho$.
	Take a maximal $\Q_p$-torus $T'_p$ of $\Aut_{\Q_p}(\mathcal{G}_{0, \overline{\F}_q})$ containing $\rho(T_p)$.
	
	It is well-known that $\End_{\overline{\F}_q}(\mathcal{G}_{0, \overline{\F}_q})\otimes_{\Z_p}\Q_p$ is a central division algebra over $\Q_p$ and $\End_{\overline{\F}_q}(\mathcal{G}_{0, \overline{\F}_q})$ is the maximal order of it; see \cite[Corollary 20.2.14]{Hazewinkel}. 
	Hence, there is a maximal commutative $\Q_p$-subalgebra
	\[
	K' \subset \End_{\overline{\F}_q}(\mathcal{G}_{0, \overline{\F}_q})\otimes_{\Z_p}\Q_p
	\]
	such that $K'^{\times}=T'_p$ as algebraic groups over $\Q_p$, and we have $\O_{K'} \subset \End_{\overline{\F}_q}(\mathcal{G}_{0, \overline{\F}_q})$.
	Moreover, since $T'_p(\Q_p)$ contains the Frobenius $\Phi$ of $\mathcal{G}_0$ over $\F_q$, the endomorphisms in $K'$ commute with $\Phi$.
	Hence, we have $\O_{K'} \subset \End_{\F_q}(\mathcal{G}_0)$.
	
	We regard the formal group law $\mathcal{G}_0$ as a formal $\O_{K'}$-module over $\F_q$ in the sense of \cite[(18.6.1)]{Hazewinkel}.
	The universal formal $\O_{K'}$-module $\mathcal{G}^{\mathrm{univ}}$ exists and it is a formal $\O_{K'}$-group over a polynomial ring $\O_{K'}[(S_i)_{i \in \N}]$ with infinitely many variables over $\O_{K'}$; see \cite[(21.4.8)]{Hazewinkel}.
	(Note that $\mathcal{G}^{\mathrm{univ}}$ does not classify isomorphism classes of formal $\O_{K'}$-modules, but formal $\O_{K'}$-modules.)
	
	We take a finite totally ramified extension $E$ of $K_0$ such that $E$ is a $K'$-algebra.
	Then there is a formal $\O_{K'}$-module $\mathcal{G} \in \O_E[[x, y]]$ over $\O_{E}$ whose reduction modulo the maximal ideal of $\O_E$ is equal to $\mathcal{G}_{0}$ such that the homomorphism $\O_{K'} \to \End_{\F_q}(\mathcal{G}_{0})$ factors as
	\[
	\O_{K'} \to \End_{\O_{E}}(\mathcal{G}) \to \End_{\F_q}(\mathcal{G}_{0}).
	\]
	Therefore, the composite of $\rho$ with the inclusion
	$
	\Aut_{\Q_p}(\mathcal{G}_0) \to \Aut_{\Q_p}(\mathcal{G}_{0})
	$
	factors as 
	\[
	T_p \to \Aut_{\Q_p}(\mathcal{G}) \to \Aut_{\Q_p}(\mathcal{G}_{0}).
	\]
\end{proof}

\subsection{Liftings of $K3$ surfaces over finite fields with actions of tori}\label{Subsection:Liftings of K3 surfaces over finite fields with actions of tori}

Let $(Y, \xi)$ be a quasi-polarized $K3$ surface over a field $k$ of characteristic $0$ or $p$.
Let $\overline{k}$ be an algebraic closure of $k$.

For a prime number $\ell \neq p$, the primitive part of the $\ell$-adic cohomology is denoted by
\[
P^2_{\mathrm{\acute{e}t}}(Y_{\overline{k}},\Q_{\ell}(1))
:= \mathrm{ch}_{\ell}(\xi)^{\perp} \subset H^2_{\mathrm{\acute{e}t}}(Y_{\overline{k}},\Q_{\ell}(1)).
\]
It is equipped with a canonical action of $\Gal(\overline{k}/k)$.

When $\overline{k}$ is a subfield of $\C$, we have a canonical isomorphism
\[
P^2_{\mathrm{\acute{e}t}}(Y_{\overline{k}},\Q_{\ell}(1))
\cong
P^2_{B}(Y_{\C}, \Q(1)) \otimes_{\Q} \Q_{\ell},
\]
where $P^2_{B}(Y_{\C}, \Q(1))$ denotes the primitive part of the singular cohomology of $Y_{\C}$.

We consider the situation as in Section \ref{Subsection:$F$-crystals on Shimura varieties and the cohomology of K3 surfaces} and Section \ref{Section:Lifting of a special point}, and keep the notation.
In particular, $(X, \mathscr{L})$ is a quasi-polarized $K3$ surface over $\F_q$.
We attach the algebraic group $I$ over $\Q$ to the $\F_q$-valued point $s \in Z_{\mathrm{K}^p}(\F_q)$;
see Definition \ref{Definition: algebraic group I}.

The following theorem concerns characteristic $0$ liftings of
$K3$ surfaces of finite height over finite extensions of $W(\overline{\F}_q)[1/p]$.
Since every $K3$ surface with CM is defined
over a number field (see 
Proposition \ref{Proposition:CM K3 is defined over a number field} and
Remark \ref{Remark:CM K3 is defined over a number field}),
Theorem \ref{Theorem:LiftingTorusAction:Restate}
implies Theorem \ref{Theorem:CMlifting} in Introduction; see Corollary \ref{Corollary:CMLifting}.

\begin{thm}
\label{Theorem:LiftingTorusAction:Restate}
Let $T \subset I$ be a maximal torus over $\Q$.
Then there exist a finite extension $K$ of $W(\overline{\F}_q)[1/p]$
and a quasi-polarized $K3$ surface $(\mathcal{X},  \mathcal{L})$ over $\O_K$
such that the special fiber
$(\mathcal{X}_{\overline{\F}_q}, \mathcal{L}_{\overline{\F}_q})$
is isomorphic to $(X_{\overline{\F}_q}, \mathscr{L}_{\overline{\F}_q})$,
and, for every embedding $K \hookrightarrow \C$,
the quasi-polarized $K3$ surface
$(\mathcal{X}_{\C},  \mathcal{L}_{\C})$
satisfies the following properties:
\begin{enumerate}
\item The $K3$ surface $\mathcal{X}_{\C}$ has CM.

\item There is a homomorphism of algebraic groups over $\Q$
\[ T \to \SO(P^2_{B}(\mathcal{X}_{\C}, \Q(1))). \]

\item
For every $\ell \neq p$,
the action of $T(\Q_{\ell})$ on
$P^2_{B}(\mathcal{X}_{\C},\,\Q(1)) \otimes_{\Q} \Q_{\ell}$
is identified with the action of $T(\Q_{\ell})$ on
$P^2_{\mathrm{\acute{e}t}}(X_{\overline{\F}_q},\Q_{\ell}(1))$
via the canonical isomorphisms
\[
P^2_{B}(\mathcal{X}_{\C}, \Q(1)) \otimes_{\Q} \Q_{\ell}
\cong P^2_{\mathrm{\acute{e}t}}(\mathcal{X}_{\C}, \Q_{\ell}(1))
\cong P^2_{\mathrm{\acute{e}t}}(X_{\overline{\F}_q},\Q_{\ell}(1)))
\]
(using the embedding $K \hookrightarrow \C$,
we consider $K$ as a subfield of $\C$).

\item The action of every element of $T(\Q)$ on
$P^2_{B}(\mathcal{X}_{\C}, \Q(1))$
preserves the $\Q$-Hodge structure on it.
\end{enumerate}
\end{thm}

\begin{proof}
Recall $\widehat{\Br}=\widehat{\Br}(X)$ is the formal Brauer group associated with $X$.
	Since the height of $X$ is $h < \infty$, there is a natural homomorphism
	\[
	I_{\Q_{p}} \to (\Aut_{\Q_p}(\widehat{\Br}))^{\mathrm{op}}
	\]
	by Lemma \ref{Lemma:ActionFormalBrauerGroup}.
	Hence we have a homomorphism
	\[
	T_{\Q_p} \to \Aut_{\Q_p}(\widehat{\Br})
	\]
	of algebraic groups over $\Q_p$. 
	
	By Lemma \ref{Lemma:FormalGroupLifting}, there exist a finite totally ramified extension $E$ of $K_0$ and a one-dimensional smooth formal group $\mathcal{G}$ over $\O_E$ whose special fiber is isomorphic to $\widehat{\Br}$ such that the homomorphism $\rho$ factors as 
	\[
	T_{\Q_p} \to \Aut_{\Q_p}(\mathcal{G}) \to \Aut_{\Q_p}(\widehat{\Br}).
	\]
	
	As in Section \ref{Section:Lifting of a special point}, we have the filtration associated with $\mathcal{G}$ 
	\[
	\mathrm{Fil}^1(\mathcal{G}) \hookrightarrow \mathbb{D}(\widehat{\Br}) \otimes_{W} E \hookrightarrow \widetilde{L}_{\cris}\otimes_{W}E. 
	\]
	This is an isotropic line on $\widetilde{L}_{\cris}\otimes_{W}E$.
	Take a generator $e$ of $\mathrm{Fil}^1(\mathcal{G})$,
	and consider an endomorphism
	\[
	i(e):=(i_{\cris}\otimes_{K_0}{E})(e),
	\]
	which is the image of $e$ under the embedding 
	\[
	i_{\cris}\otimes_{W}E \colon \widetilde{L}_{\cris}\otimes_{W}E \hookrightarrow \End_{E}(H^1_{\cris}(\mathcal{A}_s/W)^{\vee}\otimes_{W}E).
	\]
	
	By Theorem \ref{Theorem:CMliftingAbelianVariety}, there exist a finite extension $K$ of $W(\overline{\F}_q)[1/p]$ containing $E$ and an $\O_K$-valued point $\widetilde{s} \in Z_{\mathrm{K}^p}(\O_K)$ such that
	\begin{itemize}
		\item $\widetilde{s}$ is a lift of $\overline{s}$, and
		\item the $0$-th piece  
        \[
        \mathrm{Fil}^0_{\widetilde{s}} \subset H^1_{\cris}(\mathcal{A}_s/W)^{\vee}\otimes_{W}K \cong H^1_{\dR}(\mathcal{A}_{\widetilde{s}}/K)^{\vee}
        \]
        of the Hodge filtration coincides with the image
        \[ i(e)(H^1_{\cris}(\mathcal{A}_s/W)^{\vee}\otimes_{W}K). \]
	\end{itemize} 
	
Since the embedding $i_{\cris}$ is $I_{\Q_p}$-equivariant and the action of
$T_{\Q_p}$ on $\widetilde{L}_{\cris}[1/p]$ preserves $\Fil^1(\mathcal{G})$, we see that the action of $T$ on $H^1_{\dR}(\mathcal{A}_{\widetilde{s}}/K)^{\vee}$ preserves $\mathrm{Fil}^0_{\widetilde{s}}$.
Therefore the $\Q$-torus $T$ can be considered as a $\Q$-torus in $(\End_{\O_K}(\mathcal{A}_{\widetilde{s}})\otimes_{\Z}\Q)^{\times}$. 

Since the Kuga-Satake morphism is \'etale
(see Proposition \ref{Proposition:KugaSatakeMap}),
the $\O_K$-valued point $\widetilde{s} \in \mathscr{S}_{\widetilde{\mathrm{K}}}(\O_K)$ can be lifted to an $\O_K$-valued point on 
$M^{\mathrm{sm}}_{2d, \K^p_0, \Z_{(p)}}$.
We denote it also by $\widetilde{s}$.

Let $(\mathcal{X}, \mathcal{L})$ be the quasi-polarized $K3$ surface over $\O_K$ corresponding to $\widetilde{s} \in M^{\mathrm{sm}}_{2d, \K^p_0, \Z_{(p)}}(\O_K)$.
Its special fiber is isomorphic to $(X_{\overline{\F}_q}, \mathscr{L}_{\overline{\F}_q})$.

We choose an embedding $K \hookrightarrow \C$. We shall show $\mathcal{X}\otimes_{\O_K}\C$ satisfies the conditions of
Theorem \ref{Theorem:LiftingTorusAction:Restate}.
We denote by $x$ the $\C$-valued point comes from $\widetilde{s}$ and the embedding $K \hookrightarrow \C$.
The algebraic group $T$ over $\Q$ can be considered as a subgroup of $\GL(H^1_{B}(\mathcal{A}_x, \Q)^{\vee})$.
We fix an isomorphism of $\Q$-vector spaces
\[
H_{\Q} \cong H^1_{B}(\mathcal{A}_x,\Q)^{\vee}
\]
which carries $\{ s_{\alpha}\}$ to $\{ s_{\alpha, B, x} \}$ and induces the following commutative diagram:
\[
\xymatrix{\Lambda_{\Q}  \ar^-{}[r] \ar_-{\iota_{B}}[rd] &
\widetilde{L}_{\Q} \ar[d]^-{\cong} \ar[r]^-{i} & \End_{\Q}(H_{\Q}) \ar[d]_-{\cong} \\
& \widetilde{\mathbb{V}}_{B, x} \ar[r]^-{i_{B}} & \End_{\Q}(H^1_{B}(\mathcal{A}_x,\Q)^{\vee}).
}
\]
This isomorphism identifies $\GSpin(\widetilde{V})$ with $\GSpin(\widetilde{V}_{x, B})$ and identifies $\GSpin(\widetilde{V}_{x, B})$ with the subgroup of $\GL(H^1_{B}(\mathcal{A}_x, \Q))^{\vee}$ defined by $\{ s_{\alpha, B, x} \}$, where  $\widetilde{V}:=\widetilde{L}_{\Q}$.
Since $T$ is contained in $\widetilde{I}$, we can consider $T$ as a subgroup of $\GSpin(\widetilde{V}_{x, B})$.
       Moreover, since $T$ is contained in $I$, we see that $T$ is compatible with $\iota(\Lambda\otimes_{\Z}\Q)$.
Hence we have 
       \[
       T \hookrightarrow \GSpin(P^2_{B}(\mathcal{X}_{\C}, \Q(1))).
       \]
       Composing this homomorphism with
       \[
       \GSpin(P^2_{B}(\mathcal{X}_{\C}, \Q(1))) \to \SO(P^2_{B}(\mathcal{X}_{\C}, \Q(1))),
       \]
       we have a homomorphism of algebraic groups over $\Q$
       \[
       T \to \SO(P^2_{B}(\mathcal{X}_{\C}, \Q(1))).
       \]
       The base change of this homomorphism to $\Q_{\ell}$ is identified with the homomorphism
       \[
       T_{\Q_{\ell}} \to \SO(P^2_{\mathrm{\acute{e}t}}(X_{\overline{\F}_q},\Q_{\ell}(1)))
       \]
       via the canonical isomorphism 
       \[
       P^2_{B}(\mathcal{X}_{\C}, \Q(1))\otimes_{\Q}{\Q_{\ell}} \cong P^2_{\mathrm{\acute{e}t}}(X_{\overline{\F}_q},\Q_{\ell}(1)).
       \]

      Finally, we shall prove the $K3$ surface $\mathcal{X}_{\C}$ has CM.
      It is enough to show the Mumford-Tate group of $P^2_{B}(\mathcal{X}_{\C}, \Q(1))$ is commutative; see Section \ref{Subsection:K3 surfaces with complex multiplication}.
      It suffices to prove the image of 
      \[
      \mathbb{S} \to \SO(P^2_{B}(\mathcal{X}_{\C}, \R(1)))
      \] is contained in the image of $T_{\R}$.
      To prove this, it suffices to show the image of 
      \[
      \mathbb{S} \to \GSpin(P^2_{B}(\mathcal{X}_{\C}, \R(1)))\]
       is contained in $T_{\R}$.
      By Proposition \ref{Proposition:KisinGroup(I)}, it follows that $T$ is a maximal $\Q$-torus of $\GSpin(P^2_{B}(\mathcal{X}_{\C}, \Q(1)))$.
      Therefore, it suffices to show the image of $\mathbb{S}$ is contained in the centralizer of $T_{\R}$.
      Since $T(\Q)$ is Zariski dense in $T_{\R}$, this follows from the fact that every element of $T(\Q)$ comes from an element of $\End_{\C}(\mathcal{A}_x) \otimes_{\Z} \Q$.
\end{proof}

Using Theorem \ref{Theorem:LiftingTorusAction:Restate},
we can show that the assertion of
Proposition \ref{Proposition:KisinGroup(I)} (1)
holds for every $\ell$ (including $\ell = p$).

\begin{cor}
\label{Corollary:KisinGroup(II)}
For every prime number $\ell$ (including $\ell = p$),
the canonical homomorphism 
\[
I_{\Q_{\ell}} \to I_{\ell}
\]
is an isomorphism.
\end{cor}

\begin{proof}
The assertion follows from Proposition \ref{Proposition:KisinGroup(I)} and Theorem
\ref{Theorem:LiftingTorusAction:Restate}.
(See the proof of \cite[Corollary 2.3.2]{KisinModp} for details.)
\end{proof}

We shall give applications of Theorem \ref{Theorem:LiftingTorusAction:Restate}
to CM liftings and quasi-canonical liftings of $K3$ surfaces of finite height
over a finite field.
For the definition of CM liftings used in this paper,
see Section \ref{Subsection:Introduction:CMLifting}.
For the definition of quasi-canonical liftings,
see \cite[Definition 1.5]{NygaardOgus}.

\begin{cor}
\label{Corollary:CMLifting}
Let $X$ be a $K3$ surface of finite height over $\F_q$.
Then there is a positive integer $m \geq 1$
such that $X_{\F_{q^m}}$ admits a CM lifting.
\end{cor}

\begin{proof}
After replacing $\F_q$ by its finite extension, we may assume $X$ comes from an $\F_q$-valued point $s \in M^{\mathrm{sm}}_{2d, \K^p_0, \Z_{(p)}}(\F_q)$ satisfying the conditions as in Section \ref{Subsection:$F$-crystals on Shimura varieties and the cohomology of K3 surfaces}.
After replacing $\F_q$ by its finite extension again, there exist a number field $F$, a finite place $v$ of $F$ with residue field $\F_q$, and a $K3$ surface $\mathscr{X}$ over $\O_{F, (v)}$
whose special fiber $\mathscr{X}_{\F_q}$ is isomorphic to $X$,
and generic fiber $\mathscr{X}_{F}$ is a $K3$ surface with CM over $F$ by Theorem \ref{Theorem:LiftingTorusAction:Restate}
and Proposition \ref{Proposition:CM K3 is defined over a number field}.
\end{proof}

\begin{cor}
\label{Corollary:QuasiCanonicalLifting}
Let $X$ be a $K3$ surface of finite height over $\F_q$.
Then there is a positive integer $m \geq 1$
such that $X_{\F_{q^m}}$ admits a quasi-canonical lifting.
\end{cor}

\begin{proof}
After replacing $\F_q$ by its finite extension, we may assume $X$ comes from an $\F_q$-valued point $s \in M^{\mathrm{sm}}_{2d, \K^p_0, \Z_{(p)}}(\F_q)$ satisfying the conditions as in Section \ref{Subsection:$F$-crystals on Shimura varieties and the cohomology of K3 surfaces}.
Since $\Frob^m_q \in T(\Q)$ for a sufficiently divisible $m \geq 1$,
the characteristic $0$ lifting constructed in Theorem \ref{Theorem:LiftingTorusAction:Restate} is a quasi-canonical lifting in the sense of Nygaard-Ogus.
Hence the assertion follows from 
Theorem \ref{Theorem:LiftingTorusAction:Restate} and Proposition \ref{Proposition:CM K3 is defined over a number field}.
\end{proof}

\begin{rem}
Corollary \ref{Corollary:QuasiCanonicalLifting} was previously known
when $p \geq 5$ by Nygaard-Ogus \cite[Theorem 5.6]{NygaardOgus}.
Precisely, when $p \geq 5$, Nygaard-Ogus proved the existence of quasi-canonical liftings without extending the base field $\F_q$.
On the other hand, by the methods of this paper,
it is necessary to take a finite extension $\F_{q^m}$ of $\F_q$
for some $m \geq 1$.
\end{rem}

\section{The Tate conjecture for the square of a $K3$ surfaces over finite fields}
\label{Section:TateConjecture}

\subsection{The statement of the main results}

We recall the statement of the Tate conjecture.
Let $V$ be a projective smooth variety over $\F_q$.
For a prime number $\ell \neq p$,
the Tate conjecture for $\ell$-adic cohomology states
the surjectivity of the $\ell$-adic cycle class map
\[
\mathrm{cl}^i_{\ell} \colon Z^i(V) \otimes_{\Z} \Q_{\ell} \to H^{2i}_{\text{\'et}}( V_{\overline{\F}_q}, \Q_{\ell}{(i)})^{\Gal(\overline{\F}_q/\F_q)}
\]
for every $i$.
Similarly, the Tate conjecture for crystalline cohomology states
the surjectivity of the crystalline cycle class map
\[
\mathrm{cl}^{i}_{\cris} \colon Z^{i}(V)\otimes_{\Z}\Q_p \to
H^{2i}_{\cris}( V/W(\F_q))^{\varphi = p^i} \otimes_{\Z} \Q
\]
for every $i$.
Here $Z^i(V)$ denotes the group of algebraic cycles of codimension $i$ on $V$. (See
\cite[Conjecture 0.1]{NygaardOgus}, \cite[Section 1]{Tate}, \cite[Conjecture 1.1]{Totaro}.)

Here is the statement of our results on the Tate conjecture for the square of a $K3$ surface over a finite field.
As before, let $p$ be a prime number, and $q$ a power of $p$.

\begin{thm}
\label{Theorem:TateConjecture:Restate}
Let $X$ be a $K3$ surface (of any height) over $\F_q$.
We put
$X \times X := X \times_{\Spec \F_q} X$
and
$X_{\overline{\F}_q} \times X_{\overline{\F}_q} := X_{\overline{\F}_q} \times_{\Spec \overline{\F}_q} X_{\overline{\F}_q}$.
\begin{enumerate}
\item For every prime number $\ell \neq p$,
the $\ell$-adic cycle class map
\[
\mathrm{cl}^{i}_{\ell} \colon Z^{i}(X \times X)\otimes_{\Z}\Q_{\ell} \to H^{2i}_{\rm{\acute{e}t}}(  X_{\overline{\F}_q} \times X_{\overline{\F}_q}, \Q_{\ell}{(i)})^{\Gal(\overline{\F}_q/\F_q)}
\]
is surjective for every $i$.
\item The crystalline cycle class map
\[
\mathrm{cl}^{i}_{\cris} \colon Z^{i}(X \times X)\otimes_{\Z}\Q_p \to
H^{2i}_{\cris}( (X \times X)/W(\F_q))^{\varphi = p^i} \otimes_{\Z} \Q
\]
is surjective for every $i$.
\end{enumerate}
\end{thm}

\subsection{Previous results on the Tate conjecture}

In this subsection,
we recall previously known results on the Tate conjecture
which will be used to prove
Theorem \ref{Theorem:TateConjecture:Restate}.

\begin{lem}
\label{Lemma:TateConjectureFiniteExtension}
Let $V$ be a projective smooth variety over $\F_q$.
Let $\ell$ be a prime number different from $p$.
Let $i$ be an integer, and $m \geq 1$ a positive integer.
If the Tate conjecture holds for algebraic cycles of codimension $i$ on
the variety $V_{\F_{q^m}}$ over $\F_{q^m}$,
then the Tate conjecture holds for algebraic cycles of codimension $i$ on
the variety $V$ over $\F_q$.
\end{lem}

\begin{proof}
See \cite[Section 2, p.6]{Totaro} for example.
\end{proof}

\begin{lem}
\label{Lemma:SurjectivityOutsideMiddleDegree}
Let $X$ be a $K3$ surface over $\F_q$.
Let $\ell$ be a prime number different from $p$.
The $\ell$-adic cycle class map $\mathrm{cl}^i_{\ell}$ for $X \times X$
is surjective for every $i \neq 2$.
The same is true for the crystalline cycle class map
$\mathrm{cl}^{i}_{\cris}$ for every $i \neq 2$.
\end{lem}

\begin{proof}
It is enough to prove the assertion for $i=1,3$.
For $i=1$, the K\"unneth formula gives an isomorphism:
\begin{align*}
&\quad H^{2}_{\rm{\acute{e}t}}( X_{\overline{\F}_q} \times X_{\overline{\F}_q}, \Q_{\ell}{(1)}) \\
&\cong \bigoplus_{(i,j) = (0,2),(2,0)}
H^{i}_{\rm{\acute{e}t}}( X_{\overline{\F}_q}, \Q_{\ell})
\otimes_{\Q_{\ell}} H^{j}_{\rm{\acute{e}t}}( X_{\overline{\F}_q}, \Q_{\ell})
\otimes_{\Q_{\ell}} \Q_{\ell}(1).
\end{align*}
Every element in the left-hand side fixed by $\Gal(\overline{\F}_q/\F_q)$
is written as
$\mathrm{pr}_1^{\ast} \alpha + \mathrm{pr}_2^{\ast} \beta$
for some
$\alpha, \beta \in H^{2}_{\rm{\acute{e}t}}( X \otimes_{\F_q}{\overline{\F}_q}, \Q_{\ell}{(1)})$
fixed by $\Gal(\overline{\F}_q/\F_q)$,
where $\mathrm{pr}_1, \mathrm{pr}_2 \colon X \times X \to X$
are the projections.
By the Tate conjecture for the $K3$ surface $X$
(see \cite{Charles13, KimMadapusiIntModel, MadapusiTateConj, Maulik}),
such elements $\alpha$ and $\beta$ are $\Q_{\ell}$-linear combinations of
classes of divisors on $X$.
Hence $\mathrm{pr}_1^{\ast} \alpha + \mathrm{pr}_2^{\ast} \beta$
is a $\Q_{\ell}$-linear combination of classes of divisors on $X \times X$.
This proves the surjectivity of $\mathrm{cl}^{1}_{\ell}$.

To show the surjectivity of $\mathrm{cl}^{3}_{\ell}$,
we take an ample line bundle $\mathscr{L}$ on $X \times X$.
By the hard Lefschetz theorem,
the cup product with the square $(\mathrm{ch}_{\ell}(\mathscr{L}))^2$
of the first Chern class induces a $\Gal(\overline{\F}_q/\F_q)$-equivariant isomorphism:
\[
H^{2}_{\rm{\acute{e}t}}( X_{\overline{\F}_q} \times X_{\overline{\F}_q}, \Q_{\ell}{(1)})
\cong
H^{6}_{\rm{\acute{e}t}}( X_{\overline{\F}_q} \times X_{\overline{\F}_q}, \Q_{\ell}{(3)}).
\]
Taking the $\Gal(\overline{\F}_q/\F_q)$-invariants
and using the surjectivity of $\mathrm{cl}^{1}_{\ell}$,
we see that $\mathrm{cl}^{3}_{\ell}$ is surjective.

The proof for the crystalline cohomology is exactly the same.
\end{proof}

\begin{lem}
\label{Lemma:TateConjecture:Reduction}
Let $X$ be a $K3$ surface over $\F_q$.
Assume that Theorem \ref{Theorem:TateConjecture:Restate} (1) holds
for $i=2$ and for some $\ell \neq p$.
Then Theorem \ref{Theorem:TateConjecture:Restate} (1) holds
for $i=2$ and for every $\ell \neq p$,
and Theorem \ref{Theorem:TateConjecture:Restate} (2) holds for $i=2$.
\end{lem}

\begin{proof}
Let
$N^2(X \times X):= Z^{2}(X \times X)/\sim_{\mathrm{num}}$
be the group of numerically equivalent classes of algebraic cycles of codimension $2$ on $X \times X$.
It is a finitely generated abelian group.
Assume that Theorem \ref{Theorem:TateConjecture:Restate} (1) holds for a prime number $\ell_0 \neq p$.
Since the action of the geometric Frobenius morphism $\Frob_q$
on the $\ell_0$-adic cohomology
$H^{4}_{\rm{\acute{e}t}}(  X_{\overline{\F}_q} \times X_{\overline{\F}_q}, \Q_{\ell_0}{(2)})$
is semisimple (see \cite[Corollaire 1.10]{Deligne:LiftingK3}),
the characteristic polynomial of $\Frob_q$ on
$H^{4}_{\rm{\acute{e}t}}(  X_{\overline{\F}_q} \times X_{\overline{\F}_q}, \Q_{\ell_0}{(2)})$
is equal to the rank of $N^2(X \times X)$;
see \cite[Theorem 2.9]{Tate}.
For any prime number $\ell \neq p$, the following inequality holds (see \cite[Proposition 2.8 (iii)]{Tate}):
\[
\rank_{\Z} N^2(X \times X) \leq
\dim_{\Q_{\ell}} (H^{4}_{\rm{\acute{e}t}}(  X_{\overline{\F}_q} \times X_{\overline{\F}_q}, \Q_{\ell}{(2)})^{\Gal(\overline{\F}_q/\F_q)}).
\]
Since the characteristic polynomial of $\Frob_q$
does not depend on $\ell$,
the above inequality is an equality for any $\ell \neq p$.
Hence Theorem \ref{Theorem:TateConjecture:Restate} (1) holds
for $i=2$ and for every $\ell \neq p$.

The same argument works also for crystalline cohomology.
We put $q = p^r$, $K_0 := W(\F_q)[1/p]$, and
\[ H := H^{4}_{\cris}( (X \times X)/W(\F_q)) \otimes_{\Z} \Q. \]
The $r$-th power $\varphi^r$ of the absolute Frobenius automorphism
acts $K_0$-linearly on $H$,
and the characteristic polynomial of $\varphi^r$ on $H$
coincides with the characteristic polynomial of $\Frob_q$
on the $\ell$-adic cohomology for every $\ell \neq q$.
Hence we have the following equality
\[
\rank_{\Z} N^2(X \times X) =
\dim_{K_0} (H^{\varphi^r = q^2}).
\]
The action of $p^{-2} \varphi$ on the $K_0$-vector space
$H^{\varphi^r = q^2}$
can be considered as a semilinear action of $\Gal(K_0/\Q_p)$
on $H^{\varphi^r = q^2}$.
By Hilbert's theorem 90, we have
\[ \dim_{\Q_p} (H^{\varphi = p^2}) = \dim_{K_0} ( H^{\varphi^r = q^2}). \]
Hence we have
\[
\rank_{\Z} N^2(X \times X) = \dim_{\Q_p} (H^{\varphi = p^2}).
\]
Consequently, Theorem \ref{Theorem:TateConjecture:Restate} (2) holds for $i=2$.
\end{proof}

\begin{rem}
We can prove the following:
if Theorem \ref{Theorem:TateConjecture:Restate} (2) holds for $i=2$,
then Theorem \ref{Theorem:TateConjecture:Restate} (1) holds for $i=2$ and for every $\ell \neq p$.
To prove this, it is enough to show that,
for a $K3$ surface $X$ over $\F_q$ with $q = p^r$,
the action of $\varphi^r$ on the crystalline cohomology of $X$ is semisimple.
When $X$ is of finite height,
it follows from the semisimplicity of Frobenius
for the Kuga-Satake abelian variety.
(See \cite[Theorem 5.17 (3)]{MadapusiTateConj}.
For the case $p=2$, see also \cite[Appendix A]{KimMadapusiIntModel}.)
When $X$ is supersingular,
it follows from the Tate conjecture for $X$;
see Lemma \ref{Lemma:SupersingularPicardNumber} below.

\end{rem}

The followings results on supersingular $K3$ surfaces are well-known.

\begin{lem}
\label{Lemma:SupersingularPicardNumber}
Let $X$ be a $K3$ surface over an algebraically closed field $k$
of characteristic $p > 0$.
Then $X$ is supersingular (i.e.\ the height of $X$ is $\infty$)
if and only if the rank of the Picard group $\Pic(X)$ is $22$.
\end{lem}

\begin{proof}
See \cite[Theorem 4.8]{Liedtke} for example.
(Precisely, the characteristic $p$ is assumed to be odd in \cite[Theorem 4.8]{Liedtke}.
But the same proof works in the case $p=2$
because the Tate conjecture for $K3$ surfaces in characteristic $2$
is now proved by \cite[Theorem A.1]{KimMadapusiIntModel}.)
\end{proof}

\begin{lem}
\label{Lemma:TateConjectureSupersingular}
Let $X$ be a supersingular $K3$ surface over $\F_q$.
Then the Tate conjecture for $X \times X$ holds for
the $\ell$-adic cohomology for every prime number $\ell \neq p$,
and for the crystalline cohomology.
\end{lem}

\begin{proof}
Fix a prime number $\ell \neq p$.
After replacing $\F_q$ by a finite extension of it,
we may assume
$H^2_{\mathrm{\acute{e}t}}(X_{\overline{\F}_q}, \Q_{\ell}(1))$
is spanned by classes of divisors on $X$ defined over $\F_q$
by Lemma \ref{Lemma:SupersingularPicardNumber}.
We have an isomorphism
\begin{align*}
&\quad H^{4}_{\rm{\acute{e}t}}( X_{\overline{\F}_q} \times X_{\overline{\F}_q}, \Q_{\ell}{(2)}) \\
&\cong \bigoplus_{(i,j) = (0,4),(2,2),(4,0)}
H^{i}_{\rm{\acute{e}t}}( X_{\overline{\F}_q}, \Q_{\ell})
\otimes_{\Q_{\ell}} H^{j}_{\rm{\acute{e}t}}( X_{\overline{\F}_q}, \Q_{\ell})
\otimes_{\Q_{\ell}} \Q_{\ell}(2)
\end{align*}
by the K\"unneth formula.
Hence 
$H^{4}_{\rm{\acute{e}t}}( X_{\overline{\F}_q} \times X_{\overline{\F}_q}, \Q_{\ell}{(2)})$
is spanned by classes of algebraic cycles of codimension $2$ on $X \times X$.
Thus the Tate conjecture holds for $X \times X$.
The same proof works for the crystalline cohomology.
\end{proof}

\begin{rem}
\label{Remark:TateConjectureSupersingularPower}
By the same argument, we can prove the Tate conjecture holds for any power $X \times \cdots \times X$
for a supersingular $K3$ surface $X$ over $\F_q$.
\end{rem}

\subsection{Endomorphisms of the cohomology of a $K3$ surface over a finite field}
\label{Subsection:EndomorphismsCohomologyK3}

Let $X$ be a $K3$ surface of finite height over $\F_q$.
After replacing $\F_q$ by its finite extension, we may assume $X$ comes from an $\F_q$-valued point $s \in M^{\mathrm{sm}}_{2d, \K^p_0, \Z_{(p)}}(\F_q)$ satisfying the conditions as in Section \ref{Subsection:$F$-crystals on Shimura varieties and the cohomology of K3 surfaces}.
Let $I$ be the algebraic group over $\Q$ associated with $s \in Z_{\mathrm{K}^p}(\F_q)$; see Definition \ref{Definition: algebraic group I}.

In this subsection, we fix a prime number $\ell \neq p$.
Let
\[ V_{\ell} := \mathrm{ch}_{\ell}(\Pic(X_{\overline{\F}_q}))^{\perp} \subset H^2_{\mathrm{\acute{e}t}}(X_{\overline{\F}_q},\Q_{\ell}(1))
\]
be the transcendental part of the $\ell$-adic cohomology.
By the Tate conjecture for $X$ \cite{Charles13, KimMadapusiIntModel, MadapusiTateConj, Maulik}, every eigenvalue of $\Frob_q$ is not a root of unity.

\begin{lem}
\label{Lemma:TateConjecturePrimitivePart}
\begin{enumerate}
\item There is a $\Gal(\overline{\F}_q/\F_q)$-equivariant
isomorphism
\begin{align*}
&\quad 
H^{4}_{\rm{\acute{e}t}}( X_{\overline{\F}_q} \times X_{\overline{\F}_q}, \Q_{\ell}{(2)}) \\
&\cong
\Q_{\ell} \oplus
(\Pic(X_{\overline{\F}_q})^{\otimes 2} \otimes_{\Z} \Q_{\ell})
\oplus
(\Pic(X_{\overline{\F}_q}) \otimes_{\Z} V_{\ell})^{\oplus 2} \oplus \End_{\Q_{\ell}}(V_{\ell}).
\end{align*}

\item The Tate conjecture holds for $X \times X$ if and only if
the $\Q_{\ell}$-vector subspace
\[ \End_{\Frob_q}(V_{\ell}) = \End_{\Q_{\ell}}(V_{\ell})^{\Gal(\overline{\F}_q/\F_q)} \]
is spanned by classes of algebraic cycles of codimension $2$ on $X \times X$.
\end{enumerate}
\end{lem}

\begin{proof}
(1) We have
\[
H^{i}_{\rm{\acute{e}t}}( X_{\overline{\F}_q}, \Q_{\ell}) \cong
\begin{cases}
\Q_{\ell} & i = 0 \\
(\Pic(X_{\overline{\F}_q}) \otimes_{\Z} \Q_{\ell}(-1)) \oplus V_{\ell}(-1)& i = 2 \\
\Q_{\ell}(-2) & i = 4 \\
0         & i \neq 0,2,4.
\end{cases}
\]
By the Poincar\'e duality, we have isomorphisms
\[ V_{\ell} \otimes_{\Q_{\ell}} V_{\ell} \cong 
   V_{\ell}^{\vee} \otimes_{\Q_{\ell}} V_{\ell} \cong
   \End_{\Q_{\ell}}(V_{\ell}).
\]
Hence the assertion (1) follows by the K\"unneth formula.

(2) The $\Q_{\ell}$-vector space
$\Pic(X_{\overline{\F}_q}) \otimes_{\Z} V_{\ell}$
has no non-zero $\Gal(\overline{\F}_q/\F_q)$-invariants.
Hence the assertion (2) follows.
\end{proof}

Since the action of $I(\Q_{\ell})$ on
the primitive part $P^{2}_{\mathrm{\acute{e}t}}( X_{\overline{\F}_q}, \Q_{\ell}{(1)})$
commutes with $\Frob_q$,
it also acts on $V_{\ell}$.
For a sufficiently divisible $m \geq 1$,
the following conditions are satisfied:
\begin{enumerate}
\item $I_{\ell} = I_{\ell, m} = I_{\Q_{\ell}}$
(for the definition of $I_{\ell}, I_{\ell, m}$,
see Section \ref{Subsection:Kisin's algebraic groups}).

\item The image of $I_{\ell, m}$ under the homomorphism
$\GSpin(V_{\ell}) \to \SO(V_{\ell})$
is equal to the centralizer $\SO_{\Frob^{m}_q}(V_{\ell})$ of $\Frob^{m}_q$ in $\SO(V_{\ell})$.
\end{enumerate}

In the rest of this subsection, we fix an integer $m \geq 1$
satisfying the above conditions.

Let $\End_{\Frob^{m}_q}(V_{\ell})$ be the set of $\Q_{\ell}$-linear
endomorphisms of $V_{\ell}$ commuting with $\Frob^{m}_q$.
Similarly, let
$\End_{\Frob^{m}_q}(V_{\ell} \otimes_{\Q_{\ell}} \overline{\Q}_{\ell})$
be the set of $\overline{\Q}_{\ell}$-linear
endomorphisms of $V_{\ell} \otimes_{\Q_{\ell}} \overline{\Q}_{\ell}$
commuting with $\Frob^{m}_q$.
We have a map
\[ I(\Q_{\ell}) \to \End_{\Frob^{m}_q}(V_{\ell}). \]
Similarly, we also have a map
\[ I(\overline{\Q}_{\ell}) \to \End_{\Frob^{m}_q}(V_{\ell} \otimes_{\Q_{\ell}} \overline{\Q}_{\ell}). \]

\begin{lem}
\label{Lemma:Endomorphism1}
The $\overline{\Q}_{\ell}$-vector space
$\End_{\Frob^{m}_q}(V_{\ell} \otimes_{\Q_{\ell}} \overline{\Q}_{\ell})$
is spanned by the image of $I(\overline{\Q}_{\ell})$.
\end{lem}

\begin{proof}
Let $R \subset \End_{\overline{\Q}_{\ell}}(V_{\ell} \otimes_{\Q_{\ell}} \overline{\Q}_{\ell})$
be the $\overline{\Q}_{\ell}$-vector subspace
generated by the image of $I(\overline{\Q}_{\ell})$.
Since the action of $I(\overline{\Q}_{\ell})$ on
$V_{\ell} \otimes_{\Q_{\ell}} \overline{\Q}_{\ell}$ is semisimple,
$R$ is a semisimple $\overline{\Q}_{\ell}$-subalgebra.
Hence it suffices to prove every element of
$\End_{\Frob^{m}_q}(V_{\ell} \otimes_{\Q_{\ell}} \overline{\Q}_{\ell})$
commutes with every element in the commutant of $R$.

Since the action of $\Frob^{m}_q$ preserves the bilinear form on $V_{\ell}$,
if $\alpha$ is an eigenvalue of $\Frob^{m}_q$,
then $\alpha^{-1}$ is also an eigenvalue of $\Frob^{m}_q$.
Since all the eigenvalues of $\Frob^{m}_q$ on $V_{\ell}$
are not roots of unity,
we denote the distinct eigenvalues of $\Frob^{m}_q$ by
$\alpha_1, \alpha_1^{-1}, \ldots, \alpha_r, \alpha_r^{-1} \in \overline{\Q}_{\ell}$.
Let $W_1, W_1^{-},\ldots, W_r, W_r^{-}$
be the eigenspaces of the eigenvalues
$\alpha_1, \alpha_1^{-1}, \ldots, \alpha_r, \alpha_r^{-1}$, respectively.
Since $\Frob^{m}_q$ acts semisimply on $V_{\ell}$, we have
\[ V_{\ell} \otimes_{\Q_{\ell}} \overline{\Q}_{\ell} \cong
   \bigoplus_{i=1}^{r} (W_i \oplus W_i^{-}). \]
Hence we have
\begin{align*}
\End_{\Frob^{m}_q}(V_{\ell} \otimes_{\Q_{\ell}} \overline{\Q}_{\ell})
&\cong \bigoplus^{r}_{i=1} (\End_{\overline{\Q}_{\ell}}(W_i) \oplus \End_{\overline{\Q}_{\ell}}(W_i^{-})), \\
\SO_{\Frob^{m}_q}(V_{\ell} \otimes_{\Q_{\ell}} \overline{\Q}_{\ell})
&\cong \prod^{r}_{i=1} \GL(W_i).
\end{align*}
By Schur's lemma, every element
$g \in \End_{\overline{\Q}_{\ell}}(V_{\ell} \otimes_{\Q_{\ell}} \overline{\Q}_{\ell})$
in the commutant of $R$ is written as
\[
g = \bigoplus_{i=1}^{r} (g_i \oplus g_i^{-}),
\]
where $g_1,g_1^{-},\ldots,g_r,g_r^{-}$ are multiplication by scalars.
Hence $g$ commutes with every element of
$\End_{\Frob^{m}_q}(V_{\ell} \otimes_{\Q_{\ell}} \overline{\Q}_{\ell})$.
\end{proof}

\begin{lem}
\label{Lemma:Endomorphism2}
Let $G$ be an algebraic group over an algebraically closed field $k$ of characteristic $0$.
Let $V$ be a finite dimensional $k$-vector space,
and $\rho \colon G \to \GL(V)$ a morphism of algebraic groups over $k$.
For any Zariski dense subset $Z \subset G(k)$,
we have $\langle \rho(Z) \rangle = \langle \rho(G(k)) \rangle$,
where $\langle \rho(Z) \rangle$ (resp.\ $\langle \rho(G(k)) \rangle$)
is the $k$-vector subspace of $\End_k(V)$ spanned by $\rho(Z)$ (resp.\ $\rho(G(k))$).
\end{lem}

\begin{proof}
We put $d := \dim_k \langle \rho(G(k)) \rangle$.
Let $\psi$ be the composite of the following maps
\[
G(k)^d := \prod_{i=1}^{d} G(k) \to 
\prod_{i=1}^{d} \End_k(V) \to 
\wedge^d \End_k(V).
\]
If $\dim_k \langle \langle \rho(Z) \rangle < d$,
we have $\psi(Z^d) = \{ 0 \}$.
Since $Z^d \subset G(k)^d$ is Zariski dense,
we have $\psi(G(k)^d) = \{ 0 \}$,
which is absurd.
The contradiction shows $\dim_k \langle \rho(Z) \rangle = d$.
\end{proof}

\begin{lem}
\label{Lemma:SemisimpleElementGenerarion}
\begin{enumerate}
\item As a $\Q_{\ell}$-vector space, $\End_{\Frob^{m}_q}(V_{\ell})$
is spanned by the image of semisimple elements in $I(\Q)$.

\item There exist maximal tori
$T_1,\ldots,T_n \subset I$ over $\Q$ such that
the $\Q_{\ell}$-vector space $\End_{\Frob^{m}_q}(V_{\ell})$
is spanned by the image of $T_1(\Q),\ldots,T_n(\Q)$.
\end{enumerate}
\end{lem}

\begin{proof}
(1) \ Since $I$ is a connected reductive algebraic group over $\Q$,
the set of semisimple elements in $I(\Q)$ is Zariski dense in $I(\overline{\Q}_{\ell})$;
see \cite[Expose XIV, Corollaire 6.4]{SGA3}.
By Lemma \ref{Lemma:Endomorphism1} and Lemma \ref{Lemma:Endomorphism2},
the $\overline{\Q}_{\ell}$-vector space
$\End_{\Frob^{m}_q}(V_{\ell} \otimes_{\Q_{\ell}} \overline{\Q}_{\ell})$
is spanned by the image of semisimple elements in $I(\Q)$.
Since
\[ \End_{\Frob^{m}_q}(V_{\ell} \otimes_{\Q_{\ell}} \overline{\Q}_{\ell})
    = \End_{\Frob^{m}_q}(V_{\ell}) \otimes_{\Q_{\ell}} \overline{\Q}_{\ell}, \]
the $\Q_{\ell}$-vector space
$\End_{\Frob^{m}_q}(V_{\ell})$
is spanned by the image of semisimple elements in $I(\Q)$.

(2) \ The assertion follows from the fact that
every semisimple element of $I(\Q)$ is contained in
a maximal torus of $I$ over $\Q$.
\end{proof}

\subsection{The results of Mukai and Buskin}

The following theorem will be used
in our proof of Theorem \ref{Theorem:TateConjecture:Restate}.

\begin{thm}[Mukai, Buskin]
\label{Theorem:MukaiBuskin}
Let $T$ and $S$ be projective $K3$ surfaces over $\C$.
Let $\psi \colon H^2_{B}(S,\Q) \cong H^2_{B}(T,\Q)$
be an isomorphism of $\Q$-vector spaces which preserves
the cup product pairings and the $\Q$-Hodge structure.
Let $[\psi] \in H^4_{B}(S \times T,\Q(2))$
be the class corresponding to $\psi$ by the Poincar\'e duality
and the K\"unneth formula.
Then $[\psi]$ is the class of an algebraic cycle of codimension $2$ on $S \times T$.
\end{thm}

\begin{proof}
See \cite[Theorem 1.1]{Buskin}, \cite[Theorem 2]{Mukai}.
(See also \cite[Corollary 0.4]{Huybrechts}.)
\end{proof}

\subsection{Proof of Theorem \ref{Theorem:TateConjecture:Restate}}

In this subsection, we shall prove Theorem \ref{Theorem:TateConjecture:Restate}.

By Lemma \ref{Lemma:TateConjectureSupersingular},
it is enough to prove Theorem \ref{Theorem:TateConjecture:Restate}
for $K3$ surfaces of finite height.

Let $X$ be a $K3$ surface of finite height over $\F_q$.
After replacing $\F_q$ by its finite extension, we may assume $X$ comes from an $\F_q$-valued point $s \in M^{\mathrm{sm}}_{2d, \K^p_0, \Z_{(p)}}(\F_q)$ satisfying the conditions as in Section \ref{Subsection:$F$-crystals on Shimura varieties and the cohomology of K3 surfaces}.
Let $I$ be the algebraic group over $\Q$ associated with $s \in Z_{\mathrm{K}^p}(\F_q)$; see Definition \ref{Definition: algebraic group I}.

By Lemma \ref{Lemma:TateConjecture:Reduction},
it is enough prove Theorem \ref{Theorem:TateConjecture:Restate} (1)
for a fixed $\ell$.
We fix a prime number $\ell \neq p$.
We take a sufficiently divisible integer $m \geq 1$
as in Section \ref{Subsection:EndomorphismsCohomologyK3}.
Replacing $\F_q$ by a finite extension of it
(see Lemma \ref{Lemma:TateConjectureFiniteExtension}),
we may assume $m=1$.

By Lemma \ref{Lemma:SemisimpleElementGenerarion},
there exist maximal tori $T_1,\ldots,T_n \subset I$ over $\Q$
such that $\End_{\Frob_q}(V_{\ell})$
is spanned by the image of $T_1(\Q),\ldots,T_n(\Q)$.
By Lemma \ref{Lemma:SurjectivityOutsideMiddleDegree}
and Lemma \ref{Lemma:TateConjecturePrimitivePart},
it is enough to show that, for every $i$ with $1 \leq i \leq n$,
the image of $T_i(\Q)$ in $\End_{\Frob_q}(V_{\ell})$
is spanned by classes of algebraic cycle of codimension $2$ on
$X_{\overline{\F}_q} \times X_{\overline{\F}_q}$.

Fix an integer $i$ with $1 \leq i \leq n$.
By Theorem \ref{Theorem:LiftingTorusAction:Restate},
there exist a finite extension $K$ of $W(\overline{\F}_{q})[1/p]$
and a quasi-polarized $K3$ surface
$(\mathcal{X}, \mathcal{L})$
over $\O_K$ whose special fiber is isomorphic to
$(X_{\overline{\F}_q}, \mathscr{L}_{\overline{\F}_q})$
such that, for any embedding $K \hookrightarrow \C$,
there is a homomorphism of algebraic groups over $\Q$
\[ T_i \to \SO(P^2_{B}(\mathcal{X}_{\C}, \Q(1))), \]
and the action of every element of $T_i(\Q)$
on $P^2_{B}(\mathcal{X}_{\C}, \Q(1))$
preserves the $\Q$-Hodge structure on it.
We extend the action of $T_i(\Q)$ on the primitive part
$P^2_{B}(\mathcal{X}_{\C},\,\Q(1))$
to the full cohomology
$H^2_{B}(\mathcal{X}_{\C},\,\Q(1))$
so that every element of $T_i(\Q)$ acts trivially
on the first Chern class $\mathrm{ch}_B(\mathcal{L}_{\C})$.
Hence we have a homomorphism of algebraic groups over $\Q$
\[ T_i \to \SO(H^2_{B}(\mathcal{X}_{\C}, \Q(1))) \]
whose image preserves the $\Q$-Hodge structure on
$H^2_{B}(\mathcal{X}_{\C}, \Q(1))$.

By the results of Mukai and Buskin (see Theorem \ref{Theorem:MukaiBuskin}),
the image of every element of $T_i(\Q)$ in $\SO(H^2_{B}(\mathcal{X}_{\C}, \Q(1)))$
is a class of an algebraic cycle of codimension $2$ on
$\mathcal{X}_{\C} \times \mathcal{X}_{\C}$.

Taking the specialization of algebraic cycles,
we conclude that the image of every element of $T_i(\Q)$ in $\End_{\Frob_q}(V_{\ell})$
is a class of an algebraic cycle of codimension $2$ on
$X_{\overline{\F}_q} \times X_{\overline{\F}_q}$.

The proof of Theorem \ref{Theorem:TateConjecture:Restate} is complete.
\qed


\section{Compatibility of $p$-adic comparison isomorphisms}
\label{Section:CompatibilityComparisonIsom}

Throughout the article, we only use the following types of $p$-adic period morphisms
from the literature:
\begin{enumerate}
\item the de Rham comparison map of Scholze \cite{ScholzeHodge}, 
\item the crystalline comparison map of Bhatt-Morrow-Scholze \cite{BMS}, 
\item Faltings' comparison map for $p$-divisible groups over $\O_K$ \cite{Faltings99}, and
\item Lau's period morphism in display theory \cite{LauGalois}. 
\end{enumerate}

All these period morphisms are known to be compatible.
The maps (1) and (2) are compatible via the Berthelot-Ogus isomorphism by \cite[Theorem 13.1]{BMS}.
For abelian schemes over $\O_K$, (2) and (3) are compatible by \cite[Proposition 14.8.3]{Scholze};
see also \cite[Proposition 4.15]{ScholzeSurvey} for the Hodge-Tate counterpart.
Finally, (3) and (4) are compatible by \cite[Proposition 6.2]{LauGalois}. More details are given below. 

We also check some basic properties of (1) and (2) to use
the results of Blasius-Wintenberger \cite{Blasius}. 

In this section, we fix a perfect field $k$ of characteristic $p>0$.
We put $W := W(k)$.
We fix a finite totally ramified extension $K$ of $W[1/p]$,
and an algebraic closure $\overline{K}$ of $K$.
We use the same notation as in Section \ref{Section:Breuil-Kisin modules}.

\subsection{The de Rham comparison map of Scholze}
\label{Subsection:Scholze's de Rham comparison map}
We give some basic properties of the de Rham comparison map constructed by Scholze \cite{ScholzeHodge}.
(Compare with \cite[Theorem A1]{TsujiSurvey}.)

Let $X$ be a smooth proper variety over $K$,
and denote by $X^{\ad}$ the adic space associated with it.
By \cite[Theorem 8.4]{ScholzeHodge}, we have an isomorphism
\[
H^i_{\dR}(X^{\ad}/K)\otimes_K B_{\dR} \cong H^i_{\et}(X^{\ad}_{C}, \Q_p) \otimes_{\Q_p} B_{\dR}, 
\]
where $C$ is the completion of ${\overline{K}}$. Another construction of the same map is given in \cite[Theorem 13.1]{BMS}.
Combined with GAGA results such as \cite[Theorem 3.7.2]{Huber}, we get a filtered isomorphism
of $\Gal(\overline{K}/K)$-modules
\[
c_{\dR, X}\colon H^i_{\dR}(X/K)\otimes_K B_{\dR} \overset{\cong}{\longrightarrow} H^i_{\et}(X_{\overline{K}}, \Q_p) \otimes_{\Q_p} B_{\dR}.
\]

By construction, the de Rham comparison map $c_{\dR, X}$ is functorial in $X$ with respect to pullback, and compatible with cup products. 

\begin{rem}
\label{Remark:deRhamComparisonAlgebraicSpace}
The construction of $c_{\dR, X}$ can be generalized to any smooth proper algebraic space $X$ over $K$.
Indeed, \cite{ConradTemkin} supplies a functorial construction of the analytification of $X$, and the corresponding adic space $X^{\ad}$.
As before, we have an isomorphism
\[
H^i_{\dR}(X^{\ad}/K)\otimes_K B_{\dR} \cong H^i_{\et}(X^{\ad}_{C}, \Q_p) \otimes_{\Q_p} B_{\dR}
\]
by \cite[Theorem 8.4]{ScholzeHodge}.
In order to construct $c_{\dR, X}$, it is enough to show
\[
H^i_{\dR}(X/K) \cong H^i_{\dR}(X^{\ad}/K)
\quad \text{and} \quad
H^i_{\et}(X_{\overline{K}}, \Q_p) \cong H^i_{\et}(X^{\ad}_{C}, \Q_p).
\]
(The analytification induces a morphism of \'etale topoi.)
\begin{itemize}
\item Concerning the isomorphism for de Rham cohomology,
by the Hodge-de Rham spectral sequence,
it is enough to show $H^i(X,\Omega^j) \cong H^i(X^{\ad},\Omega^j)$ for every $i,j$.
It follows from GAGA results in \cite{ConradTemkin}.
\item  For the isomorphism for $p$-adic \'etale cohomology,
we take an \'etale covering
$\mathfrak{U} = \{ U_{\alpha} \to X \}_{\alpha}$
of $X$ which consists of schemes.
We put $\mathfrak{U}_{\overline{K}} := \{ (U_{\alpha})_{\overline{K}} \to X_{\overline{K}} \}_{\alpha}$
and $\mathfrak{U}^{\ad}_{C} := \{ (U_{\alpha})^{\ad}_{C} \to X^{\ad}_{C} \}_{\alpha}$.
The analytification preserves \'etale coverings.
We have the \v{C}ech-to-cohomology spectral sequences
\begin{align*}
E_2^{i,j} = \check{H}^i(\mathfrak{U}_{\overline{K}}, \underline{H}^j(\Z/p^n\Z)) &\Rightarrow
H^{i+j}_{\et}(X_{\overline{K}}, \Z/p^n\Z), \\
E_2^{i,j} = \check{H}^i(\mathfrak{U}^{\ad}_{C}, \underline{H}^j(\Z/p^n\Z)) &\Rightarrow
H^{i+j}_{\et}(X^{\ad}_{C}, \Z/p^n\Z).
\end{align*}
(See \cite[Tag 03OW]{StacksProject}.)
There is a canonical morphism between these spectral sequences.
By Huber's theorem \cite[Theorem 3.8.1]{Huber},
it induces isomorphisms between $E_2$-terms.
Therefore, we have an isomorphism
\[
H^i_{\et}(X_{\overline{K}}, \Z/p^n\Z) \cong
H^i_{\et}(X^{\ad}_C, \Z/p^n\Z)
\]
for every $i$ and $n$. 
Taking the projective limit with respect to $n$
and tensoring with $\Q_p$, the required isomorphism for $p$-adic
\'etale cohomology follows.
\end{itemize}
\end{rem}

\begin{prop}\label{Proposition:de Rham comparison and Chern class}
Let $\mathscr{L}$ be a line bundle on $X$, and denote by $\ch_{\dR}(\mathscr{L})$ and $\ch_p(\mathscr{L})$ the first Chern classes of $\mathscr{L}$ in the de Rham cohomology and the $p$-adic \'etale cohomology respectively. The following equality holds
\[
c_{\dR, X}(\ch_{\dR}(\mathscr{L})) = \ch_p (\mathscr{L}) \in H^2_{\et} (X_{\overline{K}}, \Q_p)(1) \subset H^2_{\et}(X_{\overline{K}}, \Q_p)\otimes_{\Q_p} B_{\dR}.  
\]
Here, $\Q_p(1)$ is naturally embedded into $B_{\dR}$. 
\end{prop}

\begin{proof}
The proof given below shows an analogous statement for any smooth proper rigid analytic variety over $K$ as well. 

First, we recall the definition of $\ch_{\dR}(\mathscr{L})$. There are an exact sequence of complex of sheaves on $X$
\[
0 \to (1 \overset{1 \mapsto 0}{\longrightarrow} \Omega_X^{\geq 1}) \to
(\mathcal{O}_X^{\times} \overset{d\log}{\longrightarrow} \Omega_X^{\geq 1}) \to
\mathcal{O}_X^{\times} \to 0,
\]
and an isomorphism
\[
(1 \overset{1 \mapsto 0}{\longrightarrow} \Omega_X^{\geq 1}) \cong
(0 \longrightarrow \Omega_X^{\geq 1}). 
\]
These induce the Chern class map
\[
\ch_{\dR}\colon H^1(X, \mathcal{O}_X^{\times}) \to H^2_{\dR}(X/K). 
\]
This construction can be analytified. 

There is a similar construction on the pro-\'etale site $X^{\ad}_{\proet}$. Recall that Scholze \cite{ScholzeHodgeErratum} defined a sheaf $\mathcal{O}\mathbb{B}_{\dR}^+$ of $\mathcal{O}_{X}$-modules on $X^{\ad}_{\proet}$ with a surjection 
$\theta\colon \mathcal{O}\mathbb{B}_{\dR}^+ \to \hat{\mathcal{O}}_X$. 
In the above construction, we will replace $\mathcal{O}_X^{\times} \overset{d\log}{\longrightarrow} \Omega_X^{\geq 1}$ by
\[
\mathcal{O}\mathbb{B}_{\dR}^{+\times} \overset{d\log}{\longrightarrow} \mathcal{O}\mathbb{B}_{\dR}^{+}\otimes_{\mathcal{O}_X} \Omega_X^{\geq 1}. 
\]
Similarly, we replace $(1 \overset{1 \mapsto 0}{\longrightarrow} \Omega_X^{\geq 1})$ by
\[
1 + \Ker \theta \overset{d\log}{\longrightarrow} \mathcal{O}\mathbb{B}_{\dR}^{+}\otimes_{\mathcal{O}_X} \Omega_X^{\geq 1}. 
\]
As $\mathcal{O}\mathbb{B}_{\dR}^+$ is complete with respect to $\Ker \theta$, we have the logarithm map
\[
\log\colon 1 + \Ker \theta \to \Ker \theta. 
\]
This finishes the construction on $X^{\ad}_{\proet}$. 

On the other hand, $\ch_{p}(\mathscr{L})$ is defined at the level of the pro-\'etale site as follows: it is induced by the following exact sequence of sheaves on $X^{\ad}_{\proet}$
\[
0 \to \hat{\Z}_p (1) \to \hat{\mathcal{O}}^{\times}_{X^{\flat}}=\varprojlim_{x \mapsto x^p} \hat{\mathcal{O}}_X^{\times} \to
\hat{\mathcal{O}}_X^{\times} \to 0. 
\]

The following commutative diagram completes the proof:
\[
\xymatrix{
\hat{\Z}_p (1) \ar[d] \ar[r] & 
\hat{\mathcal{O}}^{\times}_{X^{\flat}} \ar[d]^{1\otimes [-]} \ar[r] & 
\hat{\mathcal{O}}_X^{\times} \ar[d] \\
1 + \Ker \theta \ar[r] &
\mathcal{O}\mathbb{B}_{\dR}^{+\times} \ar[r] &
\hat{\mathcal{O}}_X^{\times}, 
}
\]
where $[-]$ denotes the Teichm\"uller lift. 
\end{proof}

\begin{cor}
The de Rham comparison map $c_{\dR, X}$ is compatible with Chern classes of vector bundles, hence classes of algebraic cycles.  
\end{cor}

\begin{proof}
This follows from the splitting principle and the previous proposition.
(The class of an algebraic cycle is written as a $\Q$-linear combination of Chern classes of vector bundles.)
\end{proof}

\begin{cor}
If $X$ is of equidimension $d$,
the de Rham comparison map $c_{\dR,X}$ is compatible with trace maps.
Therefore, $c_{\dR, X}$ is functorial in $X$ with respect to pushforward. 
\end{cor}

\begin{proof}
This follows from the compatibility with cycle classes of a point. 
\end{proof}

\subsection{The crystalline comparison map of Bhatt-Morrow-Scholze}
\label{Subsection:Crystalline comparison map of Bhatt-Morrow-Scholze}

Let $\mathscr{X}$ be a smooth proper algebraic space over $\O_K$
with generic fiber $X$.
We assume that the generic fiber $X$ and the special fiber $\mathscr{X}_k$ are schemes.
For example, it is satisfied for smooth proper curves and surfaces over $\O_K$
because a smooth proper algebraic space of dimension $\leq 2$
over a field is a scheme.

As in the proof of \cite[Theorem 2.4]{ChiarellottoLazdaLiedtke}, the formal completion $\mathfrak{X}$ of $\mathscr{X}$ along the special fiber is a formal scheme (rather than a formal algebraic space), and there is an isomorphism $t(\mathfrak{X})_\eta\cong X^{\ad}$ of adic spaces over $K$, where $t(\mathfrak{X})_\eta$ denotes the adic generic fiber of $\mathfrak{X}$. We remark that the isomorphsim $t(\mathfrak{X})_\eta\cong X^{\ad}$ is functorial in $\mathscr{X}$; this can be seen by an argument similar to \cite[Theorem 2.2.3]{ConradTemkin}. 

We recall the construction of the crystalline comparison map
of Bhatt-Morrow-Scholze in \cite{BMS}.
Let $A_{\cris}$ be the $p$-adic completion of the divided power envelope of $A_{\mathrm{inf}}$ with respect to $\Ker \theta$.

The absolute crystalline comparison theorem in \cite[Theorem 14.3 (iii)]{BMS} (and \cite[Theorem 14.3 (iv)]{BMS}) gives the following isomorphism
\[
H^i_{\cris}(\mathscr{X}_{\O_C /p}/ A_{\cris}) \otimes_{A_\cris}B_{\cris} \overset{\cong}{\longrightarrow} H^i_{\et} (X_{\overline{K}}, \Z_p)\otimes_{\Z_p} B_{\cris}. 
\]
By \cite[Proposition 13.21]{BMS} and the base change of crystalline cohomology, we have an isomorphism
\begin{equation*}
H^i_{\cris}(\mathscr{X}_{\O_C /p}/ A_{\cris})[1/p] \cong H^i_{\cris}(\mathscr{X}_k /W) \otimes_{W} A_{\cris}[1/p].
\end{equation*}
The above two isomorphisms give us the crystalline comparison map
\[
c_{\cris, \mathscr{X}}\colon 
H^i_{\cris}(\mathscr{X}_k /W) \otimes_{W} B_{\cris} \overset{\cong}{\longrightarrow} H^i_{\et} (X_{\overline{K}}, \Z_p)\otimes_{\Z_p} B_{\cris}. 
\]

\begin{prop}\label{Proposition:Comparison with BO}
The crystalline comparison map $c_{\cris, \mathscr{X}}$ is compatible with the de Rham comparison map $c_{\dR, X}$ under an isomorphism
\[ H^i_{\cris} (\mathscr{X}_k /W) \otimes_{W} K \cong H^i_{\dR} (X/ K) \]
of Berthelot-Ogus \cite{BO}.
\end{prop}

\begin{proof}
By \cite[Proposition 13.23, Remark 13.20]{BMS}, we have an isomorphism
\[
H^i_{\cris}(\mathscr{X}_{\O_C /p}/A_{\cris})\otimes_{A_{\cris}}B_{\dR}^+ \cong H^i_{\dR}(X/K)\otimes_K B_{\dR}^+. 
\]
It is already shown in \cite[Theorem 14.5(i)]{BMS} that $c_{\dR, X}$ is compatible with $c_{\cris, \mathscr{X}}$ under this identification. 

The above identification induces a $K$-linear map
\[
s_{\dR}\colon H^i_{\dR}(X/K) \to H^i_{\cris}(\mathscr{X}_{\O_C /p}/A_{\cris})\otimes_{A_{\cris}}B_{\dR}^+
\]
such that the composite of $s_{\dR}$ with the following map obtained by the reduction modulo $\xi$
\begin{align*}
H^i_{\cris}(\mathscr{X}_{\O_C /p}/A_{\cris})\otimes_{A_{\cris}}B_{\dR}^+
&\to H^i_{\cris}(\mathscr{X}_{\O_C /p}/A_{\cris})\otimes_{A_{\cris}} C \\
&\cong H^i_{\dR} (X_C/C)  
\end{align*}
(this has no ambiguity; see the commutative diagram $(6.5.7)$ in \cite{CK})
is equal to the canonical map
\[ H^i_{\dR}(X/K) \to H^i_{\dR} (X_C/C). \]
The $K$-linear map $s_{\dR}$ is characterized as a unique $\Gal(\overline{K}/K)$-equivariant $K$-linear map satisfying
this property.

On the other hand, \cite[Proposition 13.21]{BMS} gives a $W$-linear map
\[
s_{\cris}\colon H^i_{\cris}(\mathscr{X}_k /W)[1/p] \to H^i_{\cris}(\mathscr{X}_{\O_C /p}/A_{\cris})[1/p]
\]
such that the composite of $s_{\cris}$ with the following map obtained by the specialization $A_{\cris}\to W(\overline{k})$
\begin{align*}
H^i_{\cris}(\mathscr{X}_{\O_C /p}/A_{\cris})[1/p]
&\to H^i_{\cris}(\mathscr{X}_{\O_C /p}/A_{\cris})\otimes_{A_{\cris}} W(\overline{k})[1/p] \\
&\cong H^i_{\cris} (\mathscr{X}_{\overline{k}}/W(\overline{k}))[1/p] 
\end{align*}
is equal to the canonical map
\[ 
H^i_{\cris}(\mathscr{X}_k /W)[1/p] \to 
H^i_{\cris} (\mathscr{X}_{\overline{k}}/W(\overline{k}))[1/p]. 
\]
The $W$-linear map $s_{\cris}$ is also characterized as
a unique $\Gal(\overline{K}/K)$-equivariant $W$-linear map
satisfying this property, which is $\varphi$-equivariant.
Since the images of $s_{\dR}$ and ${s_\cris}\otimes K$ coincide, the reduction modulo $\Ker \theta$ induces an isomorphism
\[
\overline{s_{\cris}\otimes K} \colon H^i_{\cris} (\mathscr{X}_k/W)\otimes_{W} K \cong H^i_{\dR}(X/K),
\]
and $c_{\cris,\mathscr{X}}$ and $c_{\dR, X}$ are compatible under $\overline{s_{\cris}\otimes K}$. 
So, we only need to show that $\overline{s_{\cris}\otimes K}$ is equal to the Berthelot-Ogus isomorphism. 

We can finish by recalling the following well-known interpretation of the Berthelot-Ogus isomorphism: fix a uniformizer $\varpi$ and a system $\{ \varpi^{1/{p^n}} \}_{n \geq 0} \subset \overline{K}$ of $p^n$-th roots of $\varpi$ such that
$(\varpi^{1/{p^{n+1}}})^p=\varpi^{1/{p^n}}$, and let $S_\varpi$ be the $p$-adic completion of the PD envelope of the surjection $W[u] \to \O_K$ defined by $u \mapsto \varpi$. Then, there is a unique $\varphi$-equivariant section
\[
s\colon H^i_{\cris}(\mathscr{X}_k/W)[1/p] \to H^i_{\cris}(\mathscr{X}_{\O_K /p}/S_\varpi)[1/p], 
\]
and $s\otimes K$ induces the Berthelot-Ogus isomorphism. The section $s$ can be constructed in a way parallel to \cite{BO} and \cite[Proposition 13.21]{BMS};
see \cite[Proposition 5.1]{CaisLiu} for example.
This implies that $s=s_{\cris}$ under an isomorphism 
\[
H^i_{\cris}(\mathscr{X}_{\O_K /p}/S_\varpi)\otimes_{S_\varpi} A_{\cris}[1/p] \cong H^i_{\cris}(\mathscr{X}_{\O_C /p}/A_{\cris})[1/p]
\]
by base change along the embedding $S_\varpi \to A_{\cris}$ defined by $u \mapsto [\varpi^{\flat}]$, and the specialization of $s$ is indeed the Berthelot-Ogus isomorphism. 
\end{proof}

\begin{cor}\label{Corollary:Crystalline comparison and Chern class}
The crystalline comparison map $c_{\cris, \mathscr{X}}$ is
\begin{itemize}
\item functorial in $\mathscr{X}$ with respect to pullback and pushforward, and
\item compatible with cup products, Chern classes of vector bundles on $\mathscr{X}$, classes of algebraic cycles on $\mathscr{X}$, and trace maps. 
\end{itemize}
\end{cor}

\begin{proof}
The above properties are already checked for the de Rham comparison map $c_{\dR, X}$.
So, it suffices to show that the Berthelot-Ogus isomorphism satisfies the above properties.
This is done in \cite{BO} and \cite[B. Appendix]{GilletMessing} at least when $\mathscr{X}$ is a scheme .

Let us briefly explain a variant using the interpretation of the Berthelot-Ogus isomorphism given in the proof of Proposition \ref{Proposition:Comparison with BO}, which works in the semistable case as well. The functoriality with respect to pullback and the compatibility with cup products are checked (by the same argument as) in \cite[Corollary 4.4.13]{Tsuji}. Next, we check the compatibility with the first Chern classes of line bundles. We freely use the notation from the proof of Proposition \ref{Proposition:Comparison with BO}. Let $\mathscr{L}$ be a line bundle on $\mathscr{X}$. This gives the first Chern class $\ch_{\cris}(\mathscr{L}/ S_\varpi)$ in $H_{\cris}^2 (\mathscr{X}_{\mathscr{O}_K/p}/ S_\varpi)$ lifting $\ch_{\cris}(\mathscr{L})$. By the characterization of $s$, we see that the following equality holds:
\[ s(\ch_{\cris}(\mathscr{L}))=\ch_{\cris}(\mathscr{L}/ S_\varpi). \]
As it is well-known that $\ch_{\cris}(\mathscr{L}/ S_\varpi)$ maps to $\ch_{\dR}(\mathscr{L})$ modulo $\Ker (S_\varpi \to \mathscr{O}_K)$, we conclude that the Berthelot-Ogus isomorphism maps $\ch_{\cris}(\mathscr{L})$ to $\ch_{\dR}(\mathscr{L})$. Now, the other compatibilities and functoriality can be deduced from what we have shown.
\end{proof}

\begin{rem}
In the above proof, we implicitly use that the isomorphism $t(\mathfrak{X})_\eta\cong X^{\ad}$ is compatible with line bundles in the following sense: for a line bundle $\mathscr{L}$ on $\mathscr{X}$,
let $\mathfrak{L}$ be the formal completion of $\mathscr{L}$ along the special fiber,
and $t(\mathfrak{L})_\eta$ the corresponding line bundle on $t(\mathfrak{X})_\eta$. 
Let $\mathscr{L}_K$ be the restriction of $\mathscr{L}$ to the generic fiber $X$,
and $(\mathscr{L}_K)^{\ad}$ its analytification.
Then there is an isomorphism of line bundles
$t(\mathfrak{L})_\eta \cong (\mathscr{L}_K)^{\ad}$
on $t(\mathfrak{X})_\eta\cong X^{\ad}$.

This can be seen as follows: by the construction of $t(\mathfrak{X})_\eta\cong X^{\ad}$ and \'etale descent for line bundles on rigid analytic varieties, the above claim reduces to an analogous statement in the case where $\mathscr{X}$ is a scheme of finite type, but not necessary proper, over $\O_K$: we only have a morphism $t(\mathfrak{X})_\eta\to X^{\ad}$ in this case, and the statement is that the pullback of $(\mathscr{L}_K)^{\ad}$ is isomorphic to $t(\mathfrak{L})_\eta$. 
\end{rem}

\subsection{Remarks on the work of Blasius-Wintenberger on $p$-adic properties of absolute Hodge cycles}
\label{Subsection:Blasius-Wintenberger}

The theory of integral canonical models of Shimura varieties relies on the results of Blasius-Wintenberger \cite{Blasius},
where $p$-adic properties of absolute Hodge cycles were studied.

Technically speaking, one chooses constructions of comparison maps when using their results.
In this paper, we choose the comparison maps $c_{\dR, X}$ and $c_{\cris, \mathscr{X}}$.
Then the main results of \cite{Blasius} hold as all requirements are checked above.
This fact for $c_{\dR, X}$ is also used implicitly in \cite{CS}.   

Moreover, to fill in details of the proofs of \cite[Proposition 5.3, Proposition 5.6 (4)]{MadapusiTateConj}, one should generalize ``Principle B'' in Blasius' paper \cite{Blasius} to smooth algebraic spaces.
(Let $\mathcal{X}_{\widetilde{\mathrm{M}}_{2d,\Q}} \to \widetilde{\mathrm{M}}_{2d,\Q}$ be
the universal family considered in \cite{MadapusiTateConj}.
This morphism is representable by an algebraic space, not a scheme.)
For this purpose, one needs to generalize the arguments in the proof of \cite[Theorem 3.1]{Blasius} allowing $X$ in \emph{loc. cit.} to be an algebraic space. This requires us to check the following:
\begin{itemize}
\item Artin's comparison theorem and GAGA for algebraic spaces over $\C$. (See \cite[Proposition 3.4.1 (iii)]{Sun} for the former. GAGA is well-known and follows from the standard argument using Chow's lemma.)
Artin's comparison theorem for algebraic spaces was also used in the construction of $\alpha_{\ell}$ in \cite[Proposition 5.6 (1)]{MadapusiTateConj}.
\item A functorial construction of pure Hodge structure in the cohomology of smooth proper algebraic spaces over $\C$. (The usual construction works; see \cite[Part I, Section 2.5]{PetersSteenbrink}. Also, we have geometric variations in this setting; see \cite[Part IV, Corollary 10.32]{PetersSteenbrink}.)
\item A generalization of Deligne's theorem on the fixed part \cite[Th\'eor\`eme 4.1.1]{Deligne:HodgeII} to algebraic spaces over $\C$. (The argument in \emph{loc. cit.} works using Chow's lemma for algebraic spaces.)
\item Existence of a smooth compactification of $X$. (Use \cite{ConradLieblichOlsson} and \cite{BierstoneMilman}.)
\item The de Rham comparison theorem for smooth proper algebraic spaces over a $p$-adic field. (See Remark \ref{Remark:deRhamComparisonAlgebraicSpace}.)
\end{itemize}

\subsection{Comparison isomorphisms for $p$-divisible groups}\label{Subsection:Comparison isomorphisms for $p$-divisible groups}
In \cite{Faltings99}, Faltings constructed a comparison map
\[
c_{\mathscr{G}} \colon D_{\cris}((T_p\mathscr{G})^{\vee}[1/p]) \cong \mathbb{D}(\mathscr{G}_k)(W)[1/p]
\]
for a $p$-divisible group $\mathscr{G}$ over $\O_K$.
The purpose of this subsection is to show that
Faltings' comparison map $c_{\mathscr{G}}$ coincides with the comparison map used by Kim-Madapusi Pera in \cite[Theorem 2.12 (2.12.3)]{KimMadapusiIntModel}.
This fact is used implicitly in the proof of \cite[Proposition 3.12]{KimMadapusiIntModel} to apply
Kisin's result \cite[Lemma 1.1.17]{KisinModp}.

Let $\mathscr{G}$ be a $p$-divisible group over $\O_K$.
For the base change $\mathscr{G}_{{\O_K}/p}$ of $\mathscr{G}$,
we have a (contravariant) crystal $\mathbb{D}(\mathscr{G}_{{\O_K}/p})$ over $\mathrm{CRIS}(({\O_K}/p)/\Z_p)$.
Its value
\[
\mathbb{D}(\mathscr{G}_{{\O_K}/p})(A_{\cris}):=\mathbb{D}(\mathscr{G}_{{\O_K}/p})_{A_{\cris} \twoheadrightarrow {\O_C}/p}
\]
in $(\Spec {{\O_C}/p} \hookrightarrow \Spec {A_{\cris}})$
is a free $A_{\cris}$-module of finite rank,
and equipped with a Frobenius endomorphism $\varphi$.
We define a filtration
\[
\Fil^1\mathbb{D}(\mathscr{G}_{{\O_K}/p})(A_{\cris}) \hookrightarrow \mathbb{D}(\mathscr{G}_{{\O_K}/p})(A_{\cris})
\]
by the inverse image of the first piece of the Hodge filtration $\Fil^1\mathbb{D}(\mathscr{G}_{{\O_K}/p})(\O_C)$
of $\mathbb{D}(\mathscr{G}_{{\O_K}/p})(\O_C)$
under the surjection
$\mathbb{D}(\mathscr{G}_{{\O_K}/p})(A_{\cris}) \twoheadrightarrow \mathbb{D}(\mathscr{G}_{{\O_K}/p})(\O_C)$.

In \cite{Faltings99},
Faltings constructed a period map
\[
\mathrm{Per}_{\cris, \mathscr{G}}\colon T_p\mathscr{G} \to \Hom_{\Fil, \varphi}(\mathbb{D}(\mathscr{G}_{{\O_K}/p})(A_{\cris}), A_{\cris}).
\]
Here the right-hand side is the $\Z_p$-module of homomorphisms from $\mathbb{D}(\mathscr{G}_{{\O_K}/p})(A_{\cris})$ to $A_{\cris}$ which commute with Frobenius endomorphisms and map $\Fil^1\mathbb{D}(\mathscr{G}_{{\O_K}/p})(A_{\cris})$ into $\Fil^1(A_{\cris}):=\Ker \theta$.
When $p \neq 2$, the period map $\mathrm{Per}_{\cris, \mathscr{G}}$ is an isomorphism by \cite[Theorem 7]{Faltings99} (see also \cite[Section 5.2]{Fargues}).
When $p = 2$, it is an injection and the cokernel is killed by $2$.

By the rigidity of quasi-isogenies, there is a unique quasi-isogeny
\[
f \in \Hom_{\O_K/p}(\mathscr{G}_k\otimes_{k}\O_K/p, \ \mathscr{G}_{{\O_K}/p})\otimes_{\Z_p}\Q_p
\]
lifting the identity of $\mathscr{G}_k$.
The quasi-isogeny $f$ induces an isomorphism
\[
\mathbb{D}(\mathscr{G}_{{\O_K}/p})(\O_K)[1/p] \cong \mathbb{D}(\mathscr{G}_k)(W)\otimes_{W} K,
\]
and the Hodge filtration on the left-hand side makes $\mathbb{D}(\mathscr{G}_k)(W)[1/p]$ a filtered $\varphi$-module.
This isomorphism is equal to the isomorphism
given by Berthelot-Ogus \cite[Proposition 3.14]{BO}; see Remark \ref{Remark:Berthelot-Ogus map for p-divisible group}.

The quasi-isogeny $f$ also induces an isomorphism
\[
\mathbb{D}(\mathscr{G}_{{\O_K}/p})(A_{\cris})\otimes_{A_{\cris}}B^{+}_{\cris} \cong \mathbb{D}(\mathscr{G}_k)(W)\otimes_{W} B^{+}_{\cris},
\]
which in turn induces an isomorphism
\[
\Hom_{\Fil, \varphi}(\mathbb{D}(\mathscr{G}_{{\O_K}/p})(A_{\cris}), A_{\cris})[1/p] \cong \Hom_{\Fil, \varphi}(\mathbb{D}(\mathscr{G}_k)(W)[1/p], B^{+}_{\cris}).
\]
Here the right-hand side is the $\Q_p$-module of homomorphisms from $\mathbb{D}(\mathscr{G}_k)(W)[1/p]$ to $B^{+}_{\cris}$ which commute with Frobenius endomorphisms and preserving the filtrations after base change to $K$;
we equip $B^{+}_{\cris}\otimes_{W[1/p]}K$ with a filtration
\[
\Ker(\theta \otimes_{W}K) \subset B^{+}_{\cris}\otimes_{W[1/p]}K.
\]
The following composite is $\Gal(\overline{K}/K)$-equivariant
\begin{align*}
	T_p\mathscr{G}[1/p] & \cong \Hom_{\Fil, \varphi}(\mathbb{D}(\mathscr{G}_{{\O_K}/p})(A_{\cris}), A_{\cris})[1/p]  \\
	& \cong \Hom_{\Fil, \varphi}(\mathbb{D}(\mathscr{G}_k)(W)[1/p], B^{+}_{\cris}),
\end{align*}
where the first isomorphism is $\mathrm{Per}_{\cris, \mathscr{G}}$ with $p$ inverted.
It induces an isomorphism of filtered $\varphi$-modules
\[
c_{\mathscr{G}} \colon D_{\cris}((T_p\mathscr{G})^{\vee}[1/p]) \cong \mathbb{D}(\mathscr{G}_k)(W)[1/p].
\]

Next, we shall recall the construction of the comparison map given in \cite[Theorem 2.12 (2.12.2)]{KimMadapusiIntModel}, which is based on the theory of Dieudonn\'e displays developed by Zink and Lau.

We put $\mathfrak{S} := W[[u]]$.
For a Breuil-Kisin module $\M$ (over $\O_K$ with respect to $\{ \varpi^{1/{p^n}} \}_{n \geq 0}$) of height $\leq 1$, there is an $\mathfrak{S}$-linear homomorphism $\phi \colon \M \to \varphi^*\M$ such that $\varphi \circ \phi$ is the multiplication by $pE(u)/E(0)$.
Let $\M^t$ be a Breuil-Kisin module such that its underlying $\mathfrak{S}$-module is
$\M^{\vee} := \Hom_{\mathfrak{S}}(\M, \mathfrak{S})$
and its Frobenius
$\varphi^*(\M^{\vee})=(\varphi^*\M)^{\vee} \to \M^{\vee}$
is given by
$f \mapsto f\circ\phi$.

Lau constructed an equivalence $\M'$ from the category of $p$-divisible groups over $\O_K$ to the category of Breuil-Kisin modules (over $\O_K$ with respect to $\{ \varpi^{1/{p^n}} \}_{n \geq 0}$) of height $\leq 1$; see \cite[Corollary 5.4, Theorem 6.6]{LauDisplay}.
Let $\M^{L}$ be the (Cartier) dual of the equivalence of categories
constructed by Lau.
Namely, we put
\[
\M^{L}(\mathscr{G}):=\M'(\mathscr{G})^t.
\]

By \cite[Proposition 4.1, Proposition 8.5]{LauGalois}, there is an isomorphism
\[
T_p\mathscr{G} \cong \Hom_{\varphi}(\M^{L}(\mathscr{G}), \mathfrak{S}^{\mathrm{nr}}).
\]
For the ring $\mathfrak{S}^{\mathrm{nr}}$, see \cite[Section 7]{LauGalois} for example.
It is equipped with a Frobenius endomorphism such that there are inclusions
\[
\mathfrak{S} \hookrightarrow \mathfrak{S}^{\mathrm{nr}} \hookrightarrow A_{\mathrm{inf}}
\]
commuting with the Frobenius endomorphisms.

Let $\Fil^1(S_{\varpi})$ be the kernel of the surjection $S_\varpi \to \O_K$.
(Recall that $S_{\varpi}$ is the $p$-adic completion of the PD envelope of the surjection $W[u] \to \O_K$ defined by $u \mapsto \varpi$.)
Let $\mathbb{D}(\mathscr{G}_{{\O_K}/p})(S_{\varpi})$ be the value in $(\Spec \O_K/p  \hookrightarrow \Spec S_{\varpi})$.
It is equipped with a Frobenius endomorphism $\varphi$ and a filtration
\[
\Fil^1\mathbb{D}(\mathscr{G}_{{\O_K}/p})(S_{\varpi}) \hookrightarrow \mathbb{D}(\mathscr{G}_{{\O_K}/p})(S_{\varpi})
\]
defined by the inverse image of the first piece of the Hodge filtration $\Fil^1\mathbb{D}(\mathscr{G}_{{\O_K}/p})(\O_K)$
of $\mathbb{D}(\mathscr{G}_{{\O_K}/p})(\O_K)$
under the surjection
\[ \mathbb{D}(\mathscr{G}_{{\O_K}/p})(S_\varpi) \twoheadrightarrow \mathbb{D}(\mathscr{G}_{{\O_K}/p})(\O_K). \]

Let $c' \in S_\varpi$ be a unique unit which maps to $1 \in W$ with
\[ c'\varphi(c'^{-1})=\varphi(E(u)/E(0)) \in S_\varpi. \]
(For the element $\lambda \in S_\varpi[1/p]$ in \cite[Section 1.1.1]{KisinCris}, we have $c'=\varphi(\lambda)$.)
By \cite[Proposition 7.1]{LauDisplay}, there is a canonical isomorphism
\[
\varphi^*\M'(\mathscr{G})\otimes_{\mathfrak{S}} S_{\varpi} \cong \mathbb{D}(\mathscr{G}_{{\O_K}/p})(S_{\varpi})^{\vee}.
\]
The dual of this isomorphism multiplied by
$c'$
is an isomorphism
\[
\mathbb{D}(\mathscr{G}_{{\O_K}/p})(S_{\varpi}) \cong \varphi^*\M^{L}(\mathscr{G})\otimes_{\mathfrak{S}} S_{\varpi}
\]
which is compatible with Frobenius endomorphisms and maps $\Fil^1\mathbb{D}(\mathscr{G}_{{\O_K}/p})(S_{\varpi})$ onto
\[
\{\, x \in \varphi^*\M^{L}(\mathscr{G})\otimes_{\mathfrak{S}} S_{\varpi} \mid (1 \otimes \varphi)(x) \in \Fil^1(S_{\varpi})(\M^{L}(\mathscr{G})\otimes_{\mathfrak{S}} S_{\varpi}) \,\}.
\]

Let $\M((T_p\mathscr{G})^{\vee})$ be the Breuil-Kisin module associated with the dual $(T_p\mathscr{G})^{\vee}$
of the $p$-adic Tate module of $\mathscr{G}$ defined by Kisin; see Section \ref{Subsection:Breuil-Kisin modules and crystalline Galois representations}.
By construction, there is a canonical isomorphism
\[
T_p\mathscr{G} \cong \Hom_{\varphi}(\M((T_p\mathscr{G})^{\vee}), \mathfrak{S}^{\mathrm{nr}}).
\]
By \cite[Theorem 2.12 (2.12.2)]{KimMadapusiIntModel} and its proof,
there is an isomorphism of Breuil-Kisin modules
\[
\M^{L}(\mathscr{G}) \cong \M((T_p\mathscr{G})^{\vee})
\]
such that, under this isomorphism, the isomorphism
\[
T_p\mathscr{G} \cong \Hom_{\varphi}(\M^L(\mathscr{G}), \mathfrak{S}^{\mathrm{nr}})
\]
is compatible with the above isomorphism.

The comparison isomorphism used by Kim-Madapusi Pera in \cite[Theorem 2.12 (2.12.3)]{KimMadapusiIntModel} is defined as the composite of the following isomorphisms:
\begin{align*}
D_{\cris}((T_p\mathscr{G})^{\vee}[1/p]) &\cong \Mcris((T_p\mathscr{G})^{\vee})[1/p] \\
&\cong \varphi^*\M^{L}(\mathscr{G})\otimes_{\mathfrak{S}}W[1/p]\\ 
&\cong \mathbb{D}(\mathscr{G}_{{\O_K}/p})(S_{\varpi})\otimes_{S_{\varpi}}S_{\varpi}/uS_{\varpi}[1/p]\\ &\cong \mathbb{D}(\mathscr{G}_k)(W)[1/p],
\end{align*}
where the last isomorphism is provided by the crystalline property of $\mathbb{D}(\mathscr{G}_{{\O_K}/p})$. 
It is an isomorphism of filtered $\varphi$-modules, and
we denote it by $c^{L}_{\mathscr{G}}$.

\begin{prop}\label{Proposition:Compare Kim-MP with Faltings}
$c^{L}_{\mathscr{G}}$ coincides with $c_{\mathscr{G}}$.
\end{prop}
\begin{proof}
We put
	$\M^{L}_{\cris}(\mathscr{G}):=\varphi^*\M^{L}(\mathscr{G})\otimes_{\mathfrak{S}}W$.
There is an $S_{\varpi}$-linear homomorphism
\[
f_{\M^{L}(\mathscr{G})} \colon \M^{L}_{\cris}(\mathscr{G})\otimes_{W}S_{\varpi}[1/p] \cong \varphi^{*}\M^{L}(\mathscr{G})\otimes_{\mathfrak{S}} S_{\varpi}[1/p]
\]
which commutes with Frobenius endomorphisms and lifts the identity of $\M^{L}_{\cris}(\mathscr{G})$.
In fact, it is characterized by such properties.
Translating the construction of the isomorphism
\[ D_{\cris}((T_p\mathscr{G})^{\vee}[1/p]) \cong \Mcris((T_p\mathscr{G})^{\vee})[1/p] \]
in \cite[Proposition 2.1.5]{KisinCris} in terms of $\Gal(\overline{K}/K)$-representations, we see that the isomorphism $c^{L}_{\mathscr{G}}$ corresponds to the composite of the following isomorphisms:
	\begin{align*}
	    T_p\mathscr{G}[1/p] &\cong \Hom_{\varphi}(\M^L(\mathscr{G}), \mathfrak{S}^{\mathrm{nr}})[1/p]\\
		&  \cong \Hom_{\Fil, \varphi}(\varphi^{*}\M^L(\mathscr{G}), A_{\cris})[1/p]\\
		& \cong \Hom_{\Fil, \varphi}(\M^{L}_{\cris}(\mathscr{G})[1/p], B^{+}_{\cris}) \\
		& \cong \Hom_{\Fil, \varphi}(\mathbb{D}(\mathscr{G}_k)(W)[1/p], B^{+}_{\cris}).
	\end{align*}
Here, the second isomorphism is induced by the Frobenius of $\M^L(\mathscr{G})$ and the injection $\mathfrak{S}^{\mathrm{nr}} \hookrightarrow A_{\cris}$, and the third isomorphism is induced by $f_{\M^L(\mathscr{G})}$.
Using the uniqueness of the isomorphism $f_{\M^{L}(\mathscr{G})}$, we may replace the composite of the last two isomorphisms by the following composite:
\begin{align*}
    \Hom_{\Fil, \varphi}(\varphi^{*}\M^L(\mathscr{G}), A_{\cris})[1/p]
		& \cong \Hom_{\Fil, \varphi}(\mathbb{D}(\mathscr{G}_{\O_K/p})(S_\varpi), A_{\cris})[1/p] \\
		& \cong \Hom_{\Fil, \varphi}(\mathbb{D}(\mathscr{G}_{\O_K/p})(A_{\cris}), A_{\cris})[1/p] \\
		& \cong \Hom_{\Fil, \varphi}(\mathbb{D}(\mathscr{G}_k)(W)[1/p], B^{+}_{\cris}).
\end{align*}
Here, the second isomorphism is induced by the base change
\[
\mathbb{D}(\mathscr{G}_{\O_K/p})(S_{\varpi})\otimes_{S_{\varpi}}A_{\cris} \cong \mathbb{D}(\mathscr{G}_{\O_K/p})(A_{\cris}),
\]
and the third isomorphism is induced by the quasi-isogeny $f$.

Therefore, the assertion follows from the following Proposition \ref{Proposition:Comparison Lau with Faltings}, which was essentially proved by Lau in \cite{LauGalois}.
\end{proof}

\begin{prop}[Lau]\label{Proposition:Comparison Lau with Faltings}
Let $\mathrm{Per}'$ be the following composite:
\begin{align*}
	    T_p\mathscr{G} &\cong \Hom_{\varphi}(\M^L(\mathscr{G}), \mathfrak{S}^{\mathrm{nr}})\\
		&  \to \Hom_{\Fil, \varphi}(\varphi^{*}\M^L(\mathscr{G}), A_{\cris})\\
		& \cong \Hom_{\Fil, \varphi}(\mathbb{D}(\mathscr{G}_{\O_K/p})(S_\varpi), A_{\cris}) \\
		& \cong \Hom_{\Fil, \varphi}(\mathbb{D}(\mathscr{G}_{\O_K/p})(A_{\cris}), A_{\cris}).
\end{align*}
Then, Faltings' period map $\mathrm{Per}_{\cris, \mathscr{G}}$ coincides with $\mathrm{Per}'$.
\end{prop}
\begin{proof}
We shall explain how the equality
$\mathrm{Per}_{\cris, \mathscr{G}} = \mathrm{Per}'$
follows from Lau's result \cite[Proposition 6.2]{LauGalois}.
We freely use the notion of a frame and a window from
Lau's papers \cite{LauDisplay, LauGalois}.

We have the following commutative diagram of homomorphisms of rings:
\[
\xymatrix{ \mathfrak{S} \ar^-{}[r]&\mathfrak{S}^{\mathrm{nr}} \ar^-{\varkappa^{\mathrm{nr}}}[r] \ar_-{}[d]& \hat{\mathbb{W}}(\widetilde{\O}_K) \ar^-{\iota}[d] & \\
&A_{\cris} \ar^-{\varkappa_{\cris}}[r]& \hat{\mathbb{W}}^{+}(\widetilde{\O}_K).
}
\]
The above diagram induces the following commutative diagram of homomorphisms of frames:
\[
\xymatrix{ \mathscr{B} \ar^-{}[r] &\mathscr{B}^{\mathrm{nr}} \ar^-{\varkappa^{\mathrm{nr}}}[r] \ar_-{}[d]& \hat{\mathscr{D}}_{\widetilde{\O}_K} \ar^-{\iota}[d] & \\
& \mathcal{A}_{\cris} \ar^-{\varkappa_{\cris}}[r]& \hat{\mathscr{D}}^{+}_{\widetilde{\O}_K}.
}
\]
(For the above two commutative diagrams, see \cite{LauGalois}.)
The homomorphism $\varkappa^{\mathrm{nr}} \colon \mathscr{B}^{\mathrm{nr}} \to \hat{\mathscr{D}}_{\widetilde{\O}_K}$ (resp.\ $\iota \colon \hat{\mathscr{D}}_{\widetilde{\O}_K} \to \hat{\mathscr{D}}^+_{\widetilde{\O}_K}$) is a $u$-homomorphism (resp.\ $u_0$-homomorphism) of frames for a unit $u \in \hat{\mathbb{W}}(\widetilde{\O}_K)$ (resp.\ $u_0 \in \hat{\mathbb{W}}^+(\widetilde{\O}_K)$).
There is a unique unit $c \in \hat{\mathbb{W}}(\widetilde{\O}_K)$ (resp.\ $c_0 \in \hat{\mathbb{W}}^+(\widetilde{\O}_K)$) which maps to $1 \in W$ with $c\varphi(c^{-1})=u$ (resp.\ $c_0\varphi(c^{-1}_0)=u_0$).

The Breuil-Kisin module $\M'(\mathscr{G})$ corresponds to a window over $\mathscr{B}$, and let $\M'(\mathscr{G})^{\mathrm{nr}}$ be the window over $\mathscr{B}^{\mathrm{nr}}$ obtained by base change.
There is a window over $\mathcal{A}_{\cris}$ associated with $\mathbb{D}(\mathscr{G}_{\O_K/p})(A_{\cris})$, which will be denoted by the same notation; see \cite[Section 6]{LauGalois}.
The base change of $\M'(\mathscr{G})^{\mathrm{nr}}$ to $\mathcal{A}_{\cris}$ is identified with the dual $\mathbb{D}(\mathscr{G}_{\O_K/p})(A_{\cris})^t$ by \cite[Proposition 7.1]{LauDisplay}.
(For the dual of a window, see \cite[Section 2A]{LauDisplay}.)

For a window $\mathscr{P}$, let $T(\mathscr{P})$ be the module of invariants; see \cite[Section 3]{LauGalois}.
There are natural isomorphisms
\begin{align*}
    T(\M'(\mathscr{G})^{\mathrm{nr}}) &\cong \Hom_{\varphi}(\M^{L}(\mathscr{G}), \mathfrak{S}^{\mathrm{nr}}), \\
 T(\mathbb{D}(\mathscr{G}_{\O_K/p})(A_{\cris})^t) &\cong \Hom_{\Fil, \varphi}(\mathbb{D}(\mathscr{G}_{\O_K/p})(A_{\cris}), A_{\cris}).
\end{align*}
Under the above isomorphisms, the isomorphism $\mathrm{Per}'$ is identified with
the following composite:
\[
\xymatrix{
T_p\mathscr{G} \ar[r]^-{\cong}&  T(\M'(\mathscr{G})^{\mathrm{nr}}) \ar[r]^-{x \mapsto c' \otimes x}&
T(\mathbb{D}(\mathscr{G}_{\O_K/p})(A_{\cris})^t).}
\]
Here we denote the image of $c' \in S_\varpi$ in $A_{\cris}$ by the same letter.

In \cite[Proposition 4.1]{LauGalois}, Lau constructed a period isomorphism
\[
\mathrm{Per}_{\mathscr{G}} \colon T_p\mathscr{G} \cong T(\varkappa^{\mathrm{nr}}_*\M'(\mathscr{G})^{\mathrm{nr}}).
\]
Here we normalize this period map as in \cite[Remark 4.2]{LauGalois}.
By \cite[Proposition 8.5]{LauGalois}, the homomorphism
\[
T(\M'(\mathscr{G})^{\mathrm{nr}}) \to T(\varkappa^{\mathrm{nr}}_*\M'(\mathscr{G})^{\mathrm{nr}})
\]
defined by $x \mapsto c \otimes x$ is an isomorphism.
The composite of $\mathrm{Per}_{\mathscr{G}}$ with the inverse of the above isomorphism is identified with the isomorphism
$T_p\mathscr{G} \cong T(\M'(\mathscr{G})^{\mathrm{nr}})$
by the definition of $T_p\mathscr{G} \cong \Hom_{\varphi}(\M^L(\mathscr{G}), \mathfrak{S}^{\mathrm{nr}})$.

By \cite[Proposition 6.2]{LauGalois}, the following diagram commutes:
\[
\xymatrix{ T_p\mathscr{G} \ar^-{\mathrm{Per}_{\cris, \mathscr{G}}}[r] \ar_-{\cong}[d] \ar@/^3pc/[dd]^-{\mathrm{Per}_{\mathscr{G}}}& T(\mathbb{D}(\mathscr{G}_{\O_K/p})(A_{\cris})^t)  \ar^-{x \mapsto 1 \otimes x}[dd] \\ T(\M'(\mathscr{G})^{\mathrm{nr}}) \ar_-{x \mapsto c \otimes x}[d] &   \\
T(\varkappa^{\mathrm{nr}}_*\M'(\mathscr{G})^{\mathrm{nr}}) \ar^-{x \mapsto c_0 \otimes x}[r] &  T(\iota_*\varkappa^{\mathrm{nr}}_*\M'(\mathscr{G})^{\mathrm{nr}}).
}
\]
We denote by $\tau$ the right vertical homomorphism, which is an isomorphism; see \cite[Proposition 6.2]{LauGalois}.
Using $\iota(c)c_0=\varkappa_{\cris}(c')$,
we see that the composite of $\mathrm{Per}_{\mathscr{G}}$ and the bottom horizontal arrow is equal to the composite of $\mathrm{Per}'$ and $\tau$.
Therefore, we have
\[
 \tau \circ \mathrm{Per}_{\cris, \mathscr{G}} = \tau \circ \mathrm{Per}'.
\]
Since $\tau$ is an isomorphism, the equality
$\mathrm{Per}_{\cris, \mathscr{G}} = \mathrm{Per}'$
is proved.

The proof of Proposition \ref{Proposition:Comparison Lau with Faltings} is complete.
\end{proof}

\begin{rem}\label{Remark:Berthelot-Ogus map for p-divisible group}
The isomorphism
\[
\mathbb{D}(\mathscr{G}_{{\O_K}/p})(\O_K)[1/p] \cong \mathbb{D}(\mathscr{G}_k)(W)\otimes_{W} K
\]
induced by the quasi-isogeny $f$ is equal to the isomorphism
given by Berthelot-Ogus \cite[Proposition 3.14]{BO}.
This follows from the fact that
there is a unique $\varphi$-equivariant isomorphism
\[
\mathbb{D}(\mathscr{G}_{{\O_K}/p})(S_{\varpi})\otimes_{S_{\varpi}}S_{\varpi}[1/p] \cong \mathbb{D}(\mathscr{G}_k)(W)\otimes_{W} S_{\varpi}[1/p]
\]
which lifts the Berthelot-Ogus isomorphism.
This isomorphism can be constructed in a similar way to \cite[Proposition 3.14]{BO}.
\end{rem}

\begin{rem}\label{Remark:Faltings integral refinement}
Given the construction of $c^L_{\mathscr{G}}$ and Proposition \ref{Proposition:Compare Kim-MP with Faltings}, we can restate \cite[Theorem 2.12 (2.12.3), (2.12.4)]{KimMadapusiIntModel} using Faltings' comparison map $c_{\mathscr{G}}$ as follows.
\begin{enumerate}
    \item The composite
\[\Mcris((T_{p}\mathscr{G})^{\vee})[1/p] \cong D_{\cris}((T_{p}\mathscr{G})^{\vee}[1/p])  \overset{c_{\mathscr{G}}}{\cong} \mathbb{D}(\mathscr{G}_k)(W)[1/p]
	\]
	maps $\Mcris((T_{p}\mathscr{G})^{\vee})$ onto $\mathbb{D}(\mathscr{G}_k)(W)$.
    \item The composite
	\[
	\MdR((T_{p}\mathscr{G})^{\vee})[1/p] \cong D_{\dR}((T_{p}\mathscr{G})^{\vee}[1/p]) 
	\overset{c_{\mathscr{G}}}{\cong} \mathbb{D}(\mathscr{G}_{{\O_K/p}})(\O_K)\otimes_{\O_K}K
	\]
	maps $\MdR((T_{p}\mathscr{G})^{\vee})$ onto $\mathbb{D}(\mathscr{G}_{{\O_K/p}})(\O_K)$, and maps $\Fil^1(\MdR((T_{p}\mathscr{G})^{\vee}))$ onto $\Fil^1\mathbb{D}(\mathscr{G}_{{\O_K/p}})(\O_K)$.
\end{enumerate}
\end{rem}

Assume $\mathscr{G}$ is the  $p$-divisible group $\mathcal{B}[p^{\infty}]$ associated with an abelian scheme $\mathcal{B}$ over $\O_K$.
By \cite[Th\'eor\`eme 2.5.6, Proposition 3.3.7]{BBM}, there is a natural isomorphism
\[
\mathbb{D}(\mathscr{G}_{k})(W)\cong H^1_{\cris}(\mathcal{B}_{k}/W).
\]

\begin{prop}
\label{Proposition:Comparison abelian variety and p-divisible group}
Let $\mathcal{B}$ be an abelian scheme over $\O_K$,
and $\mathscr{G} := \mathcal{B}[p^{\infty}]$ the $p$-divisible group associated with $\mathcal{B}$.
Let $T_p \mathscr{G}$ be the $p$-adic Tate module of $\mathscr{G}$.
Under the isomorphisms
	\begin{align*}		(T_p\mathscr{G})^{\vee} &\cong H^1_{\mathrm{\acute{e}t}}(\mathcal{B}_{\overline{K}}, \Z_p),\\
		\mathbb{D}(\mathscr{G}_{k})(W)&\cong H^1_{\cris}(\mathcal{B}_{k}/W),
	\end{align*}
	the isomorphism $c_{\mathscr{G}}$ is compatible with the crystalline comparison map $c_{\cris, \mathcal{B}}$.
\end{prop}
\begin{proof}
By \cite[Th\'eor\`eme 2.5.6, Proposition 3.3.7]{BBM}, there is a natural isomorphism
\[
\mathbb{D}(\mathscr{G}_{\O_K/p})(A_{\cris})\cong H^1_{\cris}(\mathcal{B}_{\O_K/p}/A_{\cris}).
\]
After inverting $p$, this isomorphism is identified with the base change of the map
$
\mathbb{D}(\mathscr{G}_{k})(W) \cong H^1_{\cris}(\mathcal{B}_{k}/W)
$
along $W \to A_{\cris}[1/p]$ 
under the isomorphism
\[
\mathbb{D}(\mathscr{G}_{{\O_K}/p})(A_{\cris})[1/p] \cong \mathbb{D}(\mathscr{G}_k)(W)\otimes_{W} A_{\cris}[1/p]
\]
induced by the quasi-isogeny $f$ and the isomorphism
\begin{equation*}
H^1_{\cris}(\mathcal{B}_{\O_K /p}/ A_{\cris})[1/p] \cong H^1_{\cris}(\mathcal{B}_k /W) \otimes_{W} A_{\cris}[1/p] 
\end{equation*}
in Section \ref{Subsection:Crystalline comparison map of Bhatt-Morrow-Scholze}.
This follows from the characterization of the $W$-linear map $s_{\cris}$ in the proof of Proposition \ref{Proposition:Comparison with BO}.
Now, the assertion follows from \cite[Proposition 14.8.3]{Scholze}.
(Alternatively, one can use the Hodge-Tate version \cite[Proposition 4.15]{ScholzeSurvey} by checking a certain compatibility, but we omit the details.)
\end{proof}

\subsection*{Acknowledgements}

The authors would like to thank Keerthi Madapusi Pera for e-mail correspondences on the proof of the \'etaleness of the Kuga-Satake morphism.
The work of the first author was supported by JSPS Research Fellowships for Young Scientists KAKENHI Grant Number 18J22191.
The work of the second author 
was supported by JSPS KAKENHI Grant Number 20674001 and 26800013.

\end{document}